\documentclass[aap,preprint]{imsart}

\usepackage{multirow,mathtools,latexsym, enumerate, enumitem, amsmath, amsthm, amssymb,hyperref, pdftricks,tikz,framed,color,enumitem,enumerate,mleftright,subcaption}
\usepackage{algorithmic,commath}

\startlocaldefs

\def\Y_#1{y_{\!#1}}

\def\cay{\mathrm{cay}}
\def\E{\mathrm{E}}

\def\cov{\operatorname{cov}}
\def\Var{\operatorname{Var}}
\def\tr{\operatorname{trace}}

\def\Exp{\operatorname{Exp}}

\endlocaldefs

\newtheorem{theorem}{Theorem}[section]

\newtheorem{lemma}[theorem]{Lemma}

\newtheorem{remark}[theorem]{Remark}
\newtheorem{prop}[theorem]{Proposition}

\theoremstyle{definition}

\newtheorem{algorithm}{Algorithm}[section]

\theoremstyle{remark}

\definecolor{darkred}{rgb}{0.9,0.1,0.1}

\begin{document}

\title{Cayley Splitting for Second-Order Langevin Stochastic Partial Differential Equations}
\author{Nawaf Bou-Rabee}\thanks{ {\tt nawaf.bourabee@rutgers.edu}}

\begin{abstract}
We give accurate and ergodic numerical methods for semilinear, second-order Langevin stochastic partial differential equations (SPDE).  As a byproduct, we also give good geometric numerical methods for their infinite-dimensional Hamiltonian counterpart.  These methods are suitable for Hamiltonian Monte Carlo on Hilbert spaces without preconditioning the underlying Hamiltonian dynamics.  A key tool in our approach is Krein's theory on strong stability of symplectic maps, which gives us sufficient conditions for stability of symplectic splitting schemes in highly oscillatory Hamiltonian problems.   

\end{abstract}




\maketitle


\section{Introduction}


Our main interest lies in the numerical solution of 
\begin{equation} \label{eq:langevin_spde_intro}
\begin{aligned}
\partial_t u(t,s) &= p(t,s)  \;, \\
\partial_t p(t,s) &= \partial_{s}^2 u(t,s) - \nabla V( u(t,s)) - \gamma p(t,s)  + \sqrt{2 \gamma \beta^{-1}} \partial_t W(t,s) \;, 
\end{aligned} 
\end{equation}
for $(t,s) \in [0, \infty) \times [0, S]$, where
\begin{itemize}
\item $(t,s)$ are independent variables and $S>0$ is the spatial domain size;
\item $(u(t,s), p(t,s)) \in \mathbb{R}^{2 d}$ are (unknown) spacetime random processes;
\item $V: \mathbb{R}^d \to \mathbb{R}$ is a potential energy function;
\item $W$ is a $d$-dimensional spacetime, cylindrical Wiener process and $ \partial_t W$ is a spacetime white noise; 
\item $\gamma \ge 0$ is a friction parameter, and $\beta \ge 0$ is an inverse temperature parameter.
\end{itemize}
Here, and in the sequel, $\partial_t u$ denotes the partial derivative $\partial u / \partial t$ (the notation $\partial_t p$ and $\partial_{s}^2 u$ is defined similarly) and 
 $\nabla V$ denotes the standard gradient of the $d$-dimensional function $V$. These equations are semilinear, second-order Langevin stochastic partial differential equations (SPDE).  Along with \eqref{eq:langevin_spde_intro}, one needs to specify initial conditions at the initial time $t=0$, and boundary conditions at the endpoints of $[0, S]$, which can be Dirichlet, Neumann, periodic, or mixed.  With these conditions, the solutions to these equations are well-defined (in a mild sense) and are furthermore ergodic \cite{DaZa1996,DaZa2014}.      In analogy with mechanics, we refer to the components of their solution $(u(t,s), p(t,s))$ as position and momentum, respectively.   The Langevin SPDE \eqref{eq:langevin_spde_intro} arises in an analysis of the long-time dynamics of the nonlinear wave equation  \cite{NeVa2016}.     
 
 
 The simulation of Langevin SPDEs presents several challenges.  One is due to the presence of fast frequencies in the linear part of the dynamics when $\gamma=0$.  This issue becomes clear once we transform \eqref{eq:langevin_spde_intro} from the original coordinates to spectral coordinates by using the eigenfunctions of $\partial_{s}^2 u$ endowed with suitable boundary conditions.  In these coordinates, the dynamics of the higher modes is highly oscillatory, and designing stable explicit integrators for this type of dynamics is a hard problem in numerical analysis.  This difficulty would be a bit moot if these higher modes were not being excited, but that is exactly the effect of the spacetime white noise entering \eqref{eq:langevin_spde_intro}.  A second issue is related to the long time stability of schemes, and their capability to capture the invariant distribution of the SPDE.  For this reason, the numerical solution must not only be finite-time accurate, but also approximately preserve the invariant measure of the SPDE.  However, to our knowledge, the only numerical methods that meet these two requirements are limited to {\em first-order} Langevin SPDEs \cite{BeRoStVo2008}. The idea behind these schemes comes from finite-dimensional MCMC and numerical SDE theory, and basically involves combining a $\theta$-integrator for a semidiscrete approximation of the SPDE with a Metropolis-Hastings accept-reject step which sets the invariant distribution of the scheme \cite{RoTw1996A, RoTw1996B, BoVa2010,BoVa2012,BoHa2013,BoDoVa2014, Bo2014, Fa2014, FaHoSt2015}.  This numerical method is known as the Metropolis Adjusted Langevin Algorithm (MALA).  However, unless one chooses $\theta=1/2$ (corresponding to a Crank-Nicholson time discretization), the acceptance rate of MALA may deteriorate with decreasing spatial step size \cite{BeRoStVo2008}.    There are also methods available for so-called {\em preconditioned} first-order Langevin SPDE problems, where the random fluctuations entering \eqref{eq:langevin_spde_intro} are modeled as colored noise \cite{BeRoStVo2008}.  Although preconditioning the dynamics does bypass the numerical stability issue associated to highly oscillatory dynamics, it completely alters the dynamics of the problem, and therefore, preconditioning the dynamics is not a tool we can use to construct good integrators for \eqref{eq:langevin_spde_intro}.


Also of interest to us is the Hamiltonian counterpart of these equations \begin{equation} \label{eq:hamiltonian_pde_intro}
\begin{aligned}
\partial_t u(t,s) = p(t,s)  \;, \quad \partial_t p(t,s) = \partial_{s}^2 u(t,s)  - \nabla V( u(t,s))  \;,
\end{aligned}
\end{equation}
for all  $(t,s) \in [0, \infty) \times [0, S]$.   In finite dimensions, it is well known that Hamiltonian systems are often not ergodic, because their solutions are confined to level sets of the Hamiltonian function.  Accuracy and stability concepts for numerical approximations of finite-dimensional Hamiltonian systems strongly rely on this property of their solutions \cite{SaCa1994, LeRe2004, HaLuWa2010}.   The situation is a bit different in infinite dimensions.  Indeed, for the initial conditions of interest the energy is infinite, and therefore, new stability and accuracy analyses are needed \cite{BePiSaSt2011}.  The numerical solution of \eqref{eq:hamiltonian_pde_intro} is relevant for constructing good integrators for Hamiltonian Monte Carlo (HMC) methods \cite{BePiSaSt2011}, and as we describe next, it is also a key ingredient to our numerical method for \eqref{eq:langevin_spde_intro}.


\section{Main Results} \label{sec:main_results}

For simplicity's sake, we use a finite difference method to discretize the second derivatives appearing in \eqref{eq:langevin_spde_intro} yielding
\begin{equation} \label{eq:semidiscrete_intro}
\begin{aligned}
d \boldsymbol{u}(t) &= \boldsymbol{p}(t)  dt \;, \\
d \boldsymbol{p}(t) &= \boldsymbol{L} \boldsymbol{u}(t)  dt + \boldsymbol{F}( \boldsymbol{u}(t)) dt - \gamma \boldsymbol{p}(t) dt + \sqrt{\frac{2 \gamma \beta^{-1}}{\Delta s}} d \boldsymbol{W}(t) \;,  
\end{aligned}
\end{equation}
where $\Delta s$ is a spatial step size parameter, $(\boldsymbol{u}(t),\boldsymbol{p}(t))$ is the semi-discrete (continuous in time $t$ and discrete in space $s$) solution, $\boldsymbol{L}$ is a (symmetric) discretization matrix, $\boldsymbol{F}$ is a vectorized form of $-\nabla V$, and $\boldsymbol{W}$ is a finite-dimensional Wiener process.   Note that we use bold symbols to indicate finite-dimensional vectors and matrices.   This finite difference discretization can accommodate Dirichlet, periodic, Neumann, or mixed boundary conditions.  The main requirement is that these boundary conditions are incorporated in the spatial discretization in such a way that $\boldsymbol{L}$ is a symmetric matrix.    This requirement ensures that the equations obtained by setting $\gamma=0$ in \eqref{eq:semidiscrete_intro}, i.e., \begin{equation} \label{eq:semidiscrete_hamiltonian_intro}
\begin{aligned}
 \boldsymbol{\dot u}(t) &= \boldsymbol{p}(t)  \;, \\
\boldsymbol{\dot p}(t) &= \boldsymbol{L} \boldsymbol{u}(t)  + \boldsymbol{F}( \boldsymbol{u}(t))   \;,  
\end{aligned}
\end{equation}
are Hamiltonian.  These equations are a semi-discrete analog of \eqref{eq:hamiltonian_pde_intro}.

We discretize \eqref{eq:semidiscrete_intro} in time by using a Strang splitting method \cite{strang1968construction, mclachlan2002splitting, BoSaActaN2018}.  In the case of second-order Langevin SDEs, a natural splitting is given by splitting \eqref{eq:semidiscrete_intro} into a deterministic Hamiltonian part and an Ornstein-Uhlenbeck part in momentum \cite{BuPa2007, BoOw2010, BoVa2010, BePiSaSt2011, Bo2014, alamo2016technique},
\begin{equation} \label{eq:semidiscrete_intro_H}
\tag{H}
\begin{bmatrix}   \dot{\boldsymbol{u}}(t) \\ \dot{\boldsymbol{p}}(t)  \end{bmatrix} =  
\boldsymbol{A}  \begin{bmatrix} \boldsymbol{u}(t) \\ \boldsymbol{p}(t) \end{bmatrix}  +  \begin{bmatrix} \boldsymbol{0} \\ \boldsymbol{F}( \boldsymbol{u}(t)) \end{bmatrix} 
\;,  \quad \text{where} \quad \boldsymbol{A} = \begin{bmatrix} \boldsymbol{0}  & \boldsymbol{I} \\ \boldsymbol{L} & \boldsymbol{0} \end{bmatrix} \;,
\end{equation}
\begin{equation} \label{eq:semidiscrete_intro_O}
\tag{O}
\begin{bmatrix}   \dot{\boldsymbol{u}}(t) \\ d \boldsymbol{p}(t) \vphantom{ \sqrt{\dfrac{\gamma \beta^{-1}}{\Delta s}}}  \end{bmatrix} =  \begin{bmatrix} \boldsymbol{0} \\  - \gamma \boldsymbol{p}(t) dt + \sqrt{\dfrac{2 \gamma \beta^{-1}}{\Delta s}} d \boldsymbol{W}(t) \end{bmatrix} \;.
\end{equation}
Though other splittings of Langevin SDEs are available \cite{RiCi2003,LeMa2013,LeMaSt2015,RiCi2003}, this particular splitting has nice properties including: (i) invariant distribution accuracy \cite[Theorem 3.7]{BoOw2010}; and (ii) it is straightforward to Metropolize \cite[\S 5.2]{BoVa2010}.   The latter property is important for long time simulation since it allows one to set the invariant distribution of the integrator.  The exact flow of \eqref{eq:semidiscrete_intro_O}  over a time interval of length $t$ is given in law by \[
\varphi^{(O)}_{t}( \boldsymbol{u}, \boldsymbol{p} )  \overset{d}{=} \left(\boldsymbol{u}, e^{- \gamma t} \boldsymbol{p} + \sqrt{\dfrac{\beta^{-1}}{\Delta s}} \sqrt{1 - e^{-2 \gamma t}} \boldsymbol{\xi}  \right)
\] where $\boldsymbol{\xi}$ is a standard normal vector, i.e., its components are i.i.d.~standard normal random variables \cite[Chapter 5]{Ev2013}.

Unfortunately, the exact flow of the Hamiltonian system \eqref{eq:semidiscrete_intro_H} can rarely be solved analytically, and a numerical method is needed to approximate its solution.  However, a difficulty with numerically solving these Hamiltonian equations by an explicit symplectic integrator is that the spectral radius of $\boldsymbol{L}$ typically grows like $\Delta s^{-\kappa}$ for some $\kappa > 1/2$.  In other words, as $\Delta s$ decreases, the Hamiltonian dynamics becomes highly oscillatory.   For example, numerical stability of a Verlet integrator applied to \eqref{eq:semidiscrete_hamiltonian_intro} with $\boldsymbol{F}=0$ requires that its time step size $\Delta t$ satisfy $(\Delta t) \Delta s^{-\kappa/2} \le 2$.  To avoid this restrictive stability requirement, and preserve some of the geometric properties of the Verlet integrator, we proceed as follows.

Strongly inspired by the geometric numerical integrators developed in Ref.~\cite{BePiSaSt2011}, we approximate the flow of \eqref{eq:semidiscrete_intro_H} by splitting it into \begin{equation} \label{eq:semidiscrete_intro_A}
\begin{aligned}
\begin{bmatrix} \dot{\boldsymbol{u}}(t) \\  \dot{\boldsymbol{p}}(t) \end{bmatrix} = \boldsymbol{A} \begin{bmatrix} \boldsymbol{u}(t) \\ \boldsymbol{p}(t) \end{bmatrix}  \;, 
\end{aligned}
\tag{A}
\end{equation}
\begin{equation} \label{eq:semidiscrete_intro_B}
\tag{B}
\begin{bmatrix} \dot{\boldsymbol{u}}(t) \\  \dot{\boldsymbol{p}}(t) \end{bmatrix} = \begin{bmatrix} \boldsymbol{0} \\ \boldsymbol{F}( \boldsymbol{u}(t)) \end{bmatrix} \;,
\end{equation}
where \eqref{eq:semidiscrete_intro_A} and \eqref{eq:semidiscrete_intro_B} are deterministic Hamiltonian equations, whose exact flows over a time interval of length $t$ are given by \[
\varphi^{(A)}_{t}( \boldsymbol{u}, \boldsymbol{p} )  = ( \boldsymbol{u}', \boldsymbol{p}' )  \;, \quad \begin{bmatrix} \boldsymbol{u}' \\ \boldsymbol{p}' \end{bmatrix} =  \exp(t \boldsymbol{A} )  \begin{bmatrix} \boldsymbol{u} \\ \boldsymbol{p} \end{bmatrix}   \;,
\]
\[
\varphi^{(B)}_{t}( \boldsymbol{u}, \boldsymbol{p} )  = \left(\boldsymbol{u},  \boldsymbol{p} + t \boldsymbol{F}( \boldsymbol{u}) \right) \;,
\]
where $\exp( \cdot)$ is the matrix exponential.

Let $\Delta t>0$ be a time step size parameter.  To obtain a weak approximation to the Langevin SPDE, we use a palindromic\footnote{Other authors use the terms `symmetric' or `self-adjoint.'  See \S2 of Ref.~\cite{campos2017palindromic} or \S3 of Ref.~\cite{BoSaActaN2018} for more about palindromic integrators.} composition of these flow maps \begin{equation} \label{eq:exact_splitting}
  \varphi^{(O)}_{(1/2) \Delta t} \circ \varphi^{(B)}_{(1/2) \Delta t} \circ \varphi^{(A)}_{\Delta t} \circ  \varphi^{(B)}_{(1/2) \Delta t} \circ \varphi^{(O)}_{(1/2) \Delta t} \;.
\end{equation}  When $\boldsymbol{F}=0$ and $\gamma=0$, this exact splitting is exact, and hence, overcomes the restrictive stability requirement of a Verlet integrator.
However, in the infinite-dimensional context, this exact splitting is still not satisfactory because the map \begin{equation}  \label{eq:exact_splitting_hamiltonian}
\varphi^{(B)}_{(1/2) \Delta t} \circ \varphi^{(A)}_{\Delta t} \circ  \varphi^{(B)}_{(1/2) \Delta t} 
\end{equation} is prone to linear resonance instabilities.  This limits the performance of the exact splitting in nonlinear problems as Figure~\ref{fig:nonlinear_hamiltonian_system_energy} illustrates. 


This instability stems from the fact that the matrix  $\exp(t \boldsymbol{A})$ associated to the exact flow of \eqref{eq:semidiscrete_intro_A}  is not always a strongly stable symplectic matrix.   A symplectic matrix is said to be {\em strongly stable} if all sufficiently close symplectic matrices are stable.   A sufficient condition for this to hold is that the symplectic matrix has a simple spectrum on the unit circle in the complex plane, i.e., all eigenvalues of the matrix are distinct and each have unit modulus. This sufficient condition is part of a theory due to Krein which also provides necessary and sufficient conditions for strong stability of symplectic matrices \cite{krein1950generalization}. For a graphical illustration of this sufficient condition see Figure~\ref{fig:strongly_stable}, and for an expository introduction to this concept of strong stability of symplectic matrices see \cite[\S25 \& \S42]{arnol2013mathematical}.  The right panel of Figure~\ref{fig:eigenvalues_cayley_exp} illustrates why the matrix associated to \eqref{eq:semidiscrete_intro_A} is not strongly stable in the presence of fast frequencies in the dynamics.  Krein's theorem motivates replacing the exponential map in \eqref{eq:exact_splitting_hamiltonian} by a strongly stable map.  

As illustrated in the left panel of Figure~\ref{fig:eigenvalues_cayley_exp}, one such map is given by a Cayley approximation.  Indeed, as the figure illustrates, the eigenvalues of a Cayley approximation fulfill Krein's sufficient condition for strong stability. Thus, resonance instabilities can be avoided by replacing the matrix exponential in the exact flow of \eqref{eq:semidiscrete_intro_A} by a Cayley approximation \begin{equation} \label{eq:cayley_splitting_hamiltonian}
\varphi^{(B)}_{(1/2) \Delta t} \circ \phi^{(A)}_{\Delta t} \circ  \varphi^{(B)}_{(1/2) \Delta t}  \;.
\end{equation}
Here $\phi^{(A)}_{\Delta t}$ is defined as the linear transformation with matrix $\cay((\Delta t) \boldsymbol{A})$ where $\cay( \cdot )$ is the Cayley transform which inputs a matrix $\boldsymbol{X}$ and outputs the matrix \begin{equation} \label{eq:cayley}
\cay( \boldsymbol{X} ) = \left( \mathbf{I} - \frac{1}{2} \mathbf{X} \right)^{-1}  \left( \mathbf{I} + \frac{1}{2} \mathbf{X} \right) \;.
\end{equation}
Since $( \boldsymbol{I} - \frac{1}{2} \boldsymbol{X}) ( \boldsymbol{I} + \frac{1}{2} \boldsymbol{X})  =  ( \boldsymbol{I} + \frac{1}{2} \boldsymbol{X}) ( \boldsymbol{I} -\frac{1}{2} \boldsymbol{X}) $, we can equally write \[
\cay( \boldsymbol{X}) =  \left( \boldsymbol{I} + \frac{1}{2} \boldsymbol{X} \right) \left( \boldsymbol{I} - \frac{1}{2} \boldsymbol{X} \right)^{-1} \;.
\]  If the input matrix is a Hamiltonian matrix, then the Cayley transform outputs a symplectic matrix with unit determinant \cite[\S2.5]{MaRa1999}.  Additionally, if the input matrix is reversible with respect to $(\boldsymbol{u}, \boldsymbol{p}) \mapsto (\boldsymbol{u}, -\boldsymbol{p})$, then the Cayley transform outputs a matrix that is also reversible \cite[\S2.4]{BoSaActaN2018}.  For proofs of these statements see Lemmas~\ref{lemma:cayley_symplectic} and~\ref{lemma:cayley_reversible}, respectively.  These two properties of the Cayley transform imply that the map $\phi^{(A)}_{\Delta t}$ is volume-preserving and reversible just like the exact flow of \eqref{eq:semidiscrete_intro_A}.\footnote{A map $\boldsymbol{\varphi}$  is {\em reversible} with respect to a linear involution $\boldsymbol{\rho}$  if $\boldsymbol{\varphi} \circ \boldsymbol{\rho} \circ \boldsymbol{\varphi} = \boldsymbol{\rho}$, and {\em volume-preserving} if  $\abs{\det( D \boldsymbol{\varphi})} = 1$.} The Langevin counterpart of \eqref{eq:cayley_splitting_hamiltonian} is given by \begin{equation} \label{eq:cayley_splitting}
  \varphi^{(O)}_{(1/2) \Delta t} \circ \varphi^{(B)}_{(1/2) \Delta t} \circ \phi^{(A)}_{\Delta t}  \circ  \varphi^{(B)}_{(1/2) \Delta t} \circ \varphi^{(O)}_{(1/2) \Delta t} \;.
\end{equation}
Hereafter we call \eqref{eq:cayley_splitting} and \eqref{eq:cayley_splitting_hamiltonian} {\em Cayley splittings}, and we call \eqref{eq:exact_splitting} and \eqref{eq:exact_splitting_hamiltonian} {\em exact splittings}.

Note that for the matrix $\boldsymbol{A}$ defined in \eqref{eq:semidiscrete_intro_H}, the following formula for $\cay((\Delta t) \boldsymbol{A} )$ is better for computations than \eqref{eq:cayley} \begin{equation} \label{eq:cayley_for_computations}
\cay((\Delta t) \boldsymbol{A}) = \begin{bmatrix} \left(  \mathbf{I} - \frac{\Delta t^2}{4} \mathbf{L} \right)^{-1} \left( \mathbf{I} + \frac{\Delta t^2}{4} \mathbf{L} \right) & \Delta t  \left(  \mathbf{I} - \frac{\Delta t^2}{4} \mathbf{L} \right)^{-1}   \\  \Delta t \left(  \mathbf{I} - \frac{\Delta t^2}{4} \mathbf{L} \right)^{-1}  \mathbf{L} &  \left(  \mathbf{I} - \frac{\Delta t^2}{4} \mathbf{L} \right)^{-1}  \left( \mathbf{I} + \frac{\Delta t^2}{4} \mathbf{L} \right) \end{bmatrix} \;.
\end{equation}
In particular, since $\boldsymbol{L}$ is typically a sparse tridiagonal matrix, the action of the matrix $\cay((\Delta t) \boldsymbol{A} )$ on a vector can be performed in $O(\Delta s^{-1})$ operations using the Thomas algorithm from numerical linear algebra.   In contrast, since the matrix exponential of a large sparse matrix is generally a full matrix, the action of the matrix $\exp( (\Delta t) \boldsymbol{A} )$ on a vector would take $O(\Delta s^{-2})$ operations \cite{al2011computing}. Also, from \eqref{eq:cayley_for_computations}, it is clear that accuracy at least requires that $| \Delta t^2 | \| \mathbf{L}  \| < 4$ where $\| \cdot \|$ is a matrix norm.  In particular, this condition is necessary to ensure that the corresponding series representation of the Cayley transform converges.  For a detailed proof see Lemma~\ref{lemma:cayley_series}.

In this paper, we only consider Dirichlet boundary conditions.  However, our results are relevant to other boundary conditions.  In this setting, the main results of the paper are the following.
\begin{description}
\item[1. Stability:] For a linear Hamiltonian PDE where the initial momentum is spatial Gaussian white noise, we prove that the Cayley splitting in \eqref{eq:cayley_splitting_hamiltonian} is stable; see Proposition~\ref{prop:Cayley_Splitting_Stability}.  In contrast, the exact splitting in \eqref{eq:exact_splitting_hamiltonian} is not stable under the same conditions; see the counterexample in Figure~\ref{fig:linear_hamiltonian_system_energy} and Figure~\ref{fig:linear_hamiltonian_system_spectral_energy}.   Related to this, we show how one can use Cayley integrators to avoid linear resonance instabilities in highly oscillatory Hamiltonian ODEs; see Remark~\ref{rmk:multiple_time_step_integrators}.   More generally, exact splittings are prone to linear resonance instabilities when applied to second order Langevin SPDEs, Hamiltonian PDEs or highly oscillatory Hamiltonian ODEs.   We also provide sufficient conditions so that the Cayley splitting in \eqref{eq:cayley_splitting} is stable when applied to a linear Langevin SPDE; see Proposition~\ref{prop:Cayley_Splitting_Stability_Langevin}.   These conditions turn out to be stronger than in the Hamiltonian case because we require that the drift part of the Cayley splitting inherits the asymptotically stability of the drift part of the Langevin dynamics.
\item[2. Accuracy:] In the same model problem, we quantify the global error of the Cayley splitting in \eqref{eq:cayley_splitting} in representing position, momentum and energy; see Propositions~\ref{prop:Cayley_Splitting_strong_accuracy},~\ref{prop:Cayley_Splitting_mean_dH}, and~\ref{prop:Cayley_Splitting_nu_mean_dH}.   Note that we estimate energy errors from initial distributions at equilibrium and out-of-equilibrium.  We also provide sufficient conditions so that the Cayley splitting in \eqref{eq:cayley_splitting} is weakly accurate when applied to a linear Langevin SPDE; see Prop.~\ref{prop:Cayley_Splitting_weak_accuracy_langevin}. 
\item[3. Non-Preconditioned HMC:]  The Cayley splitting in \eqref{eq:cayley_splitting_hamiltonian} enables doing infinite-dimensional HMC with non-preconditioned Hamiltonian dynamics.  In particular, to obtain an acceptance probability that converges to a nontrivial limit for a global move in state space as the spatial step size $\Delta s$ tends to zero, we prove it is sufficient to select the time step size as $\Delta t \lesssim \Delta s^{1/4}$; see Prop.~\ref{prop:hmc_ap}.   We also link this Cayley-based HMC with non-preconditioned MALA \cite{BeRoStVo2008}; see Prop.~\ref{prop:mala_hmc}.  Related to this, we show how to set the invariant distribution of the Cayley splitting \eqref{eq:cayley_splitting} by combining \eqref{eq:cayley_splitting_hamiltonian} with Metropolis-Hastings Monte-Carlo \cite{MeRoRoTeTe1953,Ha1970}.  In particular, it suffices to Metropolize the approximation of the Hamiltonian part of the splitting \cite{BoVa2010, BoVa2012, Bo2014}.  This Metropolized Cayley splitting gives an accurate and measure-preserving numerical method for \eqref{eq:langevin_spde_intro}.    
\end{description}

\paragraph{Organization of Paper}  In \S\ref{sec:why} we explain our rationale for using the Cayley splitting method for solving infinite-dimensional Hamiltonian systems and their 
second-order Langevin counterparts.  As an application, in \S\ref{sec:diffusion_bridges} we consider a class of second-order Langevin SPDEs whose marginal invariant measure 
in position is the law of a multidimensional diffusion bridge \cite{ReVa2005}.   We end the paper with a conclusion.

\begin{figure}
\begin{center}
\includegraphics[width=0.65\textwidth]{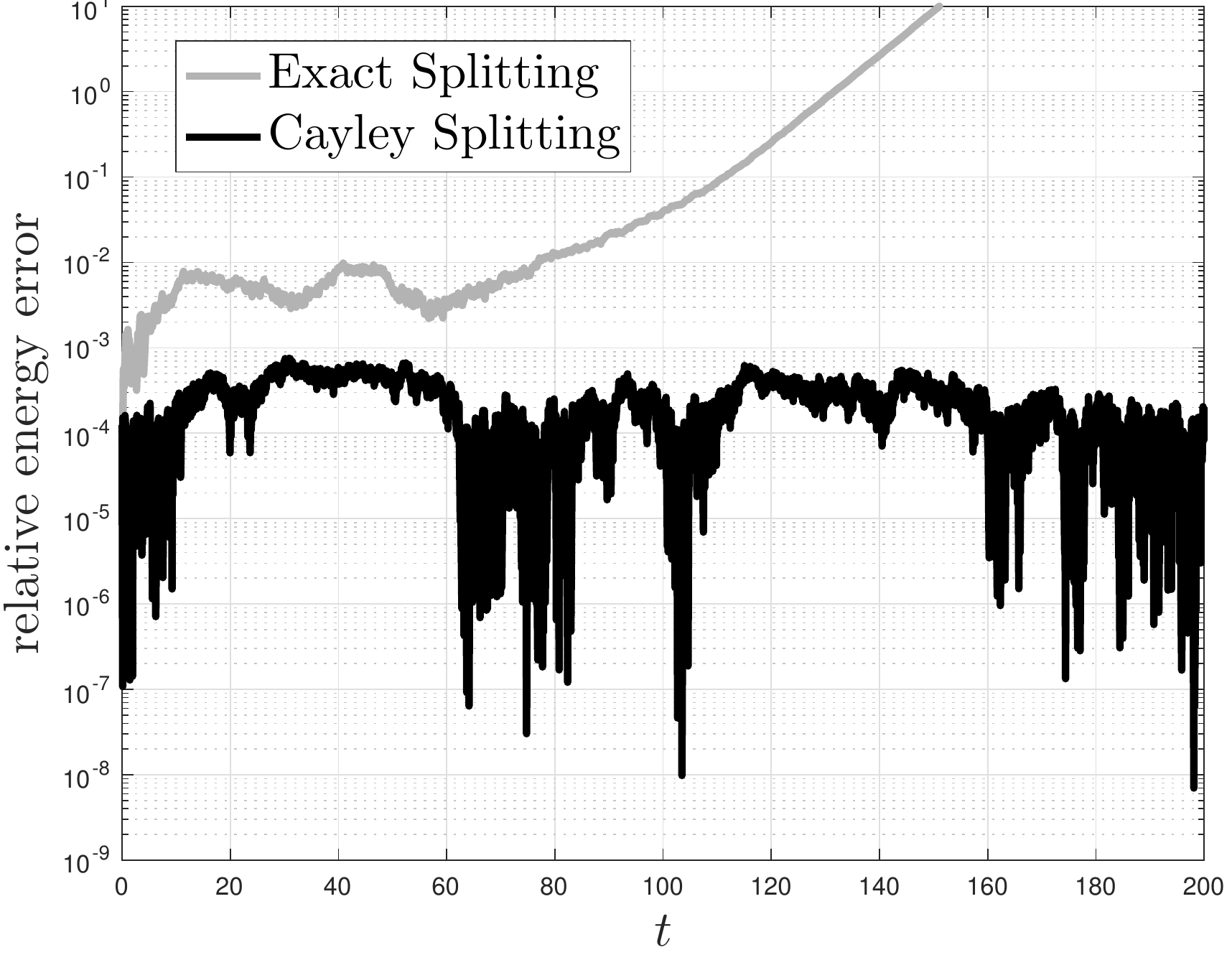}
\end{center}
\caption{\small  {\bf Stability of Cayley Splitting \& Resonance Instability of Exact Splitting.}  This figure shows the relative energy error as a function of time $t$ along trajectories
produced by the exact and Cayley splittings applied to a semilinear Hamiltonian PDE.  These splittings are given in \eqref{eq:exact_splitting_hamiltonian} and \eqref{eq:cayley_splitting_hamiltonian}, respectively.
The underlying potential energy function is a so-called path potential function associated to a two-dimensional diffusion bridge described in
\S\ref{sec:three_well_example}.   We discretize the spatial domain $[0,1]$ using an evenly spaced grid with $n=400$ grid points. 
The time step size in both splittings is $\Delta t = 0.00625$.   
}
  \label{fig:nonlinear_hamiltonian_system_energy}
\end{figure}

\begin{figure}
\begin{center}
\includegraphics[width=0.45\textwidth]{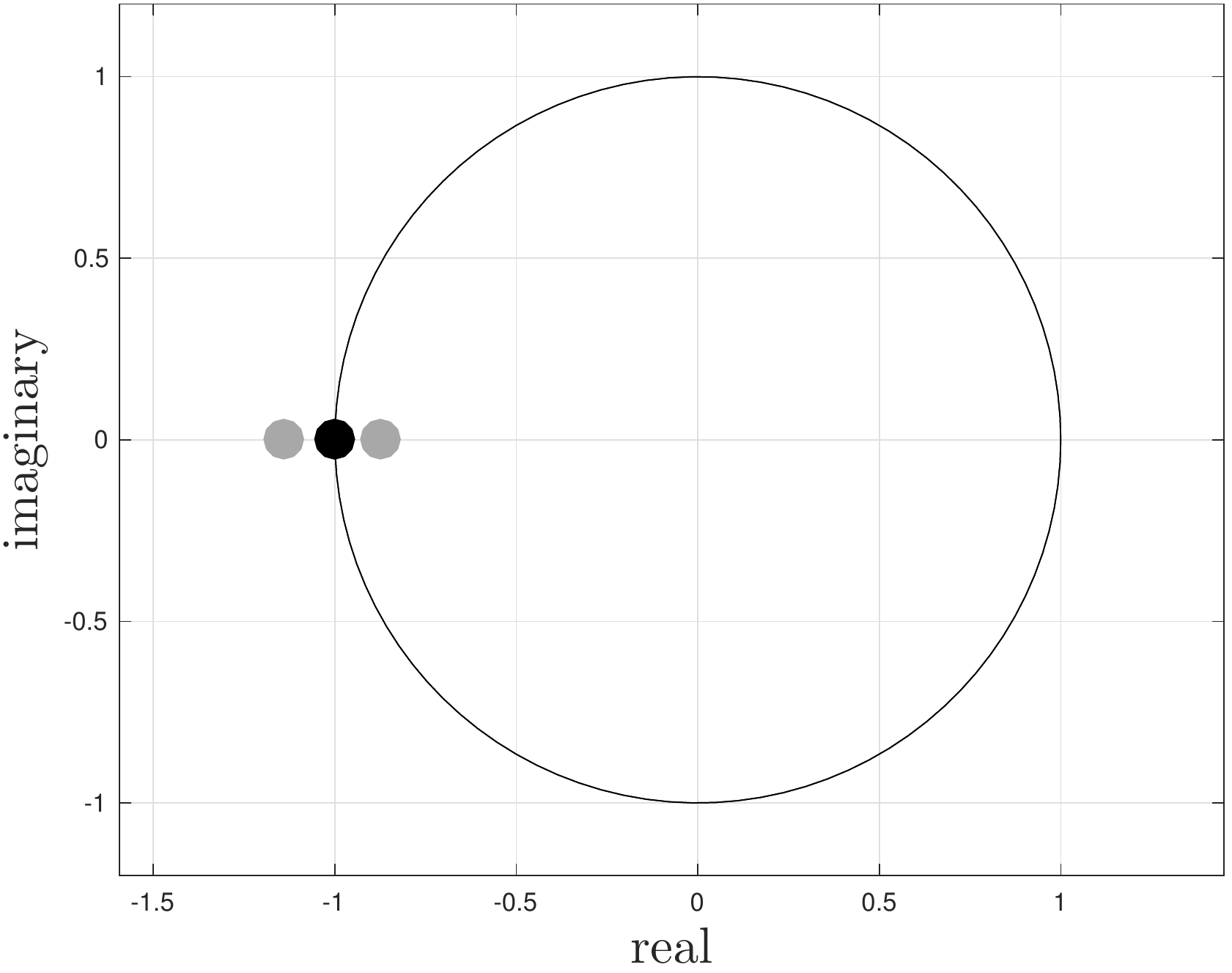} \hspace{0.1in}
\includegraphics[width=0.45\textwidth]{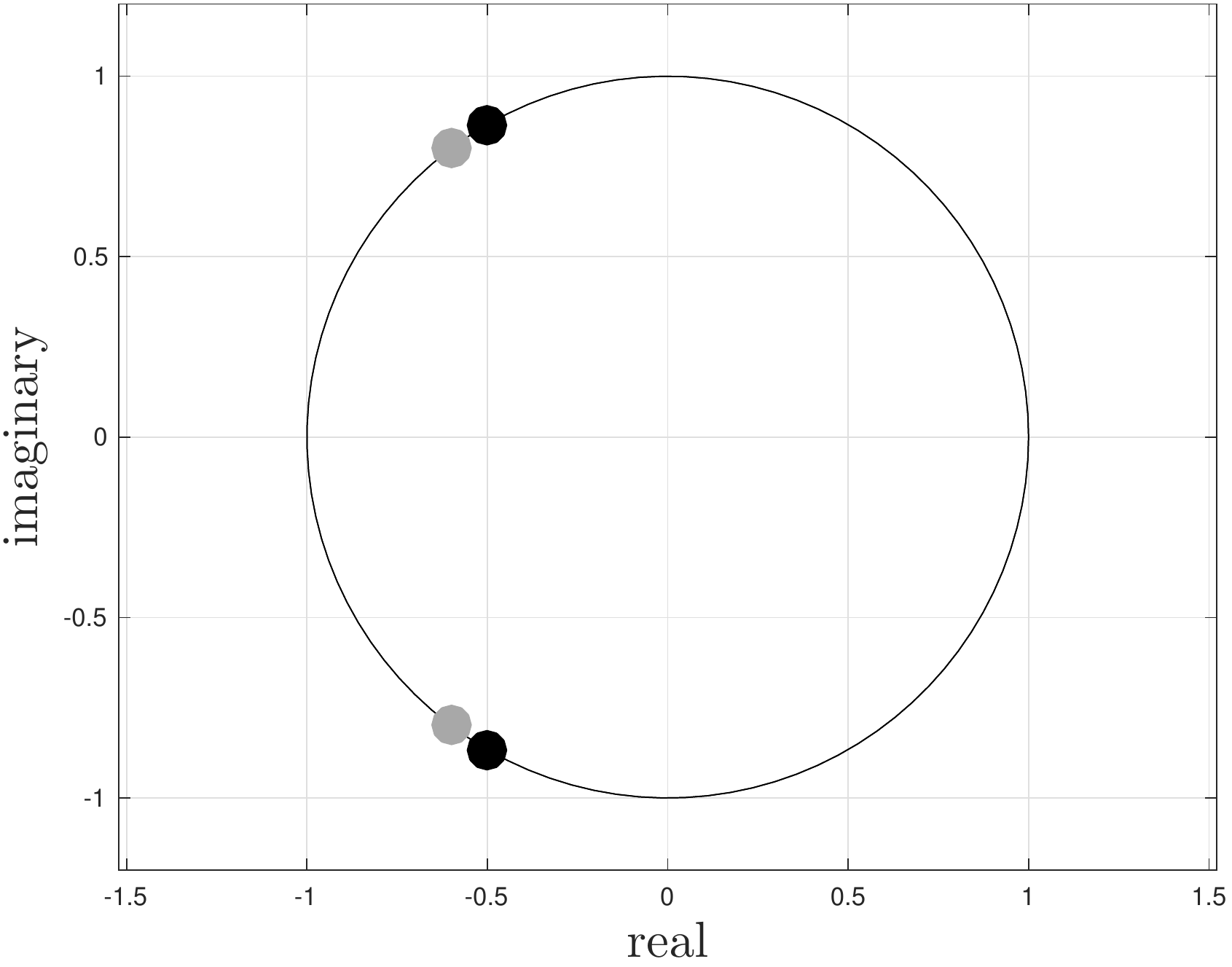} 
\end{center}
\caption{\small  {\bf Eigenvalues of Planar Symplectic Maps.}  
Eigenvalues of a symplectic matrix (black dots) along with a Hamiltonian perturbation of this matrix (grey dots) are plotted in the complex plane.  If the eigenvalues of a symplectic matrix are not distinct and lie on the unit circle (left panel), then the perturbed symplectic matrix may have an eigenvalue with modulus greater than one, and hence, the perturbed matrix loses stability.  However, if the eigenvalues are distinct and lie on the unit circle (right panel), then the eigenvalues of the perturbed map lie on the unit circle.  In the latter case, the symplectic matrix is said to be strongly stable.
}
  \label{fig:strongly_stable}
\end{figure}

\begin{figure}
\begin{center}
\includegraphics[width=0.45\textwidth]{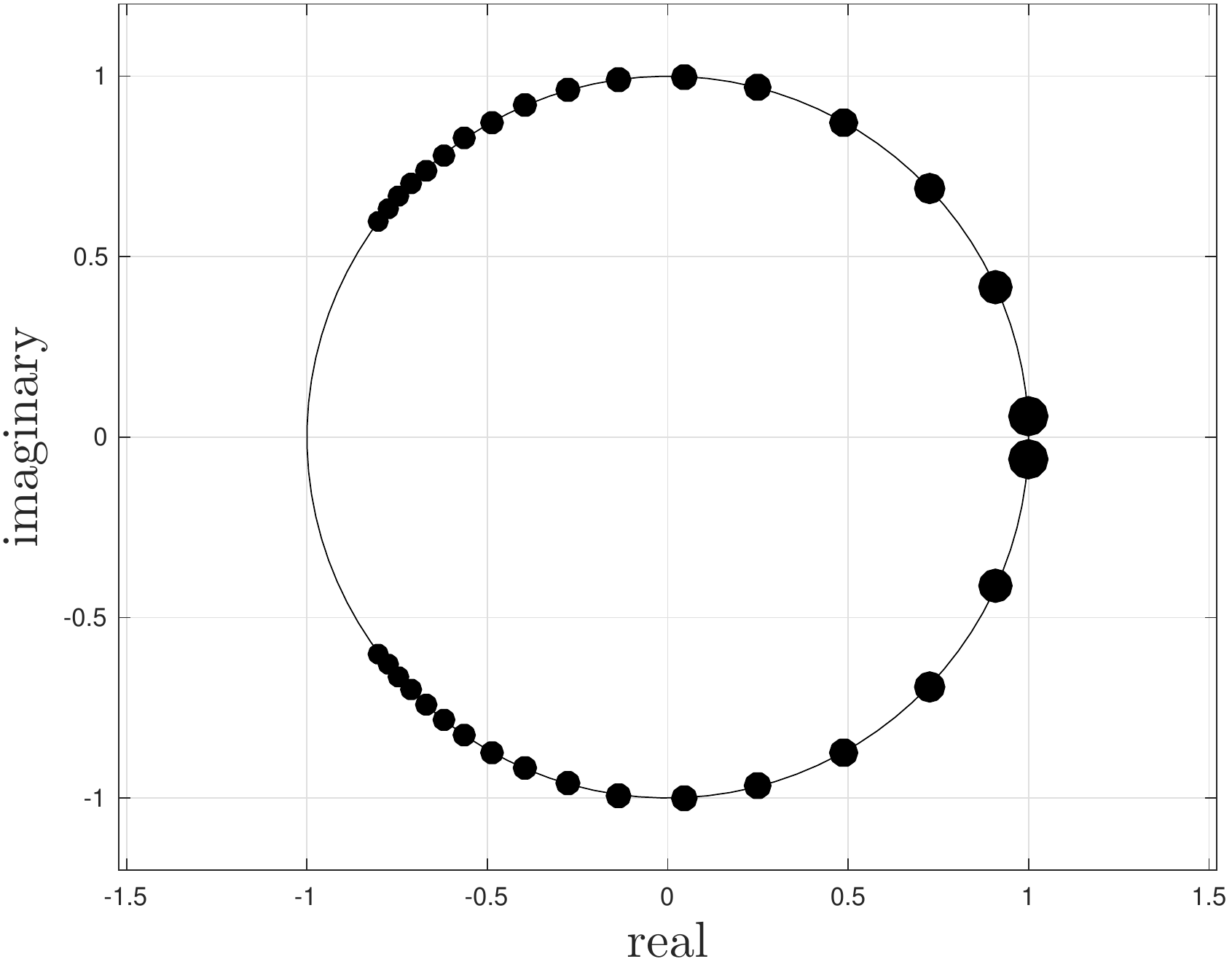} \hspace{0.1in}
\includegraphics[width=0.45\textwidth]{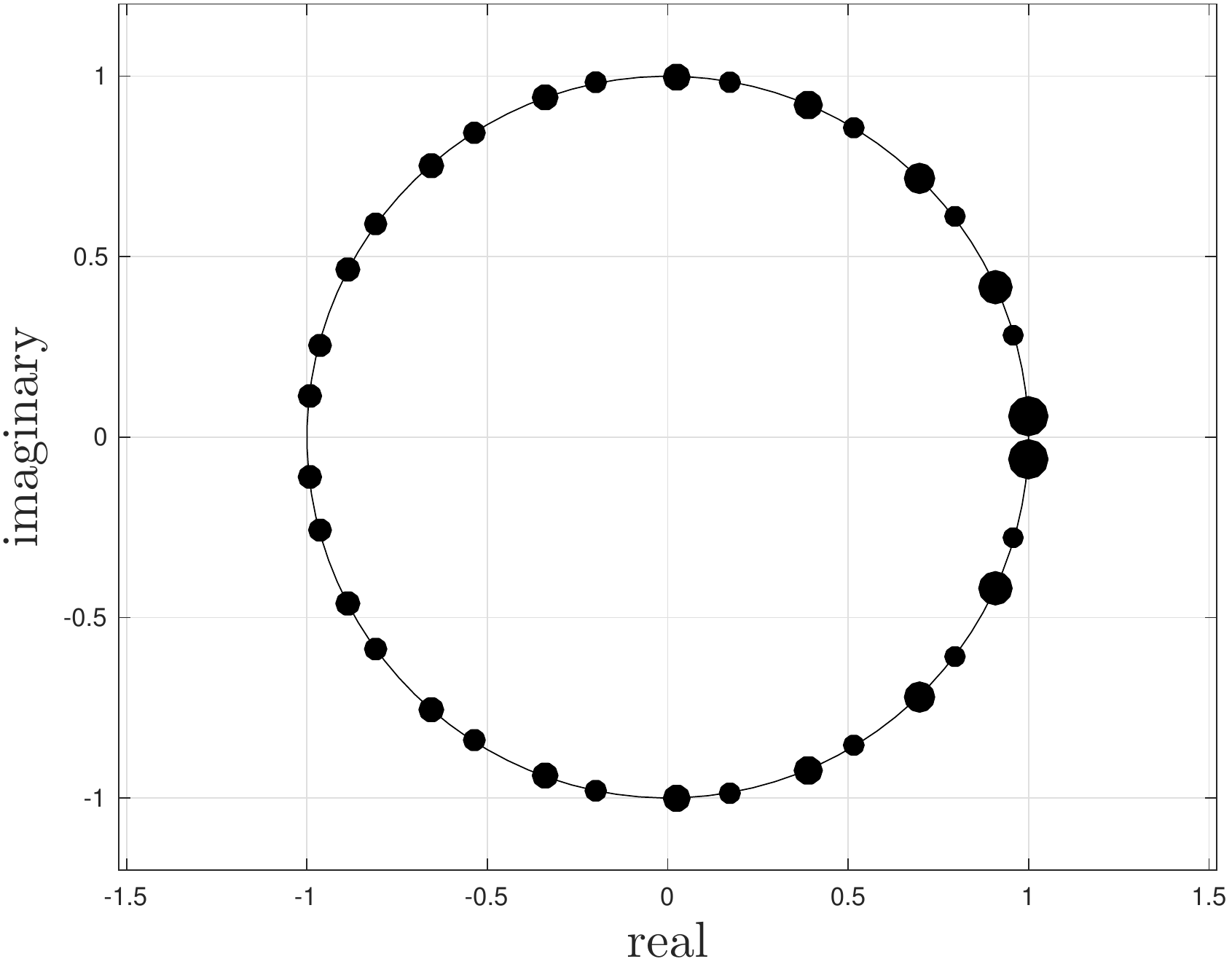}
\end{center}
\caption{\small  {\bf Eigenvalues of Cayley vs.~Exponential Maps.}  
Eigenvalues of the Cayley (left) and the exponential (right) maps applied to $H(q,p) = (1/2) (p^2 + \omega^2 q^2)$ at 
$17$ different snapshots in time and with $\omega=3$.     The size of the dots is related to time: 
dots corresponding to larger values of time have smaller size.  For the exact splitting, the eigenvalues rotate around
the unit circle multiple times.  However, for the Cayley splitting, the eigenvalues start near $(1,0)$, but never reach $(-1,0)$.  Since the
eigenvalues of the Cayley map are always distinct, the Cayley splitting is a strongly stable symplectic map, whereas the exponential
map loses strong stability every time the eigenvalues hit the horizontal axis.  The presence of high frequencies in the dynamics can
induce such linear resonance instabilities. 
}
  \label{fig:eigenvalues_cayley_exp}
\end{figure}


\section{Rationale for Cayley Splitting} \label{sec:why}

Here we assess the numerical stability and accuracy of the Cayley splittings in \eqref{eq:cayley_splitting_hamiltonian}
and \eqref{eq:cayley_splitting}.  We check that \eqref{eq:cayley_splitting_hamiltonian} is stable when applied to a linear 
Hamiltonian PDE with random initial conditions whose energy is almost surely infinite.   In this linear Hamiltonian PDE context, we also check that \eqref{eq:cayley_splitting_hamiltonian} 
is (i) strongly accurate for a point mass initial distribution; and (ii) accurate in representing the mean energy from equilibrium and non-equilibrium initial distributions.  
In a linear Langevin SPDE context, we formulate sufficient conditions for stability and check that \eqref{eq:cayley_splitting} is weakly accurate.  To assess the method's ergodic 
properties, we consider a HMC algorithm based on the Cayley splitting in \eqref{eq:cayley_splitting_hamiltonian}.  


\subsection{Preliminaries}

Here we gather some basic facts about: (i) the Cayley transform of a matrix which is an important ingredient of the Cayley splitting method; and (ii) the 
Cayley splittings applied to one-dimensional, linear Hamiltonian and second-order Langevin equations.   Depending on the context, the notation $\| \cdot \|$ 
denotes either the standard Euclidean norm of a vector or the Frobenius norm of a matrix.

\paragraph{Basic Properties of Cayley Transform}

The Cayley transform in \eqref{eq:cayley} is a matrix generalization of the rational function \begin{equation} \label{eq:rational_function}
f(x) = \frac{2+x}{2-x} \;.
\end{equation}
It has useful properties which we briefly review.   {\em Throughout this part we assume that the matrix $\boldsymbol{A}$ is such that 
$( \boldsymbol{I} - (1/2) \boldsymbol{A})$ is invertible, and hence, the Cayley transform of $\boldsymbol{A}$ is well-defined. This assumption 
is essential to every result that follows.}

\begin{lemma}
A matrix $\boldsymbol{A}$ and its Cayley transform $\cay( \boldsymbol{A})$ commute.
\end{lemma}

\begin{proof}
Since $\boldsymbol{A} (\boldsymbol{I} - \frac{1}{2}  \boldsymbol{A}) = (\boldsymbol{I} - \frac{1}{2}  \boldsymbol{A}) \boldsymbol{A}$, it follows that
\begin{align*}
\boldsymbol{A} \cay( \boldsymbol{A}) &=   ( \boldsymbol{I} - \frac{1}{2} \boldsymbol{A})^{-1}  \boldsymbol{A}  ( \boldsymbol{I} + \frac{1}{2} \boldsymbol{A}) = \cay( \boldsymbol{A}) \boldsymbol{A}  
\end{align*}
as required. 
\end{proof}

\medskip
The following result is key to proving the Cayley splitting in \eqref{eq:cayley_splitting_hamiltonian} is reversible.

\begin{lemma}  \label{lemma:cayley_reversible}
Let $\boldsymbol{A}$ be reversible with respect to an involutory matrix $\boldsymbol{S}$, i.e., $\boldsymbol{S} \boldsymbol{A}  = - \boldsymbol{A} \boldsymbol{S}$. Then $\cay(\boldsymbol{A})$ satisfies \[
 \boldsymbol{S}  \cay(\boldsymbol{A})  \boldsymbol{S}   = \cay(\boldsymbol{A})^{-1} \;.
\]
\end{lemma}

\begin{proof}
Since $ \boldsymbol{S}$ is involutory and $\boldsymbol{A}$ is reversible with respect to $ \boldsymbol{S}$, it follows that \begin{align*}
 \boldsymbol{S}  \cay(\boldsymbol{A})  \boldsymbol{S} &= (  \boldsymbol{S} - \frac{1}{2}   \boldsymbol{A} \boldsymbol{S})^{-1} (  \boldsymbol{S} + \frac{1}{2} \boldsymbol{A}  \boldsymbol{S} ) \\
 &= (  \boldsymbol{I} + \frac{1}{2}   \boldsymbol{A})^{-1}  \boldsymbol{S}^2 (  \boldsymbol{I} - \frac{1}{2} \boldsymbol{A}) = \cay(\boldsymbol{A})^{-1}
\end{align*}
as required.
\end{proof}

\medskip
The following result is key to proving the Cayley splitting in \eqref{eq:cayley_splitting_hamiltonian} is symplectic, 
and hence, volume-preserving.

\begin{lemma}[{{\cite[\S2.5]{MaRa1999}}}] \label{lemma:cayley_symplectic}
A $2N \times 2N$ matrix $\boldsymbol{A}$ is Hamiltonian if and only if $\cay(\boldsymbol{A})$ is a symplectic matrix.
\end{lemma}

\begin{proof}
Let $\mathbb{J}$ be the standard symplectic matrix, which satisfies $\mathbb{J}^{\mathrm{T}}=-\mathbb{J}$ and $\mathbb{J}^2 = -\boldsymbol{I}$.
The matrix $ \cay(\boldsymbol{A})  $ is symplectic if and only if
\[ 
 \cay(\boldsymbol{A})^{\mathrm{T}}
\mathbb{J}  \cay(\boldsymbol{A})  =  \mathbb{J},
\] 
which is equivalent  to
\[ 
(\boldsymbol{I} + \frac{1}{2} \boldsymbol{A})^{\mathrm{T}} \mathbb{J} (\boldsymbol{I} + \frac{1}{2} \boldsymbol{A})  =  (\boldsymbol{I} -\frac{1}{2} \boldsymbol{A})^{\mathrm{T}} \mathbb{J} (\boldsymbol{I} -\frac{1}{2} \boldsymbol{A}) \;.
\] 
Expanding both sides of this equation,
\[		\mathbb{J} +
\frac{1}{2} \mathbb{J} \boldsymbol{A} + \frac{1}{2} \boldsymbol{A}^{\mathrm{T}}\mathbb{J} +\frac{1}{4} \boldsymbol{A}^{\mathrm{T}}\mathbb{J} \boldsymbol{A}  =  \mathbb{J} -
\frac{1}{2} \mathbb{J}\boldsymbol{A} - \frac{1}{2} \boldsymbol{A}^{\mathrm{T}}\mathbb{J} + \frac{1}{4} \boldsymbol{A}^{\mathrm{T}}\mathbb{J}\boldsymbol{A}\]  and then simplifying yields,
\[	\mathbb{J}\boldsymbol{A}  =  - \boldsymbol{A}^{\mathrm{T}}\mathbb{J} = (\mathbb{J} \boldsymbol{A}) ^{\mathrm{T}} \] which states that $\boldsymbol{A}$  is a Hamiltonian matrix.  

\end{proof}

\medskip
The next lemma relates the eigenvalues/eigenvectors of a matrix and its Cayley transform.

\begin{lemma} \label{lemma:spectrum_cayley}
Let $\boldsymbol{A}$ be an $N \times N$ matrix with $N$ linearly independent eigenvectors and with eigendecomposition $\boldsymbol{A} = \boldsymbol{V}\boldsymbol{\Lambda}\boldsymbol{V}^{-1}$
where $\boldsymbol{\Lambda}$ is a diagonal matrix of eigenvalues and $\boldsymbol{V}$ is a matrix whose columns are the corresponding eigenvectors.  
Then \[
\cay(\boldsymbol{A}) = \boldsymbol{V} \boldsymbol{\tilde \Lambda} \boldsymbol{V}^{-1}\;,  \quad \text{where} \quad  \boldsymbol{\tilde \Lambda} = \cay (  \boldsymbol{\Lambda}  ) \;.
\] \end{lemma}

Note that $ \boldsymbol{\tilde \Lambda}$ is a diagonal matrix.

\begin{proof}
By substituting the eigendecomposition of $\boldsymbol{A}$, we obtain \begin{align*}
\cay(\boldsymbol{A}) &= \cay(\boldsymbol{V}\boldsymbol{\Lambda}\boldsymbol{V}^{-1}) \\
&=  (\boldsymbol{V}\boldsymbol{V}^{-1}  - \frac{1}{2} \boldsymbol{V}\boldsymbol{\Lambda}\boldsymbol{V}^{-1})^{-1}  ( \boldsymbol{V}\boldsymbol{V}^{-1} + \frac{1}{2} \boldsymbol{V}\boldsymbol{\Lambda}\boldsymbol{V}^{-1}) \\
&= \boldsymbol{V} \cay( \boldsymbol{\Lambda}) \boldsymbol{V}^{-1}
\end{align*}
as required.  
\end{proof}

\begin{lemma} \label{lemma:cayley_series}
In the situation of the preceding Lemma, and assuming that $\| \boldsymbol{A} \| < 2$, then \[
\cay( \boldsymbol{A} ) = \boldsymbol{I} + \sum_{k \ge 0} \frac{1}{2^k}  \boldsymbol{A}^{k+1} \;,
\] and for any $M \in \mathbb{N}$, \[
\left\| \sum_{k \ge M} \frac{1}{2^k}  \boldsymbol{A}^{k+1} \right\| \le  \| \boldsymbol{A} \|^{M+1} \frac{2^{1-M}}{2 - \| \boldsymbol{A} \|} \;.
\]
\end{lemma}

\begin{proof}
For any $x \in \mathbb{R}$ satisfying $|x|<2$, we have \begin{equation} \label{eq:cayley_series_scalar}
\frac{2+x}{2-x} = 1 + \sum_{k \ge 0} \frac{1}{2^k} x^{k+1} \;,
\end{equation}  and for any $M \in \mathbb{N}$, \begin{equation} \label{eq:cayley_series_remainder_scalar}
\sum_{k \ge M} \frac{1}{2^k} x^{k+1} = \frac{2^{1-M} x^{M+1}}{2-x} \;.
\end{equation}   By Lemma~\ref{lemma:spectrum_cayley},  \begin{align*}
\cay( \boldsymbol{A} ) &= \boldsymbol{V} \cay( \boldsymbol{\Lambda}) \boldsymbol{V}^{-1} \\
&= \boldsymbol{V} \left( \boldsymbol{I} + \sum_{k \ge 0} \frac{1}{2^k}  \boldsymbol{\Lambda}^{k+1}  \right)  \boldsymbol{V}^{-1} \\
&=  \boldsymbol{I} + \sum_{k \ge 0} \frac{1}{2^k}  \boldsymbol{A}^{k+1} 
\end{align*}
where we applied the series in \eqref{eq:cayley_series_scalar} to each eigenvalue of $\boldsymbol{A}$.  Moreover, \begin{align*}
\left\| \sum_{k \ge M} \frac{1}{2^k}  \boldsymbol{A}^{k+1}  \right\| &\le  \sum_{k \ge M} \frac{1}{2^k}  \| \boldsymbol{A} \|^{k+1} =  \frac{2^{1-M}  \| \boldsymbol{A} \|^{M+1}}{2- \| \boldsymbol{A} \|}  \;.
\end{align*} 
where we applied \eqref{eq:cayley_series_remainder_scalar} with $x = \| \boldsymbol{A} \|$, which is applicable since $ \| \boldsymbol{A} \| < 2$ by hypothesis.
\end{proof}

\paragraph{Cayley splitting for one-dimensional linear Hamiltonian} Using backward error analysis, one can show that symplectic integrators applied to Hamiltonian systems are {\em locally} interpolated by the solutions of modified equations that are themselves Hamiltonian; for a detailed exposition, see \cite{SaCa1994,LeRe2004,HaLuWa2010}.    In general, there is no global modified Hamiltonian for symplectic integrators.  However, in the linear context, there is a global modified Hamiltonian, which we can use to prove global numerical stability statements and derive rates of convergence.   The following part relies upon this modified Hamiltonian theory, particularly the results in \S4.1 of Ref.~\cite{BlCaSa2014}.   Later we extend these results to an infinite-dimensional Hamiltonian PDE.

Given $\omega>0$, consider the one-dimensional Hamiltonian function \begin{equation} \label{eq:hamiltonian_1D}
H(q,p) = \frac{1}{2} p^2 + \frac{1}{2} \omega^2 q^2 + \frac{1}{2} q^2 \;.
\end{equation} Hamilton's equations for $H(q,p)$ are \begin{equation} \label{eq:hamiltons_equations_1D}
\begin{bmatrix} \dot q(t) \\ \dot p(t) \end{bmatrix} = \begin{bmatrix} 0  &1\\ -\omega^2-1 & 0 \end{bmatrix} \begin{bmatrix} q(t) \\ p(t) \end{bmatrix}   \;.
\end{equation}
Let $\boldsymbol{A}$ and $\boldsymbol{B}$ be the following $2 \times 2$ matrices \[
\boldsymbol{A} = \begin{bmatrix} 0  &1\\ -\omega^2 & 0 \end{bmatrix}  \quad  \text{and} \quad 
\boldsymbol{B} = \begin{bmatrix} 0  & 0 \\ -1 & 0 \end{bmatrix}  \;.
\]
In this context, the Cayley integrator in \eqref{eq:cayley_splitting_hamiltonian} operated with time step size $\Delta t>0$ is the linear transformation with matrix \[
\boldsymbol{C} = \exp( (1/2) \Delta t \boldsymbol{B} ) \cay( \Delta t \boldsymbol{A} ) \exp( (1/2) \Delta t \boldsymbol{B} ) \;,
\] where $\exp( \cdot )$ is the matrix exponential and $\cay( \cdot )$ is the Cayley transform defined in \eqref{eq:cayley}.  
The entries of this matrix are given explicitly by \begin{equation} \label{eq:cayley_splitting_1D}
\boldsymbol{C} = \begin{bmatrix}  -1 + \dfrac{8 - 2 \Delta t^2}{4 + \Delta t^2 \omega^2} &  \dfrac{4 \Delta t}{4 + \Delta t^2 \omega^2}  \\
\dfrac{\Delta t (-4 + \Delta t^2) (1 + \omega^2  )}{4 + \Delta t^2 \omega^2} &    -1 + \dfrac{8 - 2 \Delta t^2}{4 + \Delta t^2 \omega^2} \end{bmatrix} \;.
\end{equation}
In all of the results that follow, one can readily rescale time by a factor in order to see the effect of a constant multiplying $\boldsymbol{B}$.

\medskip
As the next lemma shows, as long as $\Delta t < 2$, this Cayley splitting is stable uniformly in $\omega$, and 
one can bound $\left\| \boldsymbol{C}^m \right\|$ uniformly in $m$.

\begin{lemma} \label{lemma:Cayley_Splitting_Stability_1D}
For all $\omega >0$, and for any positive $\Delta t< 2$, the Cayley splitting in \eqref{eq:cayley_splitting_1D} is stable.  Moreover, \[
\left\| \boldsymbol{C}^m \right\| \lesssim 1+ \omega + \frac{1}{\sqrt{4 - \Delta t^2}} \;.
\] 
\end{lemma}

The proof that follows confirms Krein's theorem in this context \cite{krein1950generalization,arnol2013mathematical}.  

\begin{proof}
From \eqref{eq:cayley_splitting_1D}, a direct calculation shows that \[
\det(\boldsymbol{C}) = 1 \;, \quad   \tr(\boldsymbol{C}) =  -2 + \frac{4}{4 + \Delta t^2 \omega^2} (4 - \Delta t^2 )\;.
\]
Note that, for all $\omega>0$, and for any positive $\Delta t<2$, we have $|\tr(\boldsymbol{C})| < 2$.   Thus, the matrix $\boldsymbol{C}$
has complex conjugate eigenvalues of unit modulus, every matrix power of $\boldsymbol{C}$ is bounded, and hence, 
the Cayley splitting is numerically stable.  To get a more precise bound on $\boldsymbol{C}^m$, we follow \S4.1 of Ref.~\cite{BlCaSa2014} and introduce $\theta$ defined as $\cos(\theta) = \tr(\boldsymbol{C})/2$.
Since $|\tr(\boldsymbol{C})| < 2$, it follows that $\sin(\theta) \ne 0$ and we may define $\chi = \boldsymbol{C}_{12}/\sin(\theta)$.   
In terms of $\theta$ and $\chi$, one can write $\boldsymbol{C}$ as \begin{equation} \label{eq:C_theta_chi}
 \boldsymbol{C} = \begin{bmatrix} \cos(\theta) & \chi \sin(\theta) \\   - \chi^{-1} \sin(\theta) &  \cos(\theta)  \end{bmatrix} \;.
\end{equation} Comparing \eqref{eq:cayley_splitting_1D} to \eqref{eq:C_theta_chi}, we obtain \begin{equation} \label{eq:theta_chi}
\theta = 2 \arctan\left( \frac{\Delta t}{\sqrt{4-\Delta t^2}} \sqrt{1+\omega^2} \right) \;, \quad \chi = \frac{2}{\sqrt{4- \Delta t^2} \sqrt{1+\omega^2} }  \;.
\end{equation} Moreover, it is easy to check that  \[
 \boldsymbol{C}^m = \begin{bmatrix} \cos(m \theta) & \chi \sin(m \theta) \\   - \chi^{-1} \sin(m \theta) &  \cos(m \theta)  \end{bmatrix} 
\] for all $m \in \mathbb{N}$.    By the triangle inequality, \begin{align*}
\|  \boldsymbol{C}^m \| &\le  2 + | \chi^{-1} | + |\chi| \\
&\le 2 +  \frac{1}{2} \sqrt{ ( 4 - \Delta t^2) (1+\omega^2) } + \frac{2}{\sqrt{ ( 4 - \Delta t^2) (1+\omega^2) }} \\
&\le 2 +   \sqrt{ 1+\omega^2 } + \frac{2}{\sqrt{ 4 - \Delta t^2 }}
\end{align*}
as required.
\end{proof}

Next we quantify the global error of the Cayley splitting in preserving the Hamiltonian.
For all $(q,p) \in \mathbb{R}^2$ and for any $m \in \mathbb{N}$, define the global energy error after $m$ integration steps 
of the Cayley splitting as  \begin{equation} \label{eq:dH_1D}
\Delta(q,p) = H(q^m,p^m) - H(q,p)
\end{equation}
where we have defined \begin{equation} \label{eq:qm_pm}
\begin{bmatrix} q^m \\ p^m \end{bmatrix} =   \boldsymbol{C}^m \begin{bmatrix} q \\ p \end{bmatrix} \;.
\end{equation}
Prop.~4.1 of Ref.~\cite{BlCaSa2014} implies that the points $\{ (q^m, p^m) \mid m \in \mathbb{N} \}$ are interpolated by the solution of a so-called modified Hamiltonian system with modified Hamiltonian \begin{equation} 
\label{eq:modified_hamiltonian}
\tilde H(q,p) =  \frac{\chi}{2 \Delta t} \theta \left(   p^2 + \frac{1}{\chi^2} q^2 \right)  \end{equation}
where $\theta$ and $\chi$ are given in \eqref{eq:theta_chi}.   This modified Hamiltonian is globally defined. If $\Delta t<2$,  then $\chi^{-2} > 0$ and the level sets of $\tilde H$ are ellipses in the $(q,p)$-plane, 
as illustrated in Figure~\ref{fig:modified_hamiltonian_cayley}.

\medskip
The following lemma uses this modified Hamiltonian to obtain an upper bound on $\Delta(q,p)$ uniformly in $\omega$.

\begin{lemma} \label{lemma:Cayley_Splitting_Max_Energy_Error_1D}
For all $\omega >0$, for any positive $\Delta t<2$, for all $m \in \mathbb{N}$, and for all initial conditions $(q,p) \in \mathbb{R}^2$ we have \[
\Delta(q,p) \le  \frac{1}{2} \frac{\Delta t^2}{4-\Delta t^2} p^2   \;.
\] 
\end{lemma}

\begin{proof}
Since $\Delta t<2$, the Cayley splitting is numerically stable by Lemma~\ref{lemma:Cayley_Splitting_Stability_1D}. 
Moreover, it admits a global modified Hamiltonian given in \eqref{eq:modified_hamiltonian}.
Fix an initial condition $(q_0,p_0) \ne (0,0)$, to avoid a trivial solution.
By \eqref{eq:theta_chi}, we see that $\chi^{-2} = (1/4) (4-\Delta t^2) (1+\omega^2)$. 
The hypothesis $\Delta t<2$ implies that $\chi^{-2} < (1+\omega^2)$.
Therefore, on the ellipse $\{ (q,p) \mid \tilde H(q,p) = \tilde H(q_0, p_0) \}$, the Hamiltonian $H(q,p) = (1/2) (p^2 + (1+\omega^2) q^2)$ attains its maximum value at
the intersections of this ellipse with the line $p=0$.  These intersections occur at $(\pm q_{\star},0)$ where $\tilde H(q_{\star},0) = \tilde H(q_0, p_0)$ or $q_{\star}^2 = \chi^{2} p_0^2 + q_0^2$.  
 Hence, \[
 \Delta(q_0,p_0) \le H(q_{\star},0) - H(q_0, p_0) = \frac{1}{2} \frac{\Delta t^2}{4 - \Delta t^2} p_0^2
 \] as required.
\end{proof}

\medskip
The next lemma applies the modified Hamiltonian in \eqref{eq:modified_hamiltonian} to obtain a formula for the average energy error.
The parameter $\beta$ that appears in this lemma is an inverse temperature parameter.

\begin{lemma} \label{lemma:Cayley_Splitting_Mean_Energy_Error_1D}
For all $\omega >0$, for all $\beta>0$, for any positive $\Delta t<2$, and for all $m \in \mathbb{N}$, we have \[
\E ( \Delta ) =  \beta^{-1} \frac{\sin^2(m \theta)}{8} \frac{\Delta t^4}{4 -\Delta t^2 }
\] 
where the expected value is over random initial conditions with non-normalized density $e^{-\beta H(q,p)}$.
\end{lemma}

Comparing Lemma~\ref{lemma:Cayley_Splitting_Max_Energy_Error_1D} to Lemma~\ref{lemma:Cayley_Splitting_Mean_Energy_Error_1D}, note that the order of accuracy in $\E ( \Delta )$ is twice the order of accuracy in $\Delta(q,p)$; more on this point below.

\begin{proof}
Since $\Delta t<2$, the Cayley splitting is numerically stable by Lemma~\ref{lemma:Cayley_Splitting_Stability_1D}. Moreover, it admits a global modified Hamiltonian given in \eqref{eq:modified_hamiltonian}.
Set $\tau = m \Delta t$.  With the shorthand $c=\cos(m \theta)$ and $s=\sin(m \theta)$, the exact solution of the modified Hamiltonian system at time $\tau$ 
with initial condition $(q(0),p(0)) = (q, p)$ is given by: \begin{equation} \label{eq:modified_solution}
\tilde q(\tau) = c q + \chi s p  \;, \quad \tilde p(\tau) = c p - \chi^{-1} s q   \;.
\end{equation}  Hence, \begin{align*}
2 \Delta(q, p) =& \tilde p(\tau)^2 + (1+\omega^2) \tilde q(\tau)^2 - p^2 - (1+\omega^2) q^2  \\
=& s^2 ( (1+\omega^2) \chi^2 - 1) p^2 - s^2 ( 1+\omega^2  - \chi^{-2}   ) q^2  \\
&+ 2 c s ( - \chi^{-1} + (1+\omega^2) \chi ) q p 
\end{align*}
Now let $(q,p)$ be random with non-normalized density $e^{-\beta H(q,p)}$.
Since $\E ( p^2 ) = \beta^{-1}$, $\E ( q^2 ) = \beta^{-1} (1 + \omega^2 )^{-1}$, and $\E( q p ) = 0$, we have by \eqref{eq:theta_chi} \begin{align*}
\E (\Delta) = \beta^{-1} \frac{s^2}{2}  ( (1+\omega^2) \chi^2  +  (1+\omega^2 )^{-1} \chi^{-2}  - 2)  = \beta^{-1} \frac{s^2}{8} \frac{\Delta t^4}{4 -\Delta t^2 }
\end{align*}
as required.
\end{proof}

\medskip
We can similarly derive a formula for the variance of the energy error.

\begin{lemma} \label{lemma:Cayley_Splitting_Variance_Energy_Error_1D}
For all $\omega >0$, for all $\beta>0$, for any positive $\Delta t<2$, and for all $m \in \mathbb{N}$, we have \[
\Var ( \Delta ) = \beta^{-2}  \frac{\sin^2(m \theta)}{64} \frac{\Delta t^4}{(4 -\Delta t^2 )^2} \left( (8-\Delta t^2)^2 - \Delta t^4 \cos(2 m \theta) \right) 
\] 
where the variance is over random initial conditions with non-normalized density $e^{-\beta H(q,p)}$.  
\end{lemma}

It follows from Lemmas~\ref{lemma:Cayley_Splitting_Mean_Energy_Error_1D} and~\ref{lemma:Cayley_Splitting_Variance_Energy_Error_1D}
that both the mean and variance of $\Delta$ have the same order of accuracy, namely $O(\Delta t^4)$.   This is a general property of 
volume-preserving and reversible integrators \cite{BePiRoSaSt2013,BlCaSa2014}.  

\medskip 
Later in order to invoke the Lyapunov central limit theorem, we also require the following formula.

\begin{lemma} \label{lemma:Cayley_Splitting_4th_moment_Energy_Error_1D}
For all $\omega >0$, for all $\beta>0$, for any positive $\Delta t<2$, and for all $m \in \mathbb{N}$, we have \begin{align*}
\E & \left( \left( \Delta - \E ( \Delta) \right)^4 \right) =   \beta^{-4} \frac{3 \Delta t^8 \sin^4(m \theta)}{ 8192 (4 -\Delta t^2 )^4}   \left( 24576 - 12288 \Delta t^2 +2816 \Delta t^4  \right.  \\
& \left. -320 \Delta t^6 + 15 \Delta t^8 - 20 (8 - \Delta t^2)^2 \Delta t^4 \cos(2 m \theta) + 5  \Delta t^8 \cos(4 m \theta) \right) 
\end{align*}
where the expected value is over random initial conditions with non-normalized density $e^{-\beta H(q,p)}$.  
\end{lemma}

\medskip 
To obtain our infinite-dimensional strong accuracy results, we need the following estimate for the global error of 
the Cayley splitting.

\begin{lemma} \label{lemma:Cayley_Splitting_global_error_1D}
For all $\omega >0$, for any positive $\Delta t<\sqrt{3}$, for any initial condition $(q,p) \in \mathbb{R}^2$ 
and for all $T>0$, there exist positive constants $C_1, C_2$ such that \begin{equation}
\begin{aligned}
\left| q^{m} - q(m \Delta t) \right|  &\le  C_1 T \Delta t^2 ( (1+\omega^3) |q| + (1+\omega^2) |p| ) \\
 \left| p^{m} - p(m \Delta t) \right| &\le  C_2 T \Delta t^2 ( (1+\omega^4) |q| + (1+\omega^3) |p|) 
 \end{aligned}
\end{equation}
where $m = \lfloor T/\Delta t \rfloor$ and $(q^{m},p^{m} )$ is defined in \eqref{eq:qm_pm}.
\end{lemma}

These error bounds show that the Cayley splitting is second-order accurate on finite time intervals, and that
its global error depends linearly on the time interval of simulation $T$.  This result is not generic for numerical integrators, 
and as we see in the proof that follows, relies strongly upon the existence of a global modified Hamiltonian of the 
numerical solution.

\begin{proof}
Since $\Delta t<2$, by Lemma~\ref{lemma:Cayley_Splitting_Stability_1D}, the Cayley splitting is stable.  Moreover, its modified Hamiltonian in \eqref{eq:modified_hamiltonian} is globally well-defined.
Set $m = \lfloor T / \Delta t \rfloor$ and $\tau = m \Delta t$.  To estimate the global error, we compare the exact solution of \eqref{eq:hamiltons_equations_1D} \begin{align*}
q(\tau) & =  \cos(\tau \sqrt{1+\omega^2}   ) q +  \frac{1}{\sqrt{1+\omega^2}} \sin(\tau \sqrt{1+\omega^2}  ) p \\
p(\tau) & = - \sqrt{1+\omega^2}   \sin( \tau \sqrt{1+\omega^2} ) q +  \cos(\tau \sqrt{1+\omega^2}  ) p 
\end{align*} to the solution of the modified Hamiltonian system in \eqref{eq:modified_solution}. The global error in position is given by \begin{align*}
  |\tilde q(\tau) - q(\tau)|  & \le   \left| (\cos(m \theta) - \cos(m  \Delta t \sqrt{1+\omega^2} ) ) q \right| \\
 & \qquad  +  \left| \left( \chi - \frac{1}{\sqrt{1+\omega^2}}  \right)  \sin(m \theta) p \right| \\
 & \qquad +  \left| \frac{1}{\sqrt{1+\omega^2}} \left(  \sin(m \theta) - \sin(m  \Delta t \sqrt{1+\omega^2}) ) \right) p \right|  \\  
&\le   \left| \frac{\theta}{\Delta t} -  \sqrt{1+\omega^2} \right|  ( |q| + \frac{| p |}{\sqrt{1+\omega^2}}  ) m \Delta t  + \left|  \chi - \frac{1}{\sqrt{1+\omega^2}}  \right|  |p|    
\end{align*}
where we used the fact that sine and cosine are Lipschitz with Lipschitz constant $1$.  (Alternatively, we could have used Taylor's integral formula to fourth-order in place of a Lipschitz constant,
but the price of that estimate is a global error that grows faster than linearly with $T$.)  These differences can be estimated using   \begin{equation} \label{eq:phase_error}
\begin{aligned}
\frac{\theta}{\Delta t} -  \sqrt{1+\omega^2}  &=   g(\Delta t, \omega) \Delta t^2 \\
\chi - \frac{1}{\sqrt{1+\omega^2}}  &=  \frac{1}{\sqrt{1+\omega^2}} \frac{\Delta t^2}{4 - \Delta t^2 + 2 \sqrt{4 - \Delta t^2}} 
\end{aligned}
\end{equation}
where \[
| g(\Delta t, \omega)| \le \frac{\sqrt{1+\omega^2}}{\sqrt{4 - \Delta t^2} (2+\sqrt{4-\Delta t^2})} + \frac{2}{3} \frac{(1+\omega^2)^{3/2}}{ (4 - \Delta t^2)^{3/2} } \;,
\]
which follows from \eqref{eq:theta_chi}, the identity $\arctan(x) = x - \int_0^x t^2 / (1+t^2) dt$ and the bound $| \int_0^x t^2 / (1+t^2) dt | \le x^3 / 3$, which hold for all positive $x$.

Similarly, the global error in momentum is bounded by \begin{align*}
  |\tilde p(\tau) - p(\tau)|  &\le   \left| (  \sin(m \theta) - \sin(m  \Delta t \sqrt{1+\omega^2}) ) ) q \right| \sqrt{1+\omega^2}   \\
 & \qquad + \left| ( \chi^{-1} - \sqrt{1+\omega^2}  )  \sin(m \theta) q \right| \\
 & \qquad + \left| (\cos(m \theta) - \cos(m  \Delta t \sqrt{1+\omega^2}) ) p \right|   \\  
 &\le  \left| \frac{\theta}{\Delta t} -  \sqrt{1+\omega^2} \right|  ( \sqrt{1+\omega^2}  |q| +|p| ) m \Delta t   \\
 & \qquad +  \left|  \chi^{-1} - \sqrt{1+\omega^2} \right|  |p|    
\end{align*}
where the first difference can be estimated using \eqref{eq:phase_error}, and the second difference can be estimated via
\begin{align*}
\chi^{-1} - \sqrt{1+\omega^2}  &=  \sqrt{1+\omega^2} \frac{-\Delta t^2}{4 + 2 \sqrt{4 - \Delta t^2}} \;.
\end{align*}
\end{proof}

\paragraph{Cayley splitting for one-dimensional linear Langevin}  Next we consider the Langevin counterpart of \eqref{eq:hamiltons_equations_1D}  \begin{equation} \label{eq:langevin_equations_1D}
\begin{bmatrix} d q(t) \\ d p(t) \end{bmatrix} = \boldsymbol{K} \begin{bmatrix} q(t) \\ p(t) \end{bmatrix} dt + \sqrt{2 \gamma \beta^{-1}}  \begin{bmatrix} 0 \\  1 \end{bmatrix} dW(t)  \;, ~~ \boldsymbol{K} = \begin{bmatrix} 0  &1\\ -\omega^2-1 & -\gamma \end{bmatrix}
\end{equation}
where $W$ is a one-dimensional Wiener process, $\beta>0$ is an inverse temperature parameter, and $\gamma \ge 0$ is a friction parameter.  Given an initial condition $(q, p) \in \mathbb{R}^2$ and time $T>0$, the solution $(q(T), p(T))$ of \eqref{eq:langevin_equations_1D} is a Gaussian vector with mean vector and covariance matrix given respectively by \begin{equation} \label{eq:law_langevin_1D}
\boldsymbol{\mu}(T) = \boldsymbol{\Phi}_{\gamma}(T)  \begin{bmatrix} q \\ p \end{bmatrix} \;, \quad \boldsymbol{\Sigma}(T) = 2 \gamma \beta^{-1} \int_0^{\mathrm{T}} \boldsymbol{\Phi}_{\gamma}(s)  \begin{bmatrix} 0 & 0 \\ 0 & 1 \end{bmatrix} \boldsymbol{\Phi}_{\gamma}(s)^{\mathrm{T}} ds
\end{equation}
where we have introduced $\boldsymbol{\Phi}_{\gamma}(t) = \exp\left(t \boldsymbol{K}  \right)$.  


Given $\Delta t>0$ and $(q^0, p^0) \in \mathbb{R}^2$, the one step update of the Cayley splitting in \eqref{eq:cayley_splitting} can be written as  \begin{equation} \label{eq:cayley_splitting_langevin_1D}
\begin{bmatrix} q^1 \\ p^1 \end{bmatrix}  = \boldsymbol{O} \boldsymbol{C}  \boldsymbol{O} \begin{bmatrix} q^0 \\ p^0 \end{bmatrix} + \sqrt{\beta^{-1}} \sqrt{1-e^{-\gamma \Delta t}} \left(  \boldsymbol{O} \boldsymbol{C} \begin{bmatrix} 0 \\  1 \end{bmatrix} \xi^0 +   \begin{bmatrix} 0 \\  1 \end{bmatrix} \eta^0 \right)
\end{equation} where $\xi^0, \eta^0$ are independent standard normal random variables, $\boldsymbol{C}$ is the matrix defined in \eqref{eq:cayley_splitting_1D}, and $\boldsymbol{O}$ is the matrix defined as \[
 \boldsymbol{O} = \exp\left( \boldsymbol{\Gamma} \frac{\Delta t}{2} \right) \;, \quad  \boldsymbol{\Gamma} = \begin{bmatrix} 0 & 0 \\ 0 & -\gamma \end{bmatrix}  \;.
\]
Given an initial state $(q,p) \in \mathbb{R}^2$ and a number $m$ of integration steps, the numerical solution $(q^m, p^m)$ is a Gaussian vector with mean vector and covariance matrix given respectively by  \begin{equation} \label{eq:law_cayley_splitting_langevin_1D}
\boldsymbol{\mu}^m =  (\boldsymbol{O} \boldsymbol{C}  \boldsymbol{O})^m \begin{bmatrix} q \\ p \end{bmatrix} \;, \quad 
\boldsymbol{\Sigma}^m = \sum_{k=0}^m  (\boldsymbol{O} \boldsymbol{C}  \boldsymbol{O})^k \boldsymbol{Q}  (\boldsymbol{O} \boldsymbol{C}^{\mathrm{T}}  \boldsymbol{O})^k \;,
\end{equation}
where  \[
 \boldsymbol{Q} =\beta^{-1} (  1-e^{-\gamma \Delta t} )  \left( \boldsymbol{O} \boldsymbol{C}  \begin{bmatrix} 0 & 0 \\ 0 & 1 \end{bmatrix} \boldsymbol{C}^{\mathrm{T}} \boldsymbol{O} +   \begin{bmatrix} 0 & 0 \\ 0 & 1 \end{bmatrix}  \right) \;.
\]

\medskip
The following stability requirement is a bit stronger than in the Hamiltonian case.

\begin{lemma} \label{lemma:Cayley_Splitting_Stability_langevin_1D}
For all $\omega>0,\gamma>0$, and for any positive $\Delta t< 2$ satisfying \begin{equation} \label{eq:gam_stability}
g(\Delta t, \omega, \gamma) := 1 - \cos^2(\theta) \cosh^2\left(\frac{\gamma \Delta t}{2}\right) >0
\end{equation}
the matrix $\boldsymbol{O} \boldsymbol{C}  \boldsymbol{O}$ is asymptotically stable, and
for any $m \in \mathbb{N}$, \[
\| (\boldsymbol{O} \boldsymbol{C}  \boldsymbol{O})^m \| \le \frac{2 e^{- \frac{\gamma \Delta t (m-1)}{2} }}{\sqrt{g(\Delta t, \omega, \gamma)  }}  \left( \| \boldsymbol{C} \| + \sqrt{2} e^{-\frac{\gamma \Delta t}{2}} \right) \;.
\]
\end{lemma}

Note that $\theta$ in this lemma is the one we defined in \eqref{eq:theta_chi}.


\begin{proof}
Since $\Delta t <2$, Lemma~\ref{lemma:Cayley_Splitting_Stability_1D}  implies that \[
\boldsymbol{O} \boldsymbol{C}  \boldsymbol{O} = e^{-\frac{\gamma \Delta t}{2}} \begin{bmatrix} \cos(\theta) e^{\frac{\gamma \Delta t}{2}}  & \chi \sin(\theta) \\ \chi^{-1} \sin(\theta) & \cos(\theta) e^{-\frac{\gamma \Delta t}{2}} \end{bmatrix}
\]  where $\theta$ and $\chi$ are defined in \eqref{eq:theta_chi}.   Since \[
\det(\boldsymbol{O} \boldsymbol{C}  \boldsymbol{O}) = e^{-\gamma \Delta t} \;, \quad   \tr(\boldsymbol{O} \boldsymbol{C}  \boldsymbol{O}) = 2  e^{-\frac{\gamma \Delta t}{2}} \cos(\theta) \cosh(\frac{\gamma \Delta t}{2})  \;,
\] the function $g$ defined in \eqref{eq:gam_stability} satisfies \begin{equation} \label{eq:g_discriminant}
4 e^{-\gamma \Delta t} g(\Delta t, \omega, \gamma) = 4 \det(\boldsymbol{O} \boldsymbol{C}  \boldsymbol{O}) - \tr(\boldsymbol{O} \boldsymbol{C}  \boldsymbol{O})^2 \;,
\end{equation} and hence, our hypothesis on $\Delta t$ implies that $4 \det(\boldsymbol{O} \boldsymbol{C}  \boldsymbol{O}) > \tr(\boldsymbol{O} \boldsymbol{C}  \boldsymbol{O})^2$.  Thus, under this hypothesis, $\boldsymbol{O} \boldsymbol{C}  \boldsymbol{O}$ has a complex conjugate pair of eigenvalues: \[
\lambda_{\pm} =  \frac{\tr(\boldsymbol{O} \boldsymbol{C}  \boldsymbol{O}) \pm i \sqrt{ 4 \det(\boldsymbol{O} \boldsymbol{C}  \boldsymbol{O}) - \tr(\boldsymbol{O} \boldsymbol{C}  \boldsymbol{O})^2}}{2}
\] with $|\lambda_{\pm}| = e^{-\frac{\gamma \Delta t}{2}}$.  Hence, the matrix $\boldsymbol{O} \boldsymbol{C}  \boldsymbol{O}$ is asymptotically stable.

By the Cayley-Hamilton Theorem for $2 \times 2$ matrices, one can write \[
(\boldsymbol{O} \boldsymbol{C}  \boldsymbol{O})^m = \lambda_+^m \frac{1}{\lambda_+ - \lambda_-} ( \boldsymbol{O} \boldsymbol{C}  \boldsymbol{O} - \lambda_- \boldsymbol{I} ) 
+ \lambda_-^m \frac{1}{\lambda_- - \lambda_+} ( \boldsymbol{O} \boldsymbol{C}  \boldsymbol{O} - \lambda_+ \boldsymbol{I} ) 
\] and by the triangle inequality, \begin{align*}
\| (\boldsymbol{O} \boldsymbol{C}  \boldsymbol{O})^m \| &\le \frac{ |\lambda_+|^m + |\lambda_-|^m}{|\lambda_+ - \lambda_-|} \left(  \|\boldsymbol{O} \boldsymbol{C}  \boldsymbol{O}\| +  \sqrt{2} (  |\lambda_+| +  |\lambda_-| ) \right) \\
&\le  \frac{ 4 e^{-\frac{\gamma \Delta t}{2} m}}{ \sqrt{ 4 \det(\boldsymbol{O} \boldsymbol{C}  \boldsymbol{O}) - \tr(\boldsymbol{O} \boldsymbol{C}  \boldsymbol{O})^2}} \left( \| \boldsymbol{C} \| + \sqrt{2} e^{-\frac{\gamma \Delta t}{2}} \right) 
\end{align*}
from which the bound follows by substituting \eqref{eq:g_discriminant}.
\end{proof}

\medskip
The following lemma shows that the Cayley splitting applied to the one dimensional linear Langevin equation is second-order accurate in a weak or distributional sense.

\begin{lemma} \label{lemma:Cayley_Splitting_global_error_langevin_1D}
For all $\omega >0$, $\gamma \ge 0$ and $\beta>0$, and for any positive $\Delta t<\min(\sqrt{3}, 1/ \| \boldsymbol{K} \|)$ satisfying \eqref{eq:gam_stability}, and for all $T>0$, 
there exist positive constants $C_1, C_2$ such that 
\begin{align*}
\left\| (\boldsymbol{O} \boldsymbol{C} \boldsymbol{O})^m  -\boldsymbol{\Phi}_{\gamma}(m \Delta t) \right\|  \le&  \exp\left(  C_1  \| \boldsymbol{K} \| T  \right)  (1+ \| \boldsymbol{K} \|^3 ) \Delta t^2 \\
 \left\| \boldsymbol{\Sigma}^m  - \boldsymbol{\Sigma}(m \Delta t) \right\| \le&  \beta^{-1} \exp\left(  C_2  \|  \boldsymbol{K} \| T \right)  (1+ \| \boldsymbol{K} \|^2 ) \Delta t^2 
 \end{align*}
where $m = \lfloor T/\Delta t \rfloor$.  
\end{lemma}

Note that unlike Lemma~\ref{lemma:Cayley_Splitting_global_error_1D} the error constants here depend exponentially on the time interval of simulation.  This type of dependence is more typical of estimates for the global error of numerical solutions of ODEs and SDEs.   By definition of $\boldsymbol{K}$ in \eqref{eq:langevin_equations_1D}, note also that for fixed $\gamma>0$ \begin{equation} \label{eq:K_scaling}
 \|  \boldsymbol{K} \|  \lesssim (1+\omega^2) \;,  \quad \|  \boldsymbol{K} \|^2  \lesssim (1+\omega^2) \;, \quad \|  \boldsymbol{K} \|^3  \lesssim (1+\omega^4) \;.
\end{equation}
The following proof is a bit more involved than the proof of Lemma~\ref{lemma:Cayley_Splitting_global_error_1D}, because when $\gamma > 0$ we can no longer leverage the modified Hamiltonian structure of the numerical solution. 

\begin{proof}


We first prove the following preliminary estimates. \begin{align} 
& \| \boldsymbol{O} \boldsymbol{C}  \boldsymbol{O} - \boldsymbol{\Phi}_{\gamma}(\Delta t) \| \lesssim  \Delta t^3 \| \boldsymbol{K} \|^3   \;, \quad  \Delta t <1/ \| \boldsymbol{K} \|  \label{eq:local_error}  \\
& \| \boldsymbol{\Phi}_{\gamma}(t) \| \le \sqrt{2} \exp( \| \boldsymbol{K} \| t ) \;, \quad  t \ge 0   \label{eq:global_bound_on_Phi}  
\end{align}
Both global error estimates in Lemma~\ref{lemma:Cayley_Splitting_global_error_langevin_1D} rely on these bounds.  The latter inequality come from writing \begin{align*}
\boldsymbol{\Phi}_{\gamma}(t)  &= \boldsymbol{I} +\boldsymbol{K} \int_0^t \boldsymbol{\Phi}_{\gamma}(s) ds 
\end{align*} then taking norms to obtain,\begin{align*}
\| \boldsymbol{\Phi}_{\gamma}(t) \|  &\le \sqrt{2}  +\| \boldsymbol{K} \| \int_0^t \| \boldsymbol{\Phi}_{\gamma}(s) \|  ds 
\end{align*}
and then applying Gronwall's Lemma.   

To prove the local error estimate in \eqref{eq:local_error}, we use Taylor's theorem with a third order remainder term to write \begin{equation} \label{eq:taylor_formula_phi}
\boldsymbol{\Phi}_{\gamma}(\Delta t) = \boldsymbol{I} + \Delta t \boldsymbol{K} + \frac{\Delta t^2}{2} \boldsymbol{K}^2 + \boldsymbol{R} \;, \quad \| \boldsymbol{R} \| \le \frac{\Delta t^3}{6} \| \boldsymbol{K} \|^3 \exp( \Delta t \| \boldsymbol{K} \|) \;.
\end{equation}
On the other hand, since $\boldsymbol{\Gamma}$ and $\boldsymbol{B}$ commute, \[
\boldsymbol{O} \boldsymbol{C}  \boldsymbol{O}  = \exp\left( (1/2) \Delta t ( \boldsymbol{\Gamma} +  \boldsymbol{B}) \right) \cay( \Delta t \boldsymbol{A})  \exp\left( (1/2) \Delta t ( \boldsymbol{\Gamma} +  \boldsymbol{B}) \right) \;,
\] and by Lemma~\ref{lemma:cayley_series}, \[
\cay( \Delta t \boldsymbol{A})  = \boldsymbol{I} + \Delta t \boldsymbol{A} + \frac{\Delta t^2}{2} \boldsymbol{A}^2 + \boldsymbol{R}^{cay}  \;, \quad \| \boldsymbol{R}^{cay} \| \le  \frac{\Delta t^3}{2} \| \boldsymbol{A} \|^{3} \frac{1}{2 - \Delta t \| \boldsymbol{A}  \|}
\]
Since $\| \boldsymbol{A} \| \le \| \boldsymbol{K} \|$, and using our hypothesis on $\Delta t$, we obtain
\begin{equation}  \label{eq:taylor_formula_oco}
\boldsymbol{O} \boldsymbol{C}  \boldsymbol{O} = \boldsymbol{I} + \Delta t \boldsymbol{K} + \frac{\Delta t^2}{2} \boldsymbol{K}^2 + \boldsymbol{R}^{oco}  \;, \quad \| \boldsymbol{R}^{oco} \| \lesssim \Delta t^3 \| \boldsymbol{K} \|^3 
\end{equation}
The explicit bounds on the remainders in \eqref{eq:taylor_formula_phi} and \eqref{eq:taylor_formula_oco} show that \eqref{eq:local_error} holds. 


To bound the first global error, we define \begin{equation} \label{eq:first_global_error}
\epsilon_{k} = \left\| (\boldsymbol{O} \boldsymbol{C}  \boldsymbol{O} )^k -\boldsymbol{\Phi}_{\gamma}( k \Delta t) \right\| \;, \quad k \in \mathbb{N} \;.
\end{equation} Then \begin{align*}
\epsilon_{m} &\le  \left\| \boldsymbol{\Phi}_{\gamma}(\Delta t) \left( (\boldsymbol{O} \boldsymbol{C}  \boldsymbol{O} )^{m-1} -\boldsymbol{\Phi}_{\gamma}( (m-1) \Delta t) \right) \right\| \\
& \quad \quad +  \left\|  (\boldsymbol{O} \boldsymbol{C}  \boldsymbol{O} - \boldsymbol{\Phi}_{\gamma}(\Delta t) ) (\boldsymbol{O} \boldsymbol{C}  \boldsymbol{O} )^{m-1} \right\| \\
&\le (1+ 2 \Delta t \| \boldsymbol{K} \|  ) \epsilon_{m-1}  +  C_3 \Delta t^3  \|  \boldsymbol{K} \|^3
\end{align*}
where $C_3$ is a positive constant, and in the last step we used \eqref{eq:local_error}, \eqref{eq:global_bound_on_Phi} and Lemma~\ref{lemma:Cayley_Splitting_Stability_langevin_1D}.  Unraveling this linear recurrence inequality gives our estimate on $\epsilon_m$.


To bound the second global error, note that a second-order accurate trapezoidal discretization of the time integral in \eqref{eq:law_langevin_1D} yields \begin{equation} \label{eq:trapezoidal_covariance}
\boldsymbol{\Sigma}(m \Delta t) \approx \boldsymbol{\tilde \Sigma}^m: = \sum_{k=0}^m  \boldsymbol{\Phi}_{\gamma}(k \Delta t)  \boldsymbol{\tilde Q}   \boldsymbol{\Phi}_{\gamma}(k \Delta t)^{\mathrm{T}}
\end{equation} where  \begin{equation} \label{eq:tilde_Q}
 \boldsymbol{\tilde Q} =\Delta t \gamma \beta^{-1}  \left(  \boldsymbol{\Phi}_{\gamma}(\Delta t) \begin{bmatrix} 0 & 0 \\ 0 & 1 \end{bmatrix} \boldsymbol{\Phi}_{\gamma}(\Delta t)^{\mathrm{T}} +   \begin{bmatrix} 0 & 0 \\ 0 & 1 \end{bmatrix}  \right) \;.
\end{equation} In particular, it is well known that \begin{align*}
\| \boldsymbol{\Sigma}(m \Delta t ) -  \boldsymbol{\tilde \Sigma}^m \| &\lesssim  \gamma \beta^{-1} \Delta t^2  \sup_{s \in [0, m \Delta t]} \left\|\frac{d^2}{ds^2} \boldsymbol{\Phi}_{\gamma}(s) \begin{bmatrix} 0 & 0 \\ 0 & 1 \end{bmatrix}  \boldsymbol{\Phi}_{\gamma}(s)^{\mathrm{T}} \right\| \\
&\lesssim \gamma \beta^{-1}  \Delta t^2 \| \boldsymbol{K} \|^2 \exp(\| \boldsymbol{K} \| T)\;,
\end{align*}
where in the last line we used \eqref{eq:global_bound_on_Phi}.  By the triangle inequality, it suffices to bound $\epsilon_{m}^{\cov} = \left\|  \boldsymbol{\Sigma}^m - \boldsymbol{\tilde \Sigma}^m \right\|$.  By \eqref{eq:trapezoidal_covariance}, we obtain  \[
 \boldsymbol{\tilde \Sigma}^{m}  -  \boldsymbol{\tilde \Sigma}^{m-1} =  \boldsymbol{\Phi}_{\gamma}(m \Delta t) \boldsymbol{\tilde Q}  \boldsymbol{\Phi}_{\gamma}(m \Delta t)^{\mathrm{T}} \;.
 \]  Similarly, by \eqref{eq:law_cayley_splitting_langevin_1D}, we obtain \[
 \boldsymbol{\Sigma}^{m}  -  \boldsymbol{\Sigma}^{m-1} = ( \boldsymbol{O} \boldsymbol{C}  \boldsymbol{O} )^m \boldsymbol{Q} ( \boldsymbol{O} \boldsymbol{C}^{\mathrm{T}} \boldsymbol{O} )^m  \;.
 \]
 Taking the difference between these equations yields,
 \begin{align*}
\epsilon_{m}^{\cov} & \le \epsilon_{m-1}^{\cov} +  \left\|   \boldsymbol{\Phi}_{\gamma}(m \Delta t) \boldsymbol{\tilde Q}  \boldsymbol{\Phi}_{\gamma}(m \Delta t)^{\mathrm{T}}  - ( \boldsymbol{O} \boldsymbol{C}  \boldsymbol{O} )^m \boldsymbol{Q} ( \boldsymbol{O} \boldsymbol{C}^{\mathrm{T}}  \boldsymbol{O} )^m \right\| \\
&\le \epsilon_{m-1}^{\cov} +  \left\|   ( \boldsymbol{\Phi}_{\gamma}(m \Delta t)  - ( \boldsymbol{O} \boldsymbol{C}  \boldsymbol{O} )^m ) \boldsymbol{\tilde Q} ( \boldsymbol{O} \boldsymbol{C}^{\mathrm{T}}  \boldsymbol{O} )^m \right\| \\
& \quad \quad  +  \left\|   \boldsymbol{\Phi}_{\gamma}(m \Delta t) \boldsymbol{\tilde Q}  ( \boldsymbol{\Phi}_{\gamma}(m \Delta t)^{\mathrm{T}}  - ( \boldsymbol{O} \boldsymbol{C}^{\mathrm{T}}  \boldsymbol{O} )^m )\right\| \\
& \quad \quad  +  \left\|    ( \boldsymbol{O} \boldsymbol{C}  \boldsymbol{O} )^m ( \boldsymbol{\tilde Q}  - \boldsymbol{Q} ) ( \boldsymbol{O} \boldsymbol{C}^{\mathrm{T}}  \boldsymbol{O} )^m \right\|  \\
&\le \epsilon_{m-1}^{\cov} + \left( \left\|   \boldsymbol{\Phi}_{\gamma}(m \Delta t) \boldsymbol{\tilde Q}  \right\| +  \left\| \boldsymbol{\tilde Q} ( \boldsymbol{O} \boldsymbol{C}^{\mathrm{T}}  \boldsymbol{O})^m \right\| \right) \epsilon_m + C_4 \Delta t^3  \|  \boldsymbol{K} \|^2
\end{align*}
where $C_4$ is a positive constant, $\epsilon_m$ is the first global error defined in \eqref{eq:first_global_error}, and in the last step we used \eqref{eq:global_bound_on_Phi} and \[
\left\| \boldsymbol{\tilde Q} - \boldsymbol{Q} \right\| \lesssim \Delta t^3 \| \boldsymbol{K} \|^2  \;.
\]
By Lemma~\ref{lemma:Cayley_Splitting_Stability_langevin_1D}, \eqref{eq:global_bound_on_Phi}, our previous estimate on $\epsilon_m$, and noting that the matrix $\boldsymbol{\tilde Q}$ defined in \eqref{eq:tilde_Q} is  $O(\Delta t)$, one can unravel this linear recurrence inequality to obtain our estimate on $\epsilon_{m}^{\cov}$.
\end{proof}

\begin{figure}
\begin{center}
\includegraphics[width=0.45\textwidth]{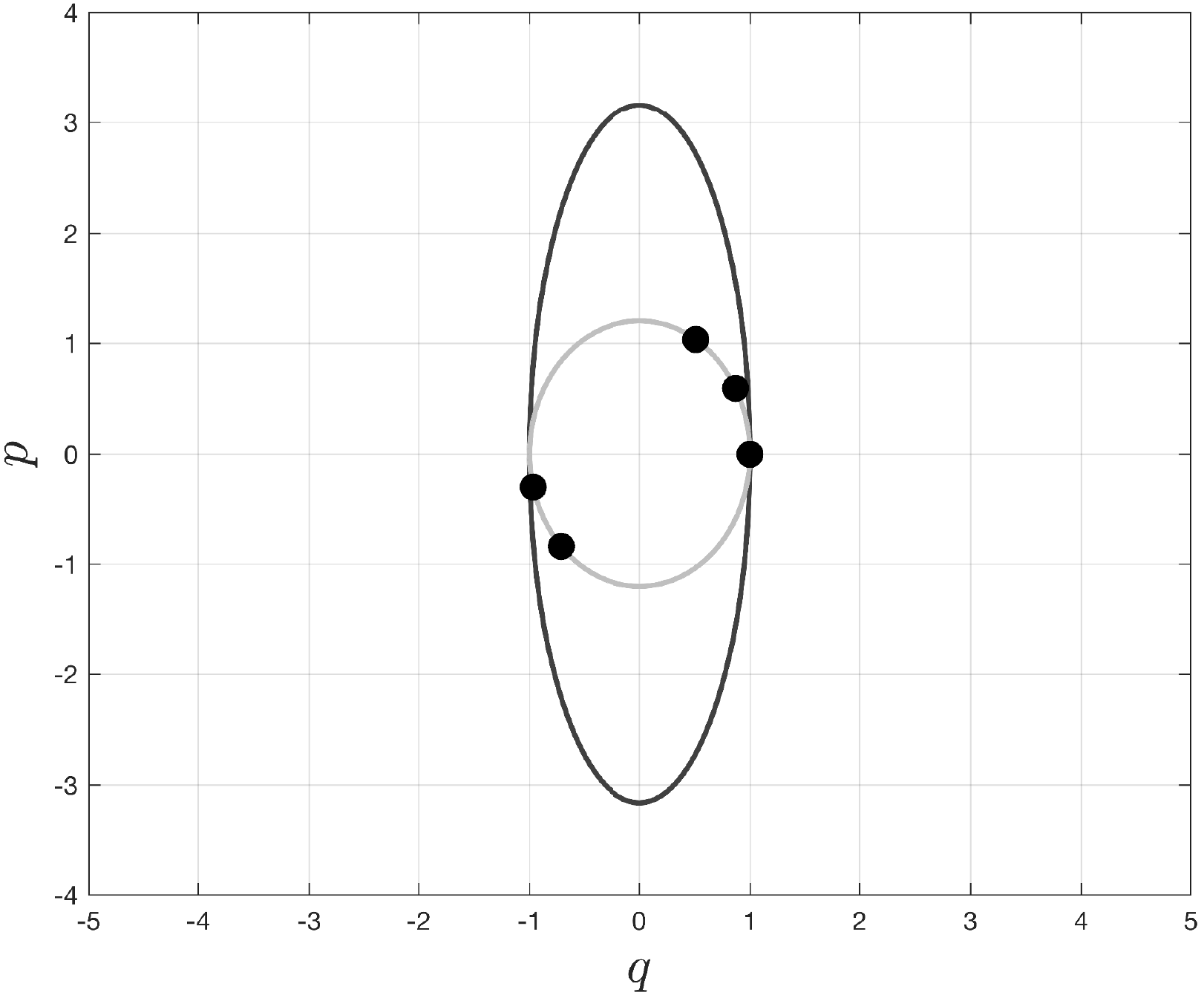} \hspace{0.1in}
\includegraphics[width=0.45\textwidth]{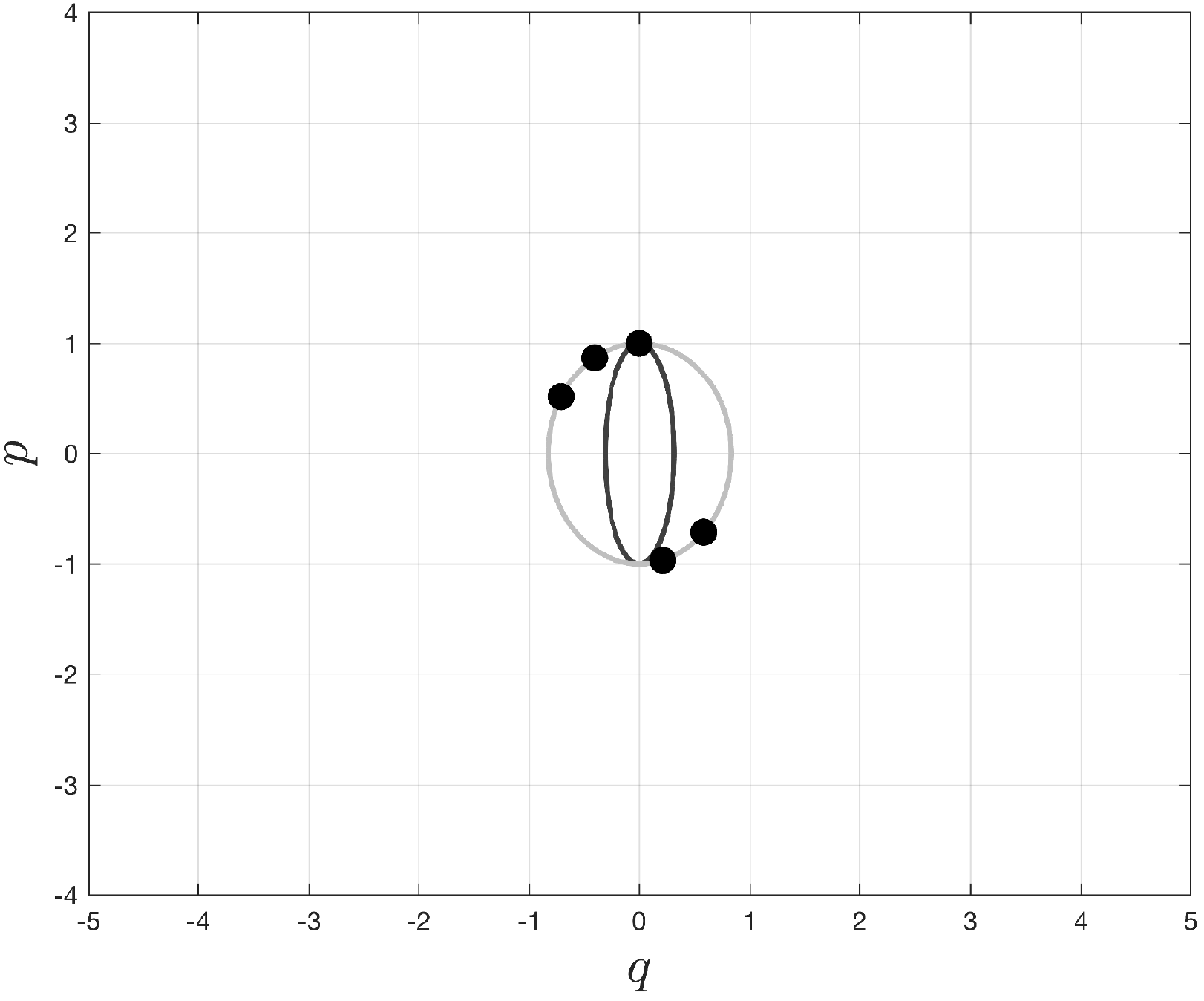}
\end{center}
\caption{\small  {\bf Modified Hamiltonian of Cayley Splitting.}  
This figure illustrates two level sets of the modified Hamiltonian for the Cayley splitting with $\omega = 3$ and $\Delta t=1.85$. The dots are the first five points along a discrete orbit of the Cayley splitting  initiated at $(q, p) = (1, 0)$ (left panel) and $(q,p) = (0,1)$ (right panel) and the grey lines provide the associated level set of the modified Hamiltonian. 
For comparison, the black contour lines show the level sets of the exact Hamiltonian that contain the initial conditions of the Cayley splitting.
}
  \label{fig:modified_hamiltonian_cayley}
\end{figure}


\subsection{Linear Hamiltonian PDE}

Consider an infinite-dimensional, linear Hamiltonian system \begin{equation} \label{eq:linear_hamiltonian_pde}
\partial_t u(t,s) = p(t,s) , ~ \partial_t p(t,s) = \partial^2_s u(t,s) -  u(t,s) , ~~ (t,s) \in [0, \infty) \times [0, S],
\end{equation}
with homogeneous Dirichlet boundary conditions \begin{equation}
u(t,0) = 0 \;, \quad  u(t,S) = 0 \;, \quad  t>0 \;,
\end{equation}
and initial conditions \begin{equation} \label{eq:initial_condition_linear_hamiltonian_system}
u(0,s) \overset{d}{=} \sum_{k \ge 1} \eta_k  \sqrt{ \frac{1}{1-\mu_k} } e_k(s) \;, \quad  p(0,s) \overset{d}{=} \sum_{k \ge 1} \xi_k e_k(s) \;, \quad  s \in [0,S] \;,
\end{equation}
where  $\{ \eta_k \}$ and $\{ \xi_k \}$ are i.i.d.~standard normal random variables, $\{ e_k \}$ are eigenfunctions of the operator $\partial^2_s$ endowed with Dirichlet boundary conditions, and $\{ \mu_k \}$ are their associated eigenvalues.  For any $k \in \mathbb{N}$, these eigenfunctions and eigenvalues are given explicitly by \[
e_k(s) = \sqrt{\frac{2}{S}} \sin\left( \frac{k \pi s}{S} \right) \;,  \quad \mu_k = - \frac{k^2 \pi^2}{S^2} \;,
\] with $\mu_{k} \ge \mu_{k+1}$.  The Hamiltonian functional associated to \eqref{eq:linear_hamiltonian_pde} is given by \begin{equation} \label{eq:hamiltonian_linear_hamiltonian_system}
\mathcal{H}( u, p ) =  \frac{1}{2} \int_0^S |p(s)|^2 ds + \frac{1}{2} \int_0^S |\partial_s u(s)|^2 ds + \frac{1}{2} \int_0^S |u(s)|^2 ds 
\end{equation}
Note that the initial position is a second-order Gaussian process, whereas the initial momentum is spatial Gaussian white noise whose mean-squared norm diverges.  
Moreover, the distribution of these initial conditions is preserved under the Hamiltonian dynamics.  Indeed, for all $t \ge 0$ (including $t=0$), the position component $u(t,s)$ is a 
second-order Gaussian process with zero mean and the same spatial covariance  \[
\E \{ u(t,s) u(t,s') \} = \sum_{k \ge 1} e_k(s) e_k(s') \frac{1}{1 - \mu_k} \;,   \quad  s,s' \in [0,S] \;.
\]
This series is easily seen to converge because the eigenvalues $\mu_k$ grow quadratically with $k$. 
Similarly, the momentum component $p(t,s)$ is a spatial Gaussian white noise for all $t \ge 0$.  In this section, we aim to approximate $u(t,s)$ in a strong sense.


\subsection{Semidiscrete Equations}

As shown in Figure~\ref{fig:grid}, given $n \in \mathbb{N}$ and $S>0$, we discretize the spatial domain $[0,S]$ using a uniform grid with $n+1$ grid points \begin{equation} \label{eq:uniform_grid}
\{ s_i = i \Delta s~ \mid~  0 \le i \le n \}  
\end{equation} where $\Delta s = S/n$ is the spatial step size.  We approximate the second derivative $\partial^2_s$ in \eqref{eq:linear_hamiltonian_pde} by using a central difference formula evaluated on this grid.  The Dirichlet boundary conditions imply that there are only $2 \times (n-1)$ unknown functions of time $(\boldsymbol{u}(t),\boldsymbol{p}(t)) \in \mathbb{R}^{2 (n-1)}$ which satisfy  \begin{equation} \label{eq:semidiscrete_linear_hamiltonian}
\dot{\boldsymbol{u}}(t) = \boldsymbol{p}(t) \;, \quad \dot{\boldsymbol{p}}(t) = \boldsymbol{L} \boldsymbol{u}(t) -  \boldsymbol{u}(t) \;,
\end{equation} 
where $ \boldsymbol{L}$ is the $(n-1) \times (n-1)$ discrete Laplacian matrix with Dirichlet boundary conditions \[
\boldsymbol{L} = \frac{1}{\Delta s^2} \begin{bmatrix} 
-2 & 1 &  & \\
1 & \ddots & \ddots &    \\
& \ddots & \ddots &  1 \\
&  & 1 &  -2 \end{bmatrix} \;.
\] Since $\boldsymbol{L}$ is symmetric, \eqref{eq:semidiscrete_linear_hamiltonian} is Hamiltonian and its solutions preserve the Hamiltonian function \begin{equation}
\label{eq:hamiltonian_semidiscrete_linear_hamiltonian_system}
H( \boldsymbol{u}, \boldsymbol{p}) = \frac{1}{2} \| \boldsymbol{p} \|^2 - \frac{1}{2} \boldsymbol{u}^{\mathrm{T}} \boldsymbol{L} \boldsymbol{u} + \frac{1}{2} \| \boldsymbol{u} \|^2 \;.
\end{equation} To approximate the initial conditions, write $ \boldsymbol{L} = \boldsymbol{V} \boldsymbol{\Lambda} \boldsymbol{V}^{\mathrm{T}}$ where $ \boldsymbol{\Lambda}$ is a diagonal
matrix of eigenvalues of $\boldsymbol{L} $ and $\boldsymbol{V}$ are its corresponding orthonormal eigenvectors.
The semidiscrete analog of \eqref{eq:initial_condition_linear_hamiltonian_system} is given by \begin{equation} \label{eq:initial_condition_linear_hamiltonian_system_semidiscrete}
\boldsymbol{u}(0) \overset{d}{=} \frac{1}{\sqrt{\Delta s}} \boldsymbol{V} (-\boldsymbol{\Lambda} + \boldsymbol{I})^{-1/2} \boldsymbol{\eta}  \;, \quad  
\boldsymbol{p}(0) \overset{d}{=}  \frac{1}{\sqrt{\Delta s}} \boldsymbol{\xi}  
\end{equation}
where $\boldsymbol{\eta}$ and $\boldsymbol{\xi}$ are independent standard normal vectors.  The law of
these random initial conditions has non-normalized density $e^{-(\Delta s) H(\boldsymbol{u}, \boldsymbol{p})}$.  Note that the factor $\Delta s$
ensures that the sums in $(\Delta s) H(\boldsymbol{u}, \boldsymbol{p})$ converge to the integrals appearing in \eqref{eq:hamiltonian_linear_hamiltonian_system}
as $\Delta s \to 0$.


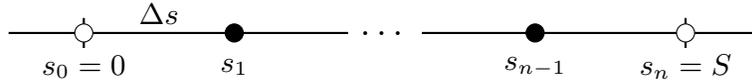
\begin{figure}
\begin{center}
\begin{tikzpicture}[scale=1.0,align=center]
\draw[-, thick](0,0.0) -- (10,0.0);
\draw[-, thick](1,-0.2) -- (1,0.2);
\draw[-, thick](9,-0.2) -- (9,0.2);
\node[black,scale=1.25] at (2,0.25) {$\Delta s$};
\node[black,scale=1.25] at (1,-0.45) {$s_0=0$};
\node[black,scale=1.25] at (9,-0.45) {$s_n=S$};
\node[black,scale=1.25] at (3,-0.45) {$s_1$};
\node[black,scale=1.25] at (7,-0.45) {$s_{n-1}$};
\node[black, scale=1.5,fill=white] at (5.0,0.0) {$\dotsc$};
\filldraw[color=black,fill=white] (1,0) circle (0.12);
\filldraw[color=black,fill=black] (3,0) circle (0.12);
\filldraw[color=black,fill=black] (7,0) circle (0.12);
\filldraw[color=black,fill=white] (9,0) circle (0.12);
\end{tikzpicture}
\end{center}
\caption{\small  {\bf Uniform Grid.}  This figure shows the evenly spaced grid we use to derive our spatial discretization.  There are a total of $n+1$ grid points.  The black dots mark the interior grid points, and the white dots mark the boundary grid points.  
}
  \label{fig:grid}
\end{figure}


\subsection{Strang Splittings}

We next apply a Strang splitting to discretize \eqref{eq:semidiscrete_linear_hamiltonian} in time.  To this end, define the $2 (n-1) \times 2 (n-1)$ Hamiltonian matrices \[
\boldsymbol{A} = \begin{bmatrix} \boldsymbol{0}  & \boldsymbol{I} \\ \boldsymbol{L} & \boldsymbol{0} \end{bmatrix}  \quad  \text{and} \quad 
\boldsymbol{B} = \begin{bmatrix} \boldsymbol{0}  & \boldsymbol{0} \\ -\boldsymbol{I} & \boldsymbol{0} \end{bmatrix} \;.
\]
and split \eqref{eq:semidiscrete_linear_hamiltonian} into \[
\tag{A}  \begin{bmatrix} \dot{\boldsymbol{u}}(t) \\  \dot{\boldsymbol{p}}(t) \end{bmatrix} = \boldsymbol{A} \begin{bmatrix} \boldsymbol{u}(t) \\ \boldsymbol{p}(t) \end{bmatrix}  \;, 
\]
\[
\tag{B}  \begin{bmatrix} \dot{\boldsymbol{u}}(t) \\  \dot{\boldsymbol{p}}(t) \end{bmatrix} = \boldsymbol{B} \begin{bmatrix} \boldsymbol{u}(t) \\ \boldsymbol{p}(t) \end{bmatrix}  \;.
\]
Define the exact splitting as the linear transformation with matrix  \begin{equation} \label{eq:exact_splitting_linear_hamiltonian}
\boldsymbol{E} =\exp( (1/2) \Delta t \boldsymbol{B} ) \exp( \Delta t \boldsymbol{A} ) \exp( (1/2) \Delta t \boldsymbol{B} ) \;,
\end{equation} and define the Cayley splitting as the linear transformation with matrix  \begin{equation} \label{eq:cayley_splitting_linear_hamiltonian}
\boldsymbol{C} = \exp( (1/2) \Delta t \boldsymbol{B} ) \cay( \Delta t \boldsymbol{A} ) \exp( (1/2) \Delta t \boldsymbol{B} ) \;,
\end{equation}
where $\cay$ is the Cayley transform defined in \eqref{eq:cayley}. 
The only difference between these splittings is that $\boldsymbol{E}$ uses the matrix exponential to compute the exact flow of (A), while
$\boldsymbol{C} $ uses an approximation of the matrix exponential given by the Cayley transform.  
Both of these matrices have determinant equal to $1$, since they are both symplectic matrices.


\subsection{Stability of Cayley Splitting \& Instability of Exact Splitting} \label{sec:stability}

Figure~\ref{fig:linear_hamiltonian_system_energy} illustrates the relative energy error as a function of time  along trajectories produced by the exact and Cayley splittings with initial condition given in \eqref{eq:initial_condition_linear_hamiltonian_system_semidiscrete}, domain size $S=10$, $n=10^3$ grid points (or spatial grid size $\Delta s = 10^{-2}$), and time step size $\Delta t = 0.2$.    This figure shows that the energy along the Cayley splitting stays roughly constant as one would expect from a geometric integrator. However, the energy along the exact splitting increases, which is a bit unexpected because one would think that the exact splitting is more accurate than the Cayley splitting.

The reason for this difference becomes obvious when we transform to spectral coordinates by using the eigenvectors of the matrix $\boldsymbol{L}$.  Recall that these eigenvectors are the restriction to the grid of the eigenfunctions of the operator $\partial_s^2$ endowed with Dirichlet boundary conditions.   In these coordinates the linear Hamiltonian system decouples into $n-1$ oscillators.    Since time discretization commutes with this change of variables, it suffices to consider exponential and Cayley splittings applied to the $i$th oscillator.  We define the natural frequency $\omega_i$ of the $i$th oscillator as the square root of negative the $i$th largest eigenvalue of $\boldsymbol{L}$, i.e.,  \begin{equation} \label{eq:omi2}
\omega_i^2 = \frac{4}{\Delta s^2} \sin^2 \left( \frac{i \pi}{2 n} \right) \;, \quad 1 \le i \le n-1 \;.
\end{equation}  The dynamics of the $i$th oscillator in the spectral domain is governed by \[
\begin{bmatrix} \dot U_i \\ \dot P_i \end{bmatrix} = \left( \boldsymbol{A}_i  + \boldsymbol{B}_i \right) \begin{bmatrix} U_i \\ P_i \end{bmatrix} 
\] where we have introduced the $2 \times 2$ matrices \[
\boldsymbol{A}_i = \begin{bmatrix} 0  &1\\ -\omega_i^2 & 0 \end{bmatrix}  \quad  \text{and} \quad 
\boldsymbol{B}_i = \begin{bmatrix} 0  & 0 \\ -1 & 0 \end{bmatrix} \;.
\]
For this oscillator, we consider an exact splitting  \[
\boldsymbol{E}_i =\exp( (1/2) \Delta t \boldsymbol{B}_i ) \exp( \Delta t \boldsymbol{A}_i ) \exp( (1/2) \Delta t \boldsymbol{B}_i ) \;,
\] and a Cayley splitting \[
\boldsymbol{C}_i = \exp( (1/2) \Delta t \boldsymbol{B}_i ) \cay( \Delta t \boldsymbol{A}_i ) \exp( (1/2) \Delta t \boldsymbol{B}_i ) \;.
\]

\medskip
For the Cayley splitting, Lemma~\ref{lemma:Cayley_Splitting_Stability_1D} implies the following holds.

\begin{prop} \label{prop:Cayley_Splitting_Stability}
For all $S>0$, for all $\Delta s>0$, and for any positive $\Delta t<2$, the Cayley splitting in \eqref{eq:cayley_splitting_linear_hamiltonian} is a stable numerical method for the infinite-dimensional, linear Hamiltonian system \eqref{eq:linear_hamiltonian_pde}.
\end{prop}

However, for the exact splitting, we have \[
\boldsymbol{E}_i = \begin{bmatrix}  \cos(\Delta t \omega_i) - \frac{\Delta t}{2 \omega_i} \sin(\Delta t \omega_i)  &  \frac{1}{\omega_i} \sin(\Delta t \omega_i) \\
- \Delta t \cos(\Delta t \omega_i) + \frac{-4 \omega_i^2 + \Delta t^2}{4 \omega_i} \sin(\Delta t \omega_i) &   \cos(\Delta t \omega_i) - \frac{\Delta t}{2 \omega_i} \sin(\Delta t \omega_i)   \end{bmatrix}
\] and in particular, if $\Delta t \cdot \omega_i$ is an odd integer multiple of $\pi$, we get \[
\boldsymbol{E}_i = \begin{bmatrix} - 1  &  0 \\
 \Delta t  &   -1  \end{bmatrix}
\] which implies that the exact splitting is weakly unstable; and similarly if $\Delta t \cdot \omega_i$ is an even integer multiple of $\pi$.   For the infinite energy solutions of interest, numerical experiments show that these resonance instabilities can be triggered if $\Delta t \cdot \omega_i$ is approximately an integer multiple of $\pi$.

Figure~\ref{fig:linear_hamiltonian_system_spectral_energy} shows plots of the energy in each of the $n-1$ oscillators in the spectral domain at four different snapshots in time, as indicated in the figure titles.  Note that we have taken $n=1000$ in Figure~\ref{fig:linear_hamiltonian_system_energy}, so that there are a total of $999$ oscillators. In this figure, we sort the oscillators in ascending order according to their natural frequency.  The energy in each oscillator for the exact splitting is shown in grey.  Note that the oscillator with maximum energy corresponds to the index $i=50$ with natural frequency approximately equal to $\sqrt{246}$, and hence, $\Delta t \cdot \omega_{50} = 0.2 \sqrt{246} = 3.13688... \approx \pi$.  Moreover, since the trace of $\boldsymbol{E}_{50}$ is greater than two in magnitude, the growth in the energy of this oscillator is exponential in time, as Figure~\ref{fig:linear_hamiltonian_system_spectral_energy} illustrates.    As illustrated in the right panel of Figure~\ref{fig:eigenvalues_cayley_exp}, this instability is not surprising in view of the fact that $\exp( \Delta t \boldsymbol{A}_i )$ is not a strongly stable symplectic map for all $i$.    Because of this limitation of the exact splitting, we will only consider the Cayley splitting in the rest of the paper.

\begin{remark} \label{rmk:multiple_time_step_integrators}
This resonance instability of the exact splitting is very reminiscent of what happens when we apply a multiple-time-step integrator to a high-dimensional Hamiltonian ODE, e.g., the popular RESPA algorithm from molecular dynamics \cite{LeMaOrWe2003,TuBeMa1992}.  RESPA is a variant of Verlet that relaxes a time-step restriction imposed by rapidly changing or fast potential forces. The basic idea in RESPA is to split the Hamiltonian into fast and slow parts, and then combine their flows using a palindromic splitting to ensure reversibility.  For the fast parts, one integrates with such a tiny time-step that essentially the exact flow is used, while one uses a large time step for the slow parts. A speedup is achieved if the slow force is computationally costly to evaluate compared to the fast forces.  Unfortunately, this speedup is not dramatic because of resonance instabilities \cite{FoDaLe2008}. This resonance occurs between the slow potential forces and the fast Hamiltonian systems. These numerical instabilities also appear in generalizations of RESPA to Langevin SDEs, though they are not as severe  \cite{IzCaWoSk2001, MaIzSk2003, Fo2009}.    The proof of Proposition~\ref{prop:Cayley_Splitting_Stability} immediately implies that we can resolve linear resonance instabilities by using a Cayley integrator for the fast forces, instead of the exact flow, as Figure~\ref{fig:cayley_respa} illustrates.  
\end{remark}


\begin{figure}
\begin{center}
\includegraphics[width=0.65\textwidth]{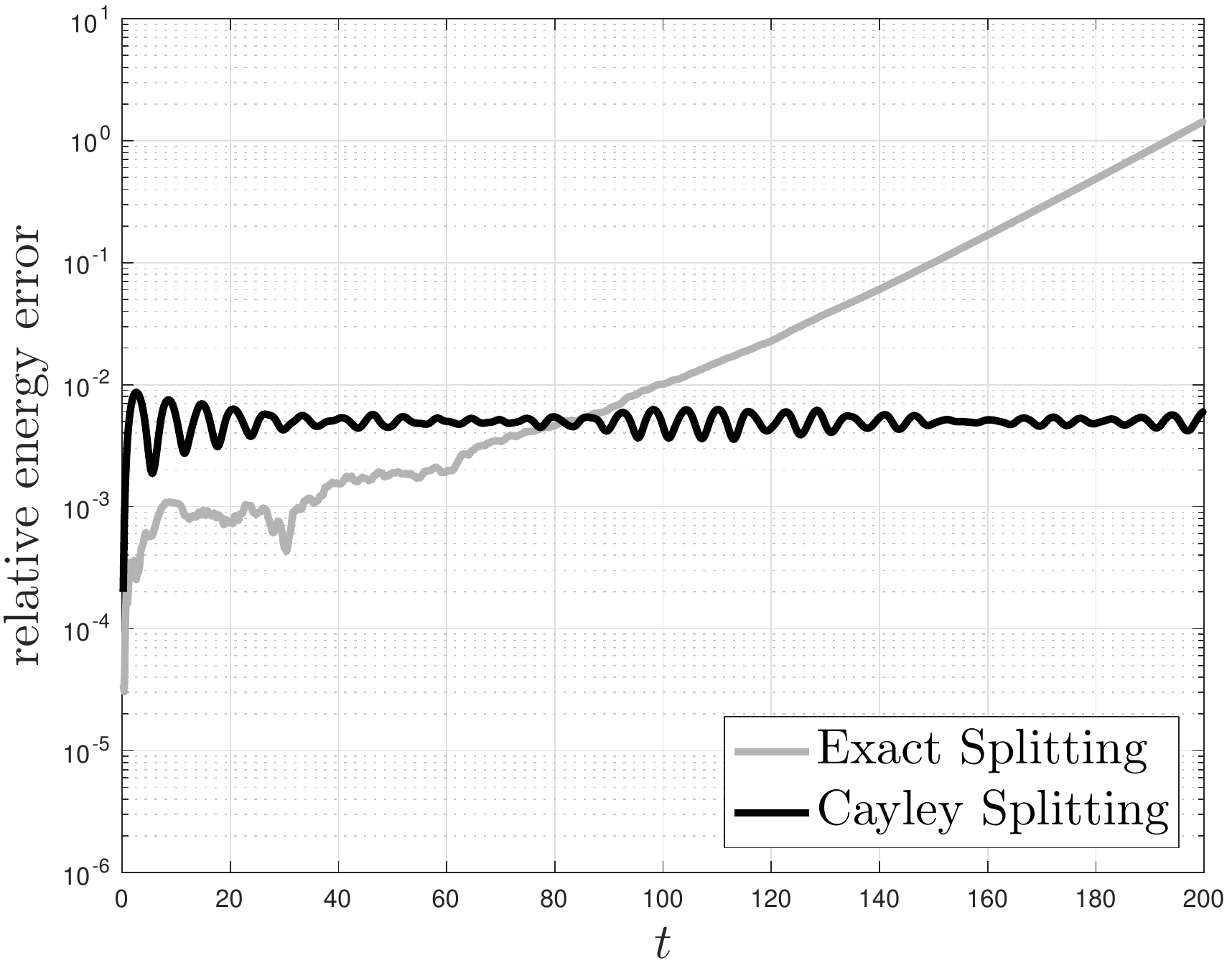}
\end{center}
\caption{\small  {\bf Stability of Cayley Splitting \& Instability of Exact Splitting.}  This figure shows the relative energy error along trajectories
produced by the exact and Cayley splittings given in \eqref{eq:exact_splitting_linear_hamiltonian} and \eqref{eq:cayley_splitting_linear_hamiltonian}, respectively.
We discretize the spatial domain $[0,10]$ using an evenly spaced grid with $n=10^3$ grid points, as shown in Figure~\ref{fig:grid}. 
The time step size is set equal to $\Delta t = 0.2$.   The initial condition is given in \eqref{eq:initial_condition_linear_hamiltonian_system}.  
}
  \label{fig:linear_hamiltonian_system_energy}
\end{figure}

\begin{figure}
\begin{center}
\includegraphics[width=0.45\textwidth]{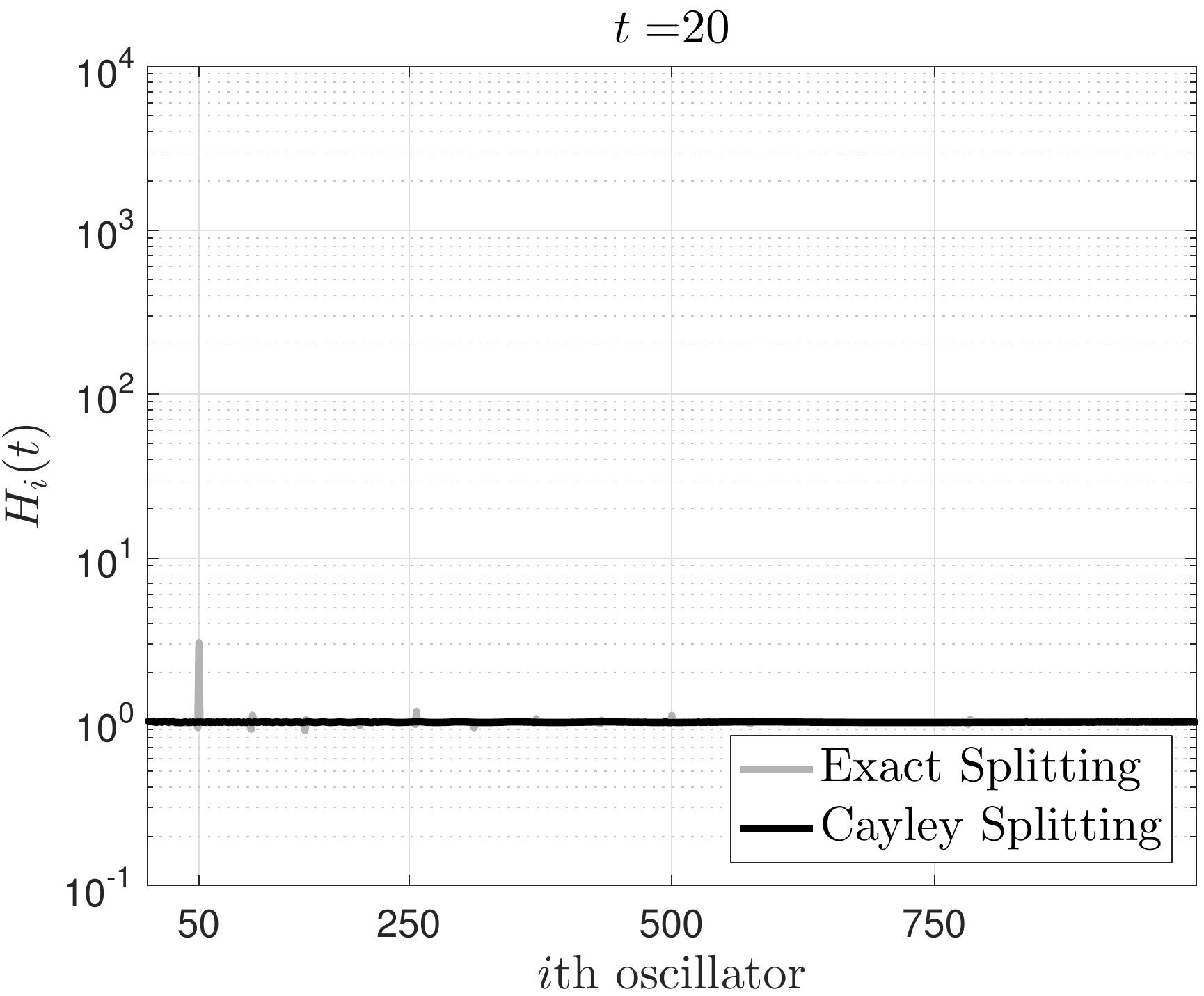}  \hspace{0.15in}
\includegraphics[width=0.45\textwidth]{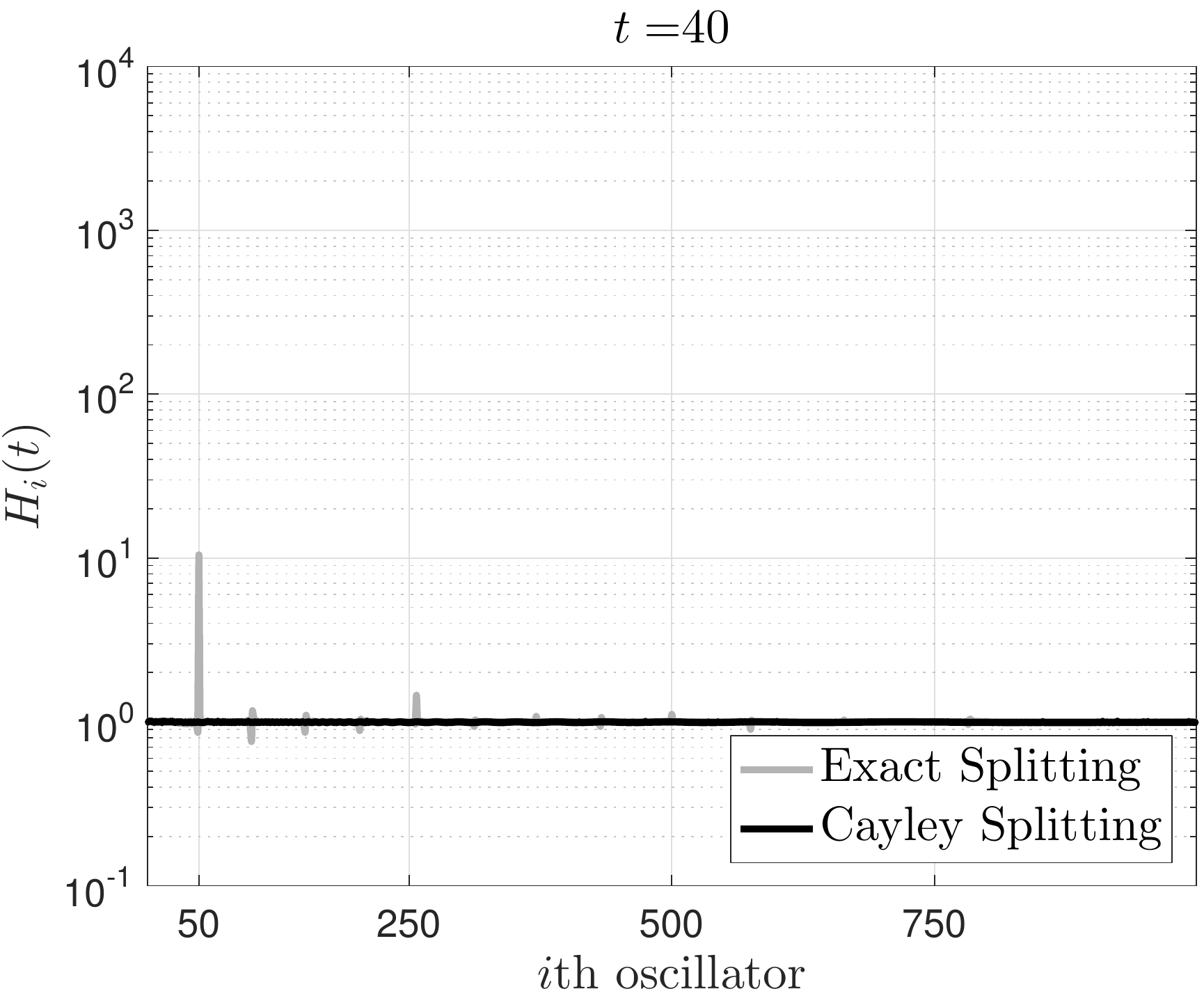}  \\
\vspace{0.15in}
\includegraphics[width=0.45\textwidth]{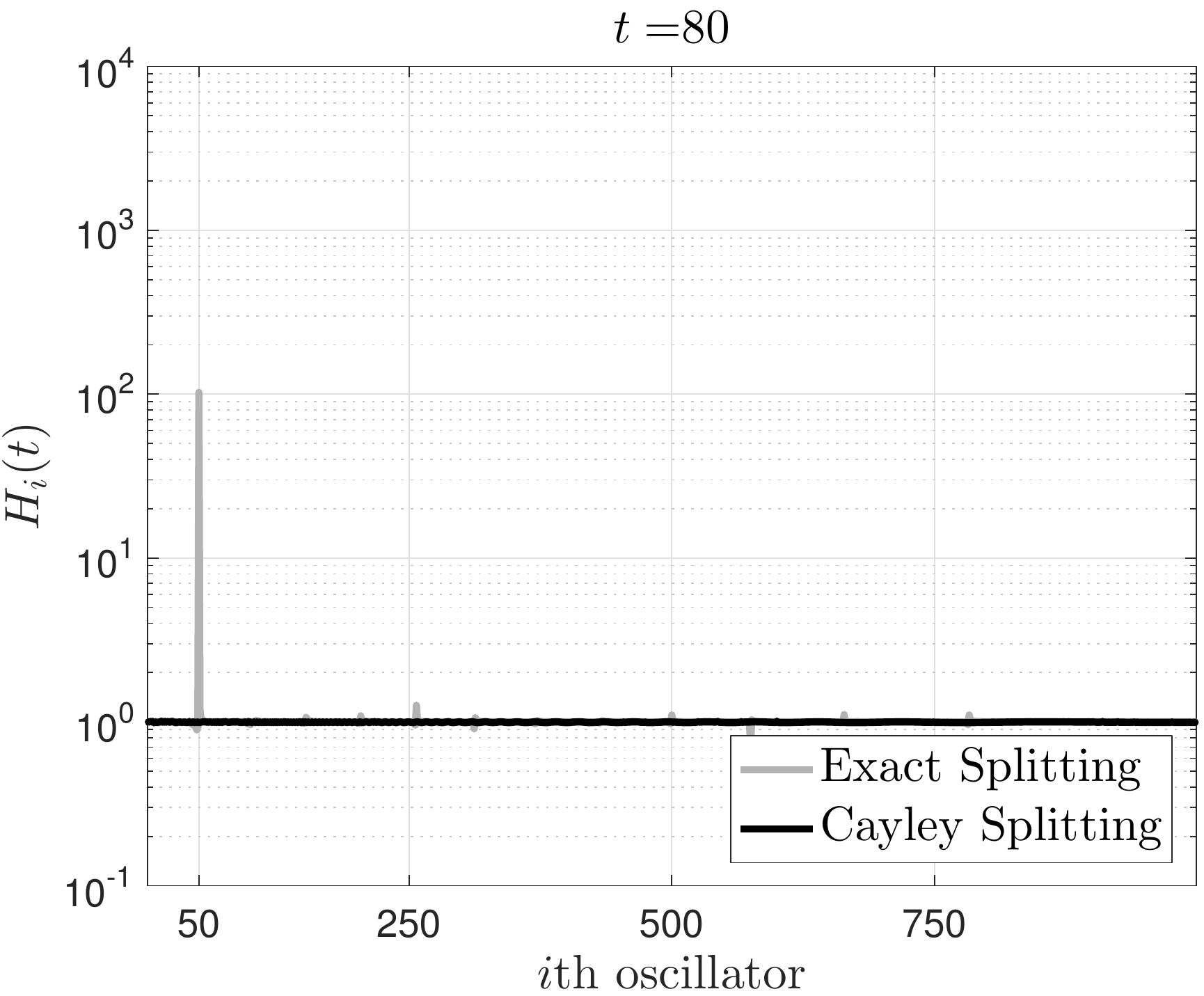}  \hspace{0.15in}
\includegraphics[width=0.45\textwidth]{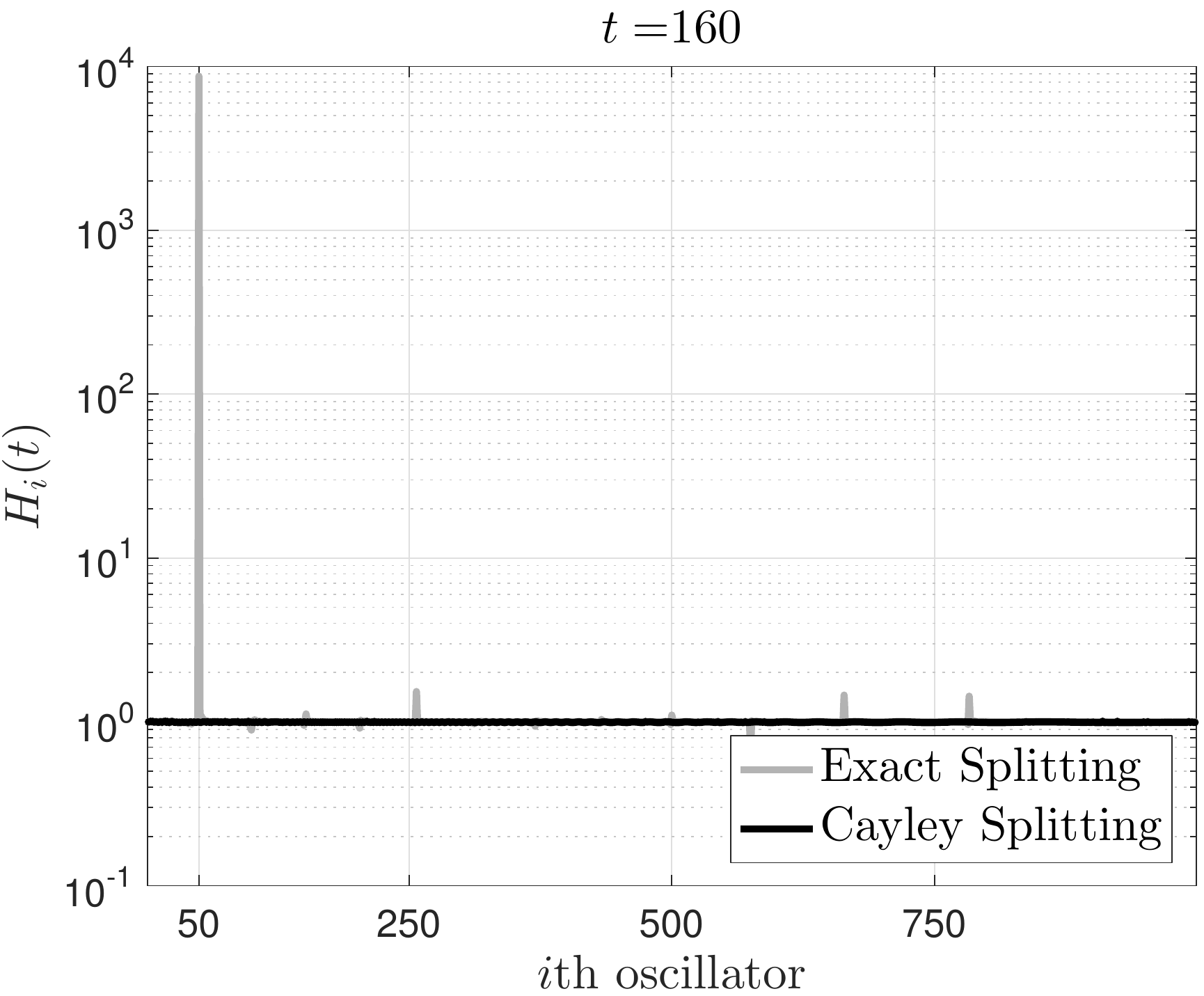} 
\end{center}
\caption{\small  {\bf Stability of Cayley Splitting \& Instability of Exact Splitting.}  The plots show the energy in each spectral coefficient at four different
snapshots in time (as indicated in the figure titles) for the Cayley splitting (black) and the exact splitting (grey).   
The oscillators are sorted by their natural frequencies in ascending order.  For the Cayley splitting, the normalized energy is essentially
constant for the duration of the simulation.  However, for the exact splitting, the oscillator with maximum (normalized) energy corresponds to $i=50$ 
as labelled on the horizontal axes, and its energy grows exponentially with time.
}
  \label{fig:linear_hamiltonian_system_spectral_energy}
\end{figure}

\begin{figure}
\begin{center}
\includegraphics[width=0.65\textwidth]{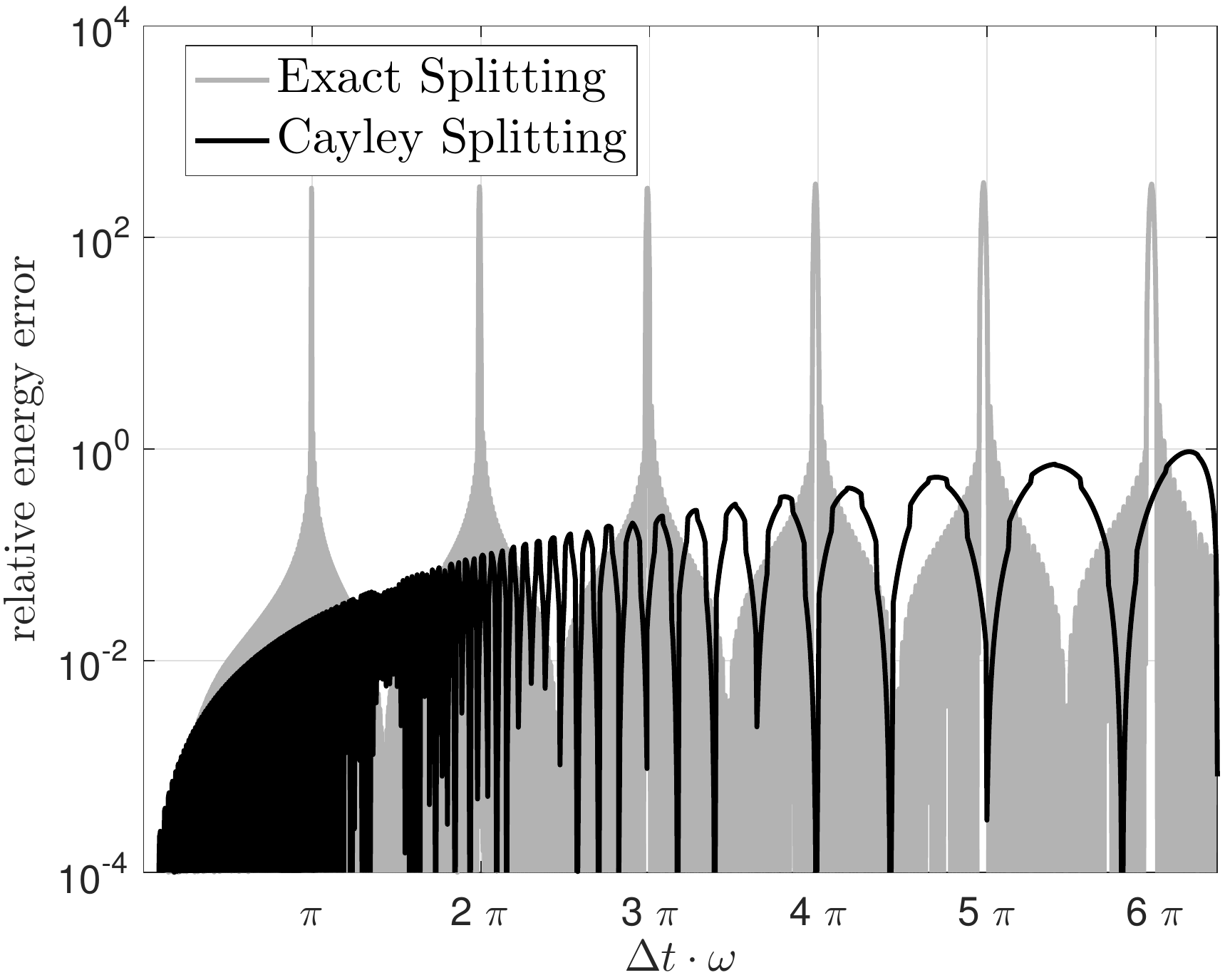}  
\end{center}
\caption{\small  {\bf Cayley Splitting for a Highly Oscillatory Hamiltonian ODE.}   
The plot shows the relative energy error as a function of $\Delta t \cdot \omega$  for the Cayley splitting (black) and the exact splitting (grey).  
The underlying Hamiltonian system is an oscillator with Hamiltonian $H(q,p) = (1/2) p^2 + (1/2) q^2 + (1/2)  \omega^2  q^2$.  We split this Hamiltonian into 
$H_{\text{fast}}(q,p) = (1/2) p^2 + (1/2) \omega^2 q^2$ and $H_{\text{slow}}(q,p) = (1/2) q^2$ with exact flows $\varphi_{t}^{\text{fast}}$ and $\varphi_{t}^{\text{slow}}$, respectively.   
We then combine these flows using an exact splitting $\varphi_{(1/2) \Delta t}^{\text{slow}} \circ \varphi_{\Delta t}^{\text{fast}} \circ \varphi_{(1/2) \Delta t}^{\text{slow}}$ 
or the Cayley splitting $\varphi_{(1/2) \Delta t}^{\text{slow}} \circ \phi_{\Delta t}^{\text{fast}} \circ \varphi_{(1/2) \Delta t}^{\text{slow}}$.   In this simulation we take $\omega=10$, the duration of the simulation to be $100$ times the period of the oscillator $2 \pi/\omega$, and the initial condition $(q,p) = (1,0)$.  This result is consistent with Lemma~\ref{lemma:Cayley_Splitting_Stability_1D}.
}
  \label{fig:cayley_respa}
\end{figure}


\subsection{Mean Energy Errors of Cayley Splitting} \label{sec:mean_dH}

In finite dimensions, it is known that the global energy error $\Delta$ under a volume-preserving, reversible integrator satisfies \[
0 \le \E ( \Delta ) \le \E (\Delta^2)
\] where the expected value is over initial conditions distributed according to an equilibrium measure  \cite{BePiRoSaSt2013,BlCaSa2014}.  The upper bound implies that the accuracy of this integrator in approximating the average change in energy is actually twice the accuracy of the method.  The lower bound states that the average change in energy is strictly positive.  Recall that the Cayley splitting is volume-preserving and reversible. {\em Then to what extent does this result hold for the Cayley splitting applied to infinite-dimensional Hamiltonian systems?}

\medskip
For any $m \in \mathbb{N}$, define the global energy error of the Cayley splitting as \begin{equation} \label{eq:global_energy_error_cayley}
\Delta(\boldsymbol{u}, \boldsymbol{p}) = H(\boldsymbol{u}^m, \boldsymbol{p}^m) -  H(\boldsymbol{u}, \boldsymbol{p}) 
\end{equation}
where $(\boldsymbol{u}^m, \boldsymbol{p}^m)$ is the output of the Cayley splitting after $m$ integration steps with input $(\boldsymbol{u}, \boldsymbol{p})$, i.e., \[
\begin{bmatrix}
\boldsymbol{u}^m \\
\boldsymbol{p}^m \end{bmatrix} = \boldsymbol{C}^m \begin{bmatrix}
\boldsymbol{u} \\
\boldsymbol{p} \end{bmatrix}  \;.
\]
The following proposition quantifies the average energy error of the Cayley splitting.

\begin{prop} \label{prop:Cayley_Splitting_mean_dH}
The following hold for $\E(\Delta)$ defined as the expected value of the global energy error \eqref{eq:global_energy_error_cayley} of the Cayley splitting
over random initial conditions with non-normalized density $e^{-(\Delta s) H(\boldsymbol{u}, \boldsymbol{p})}$.  
\begin{enumerate}
\item
For all $S>0$, for all $\Delta s>0$,  for all $m \in \mathbb{N}$, and for any positive $\Delta t<2$,  \begin{equation}
\label{eq:mean_dH}
0 \le \E( \Delta ) \le \frac{S}{8} \frac{\Delta t^4}{\Delta s^2} \frac{1}{4 - \Delta t^2} \;.
\end{equation}
\item
Let $\Delta t = \Delta s^{1/2}$ and set $m = \lfloor T/ \Delta t \rfloor$.  Then for all $T>0$ and for all $S>0$,  \begin{equation} \label{eq:asymptotic_mean_dH}
\lim_{\Delta s \to 0} \E( \Delta ) = 0 \quad \text{and} \quad \lim_{\Delta s \to 0} \Var( \Delta ) = 0 \;.
\end{equation}
\end{enumerate}
\end{prop}

Note that the first asymptotic result in \eqref{eq:asymptotic_mean_dH} sharpens the bound in \eqref{eq:mean_dH} in the critical case when $\Delta t = \Delta s^{1/2}$.  The left panel of Figure~\ref{fig:mean_dH_hamiltonian_pde} confirms the upper bound given in \eqref{eq:mean_dH}.  The figure graphs the mean global energy error of the Cayley splitting as a function of the spatial step size $\Delta s$ at time $T=1$ with an initial distribution that has non-normalized density $e^{-(\Delta s) H(\boldsymbol{u}, \boldsymbol{p})}$.  This distribution is an invariant measure for the exact semidiscrete Hamiltonian dynamics in \eqref{eq:semidiscrete_linear_hamiltonian}.   We run the Cayley splitting with time step size $\Delta t = \Delta s^{3/4}$ and $\Delta t = \Delta s$.  According to \eqref{eq:mean_dH}, the mean energy error $\E( \Delta )$ should scale like $O( \Delta s)$ with the former choice of time step size, and like $O(\Delta s^2)$ with the latter choice, as the graphs confirm.

\begin{proof}
We transform to spectral coordinates by using the eigenvectors of the matrix $\boldsymbol{L}$.  
In these coordinates the linear Hamiltonian system decouples into $n-1$ oscillators.   
Since time discretization commutes with this change of variables, \[
\Delta(\boldsymbol{u}, \boldsymbol{p}) = \sum_{i=1}^{n-1} \Delta_i(\boldsymbol{U}_i, \boldsymbol{P}_i) \;, \quad \begin{bmatrix} \boldsymbol{u} \\ \boldsymbol{p} \end{bmatrix} = \begin{bmatrix} \boldsymbol{V}^{\mathrm{T}} &  \\ & \boldsymbol{V}^{\mathrm{T}} \end{bmatrix} \begin{bmatrix} \boldsymbol{U} \\ \boldsymbol{P} \end{bmatrix} \;,
\] where $\Delta_i$ is the energy error in the Cayley splitting applied to the $i$th-oscillator and the matrix $\boldsymbol{V}$ is the orthogonal
matrix whose columns are the orthonormal eigenvectors of the discrete Laplacian $\boldsymbol{L}$.  
Using the formula in Lemma~\ref{lemma:Cayley_Splitting_Mean_Energy_Error_1D} with $\beta = \Delta s$ we obtain  \begin{equation} \label{eq:mean_dH_identity}
\E ( \Delta ) = \frac{1}{8 \Delta s}  \frac{\Delta t^4}{4 - \Delta t^2} \sum_{i=1}^{n-1} \sin^2( m \theta_i ) \;,
\end{equation} where \[
\theta_i = \arccos\left(-1 + \frac{8-2 \Delta t^2}{4 + \Delta t^2 \omega_i^2} \right) \;, \quad \omega_i^2 = \frac{4}{\Delta s^2} \sin^2 \left( \frac{i \pi}{2 n} \right) \;.
\] 
Applying the trivial bounds  $\sin^2(x) \le 1$ and $n-1 < S/\Delta s$ to $\E (\Delta)$ gives the bound in \eqref{eq:mean_dH}.

For the asymptotic result, we set $\Delta t = \Delta s^{1/2}$, $\Delta s=S/n$ and $m = \lfloor T/ \Delta t \rfloor$ in \eqref{eq:mean_dH_identity} to obtain \begin{align*}
\E( \Delta ) &=  \frac{1}{16} \frac{S}{4 - \Delta s} \left( 1-  \frac{1}{n} -  \frac{1}{n} \sum_{i=1}^{n-1}  \cos(2 m \theta_i ) \right) \;.
\end{align*}
The asymptotic result then follows from the limit  \begin{equation} \label{eq:asymptotic_mean_dH_limit}
\lim_{n \to \infty}  \frac{1}{n} \sum_{i=1}^{n-1}  \cos(2 m \theta_i ) = 1
\end{equation} which is valid 
because for any $\epsilon>0$, we can find $N$ large enough such that $n \ge N$, and for all $i \ge N$, $|  \theta_i  - \pi | < \epsilon$ and $|\cos(2 m \theta_i ) - 1| < \epsilon$; and then estimate \begin{align*}
 \left| \frac{1}{n} \sum_{i=1}^{n} \left( \cos(2 m \theta_i ) - 1 \right) \right| &\le  
 \frac{1}{n} \sum_{i=1}^{N} \left| \cos(2 m \theta_i ) - 1 \right| + \frac{1}{n}  \epsilon (n - N - 1) \;.
\end{align*}
Since $\epsilon$ can be made arbitrarily small in this estimate, passing to the limit as $n \to \infty$ implies the first asymptotic result in \eqref{eq:asymptotic_mean_dH_limit} holds, and the second asymptotic is obtained similarly using the formula in Lemma~\ref{lemma:Cayley_Splitting_Variance_Energy_Error_1D}.
\end{proof}

\medskip
The next lemma quantifies the average energy error for more general initial distributions.

\begin{prop} \label{prop:Cayley_Splitting_nu_mean_dH}
For all $S>0$, for all $\Delta s>0$,  for all $m \in \mathbb{N}$, and for any positive $\Delta t<2$, the Cayley splitting in \eqref{eq:cayley_splitting_linear_hamiltonian} satisfies \begin{equation} \label{eq:nu_mean_dH}
\E_{\nu}( \Delta ) \le \frac{1}{2} \frac{\Delta t^2}{4 - \Delta t^2}  \E_{\nu}( \| \boldsymbol{p} \|^2 ) \;,
\end{equation}
where $\E_{\nu}$ denotes an expected value over a given initial distribution $\nu$.
\end{prop}

The right panel of Figure~\ref{fig:mean_dH_hamiltonian_pde} confirms this bound.    The figure graphs the mean global energy error of the Cayley splitting as a function of the spatial step size $\Delta s$ at time $T=1$ with an initial distribution $\nu$ that  is a point mass at zero in position and an $n-1$-dimensional, standard normal distribution in momentum, hence $\E_{\nu}( \| \boldsymbol{p} \|^2 ) = n-1$.  We run the Cayley splitting with time step size $\Delta t = \Delta s^{1/2}$ and $\Delta t = \Delta s$.  According to \eqref{eq:nu_mean_dH}, the mean energy error $\E_{\nu}( \Delta )$ should scale like $O(1)$ with the former choice of time step size, and like $O(\Delta s)$ with the latter choice, as the graphs show.

\begin{proof}
This proof is nearly identical to the first part of the proof of Prop.~\ref{prop:Cayley_Splitting_mean_dH} except one should invoke the cruder but more general estimate in Lemma~\ref{lemma:Cayley_Splitting_Max_Energy_Error_1D} instead of Lemma~\ref{lemma:Cayley_Splitting_Mean_Energy_Error_1D}, and then take expectations with 
respect to $\nu$ to obtain \eqref{eq:nu_mean_dH}.
\end{proof}

\begin{figure}
\begin{center}
\includegraphics[width=0.45\textwidth]{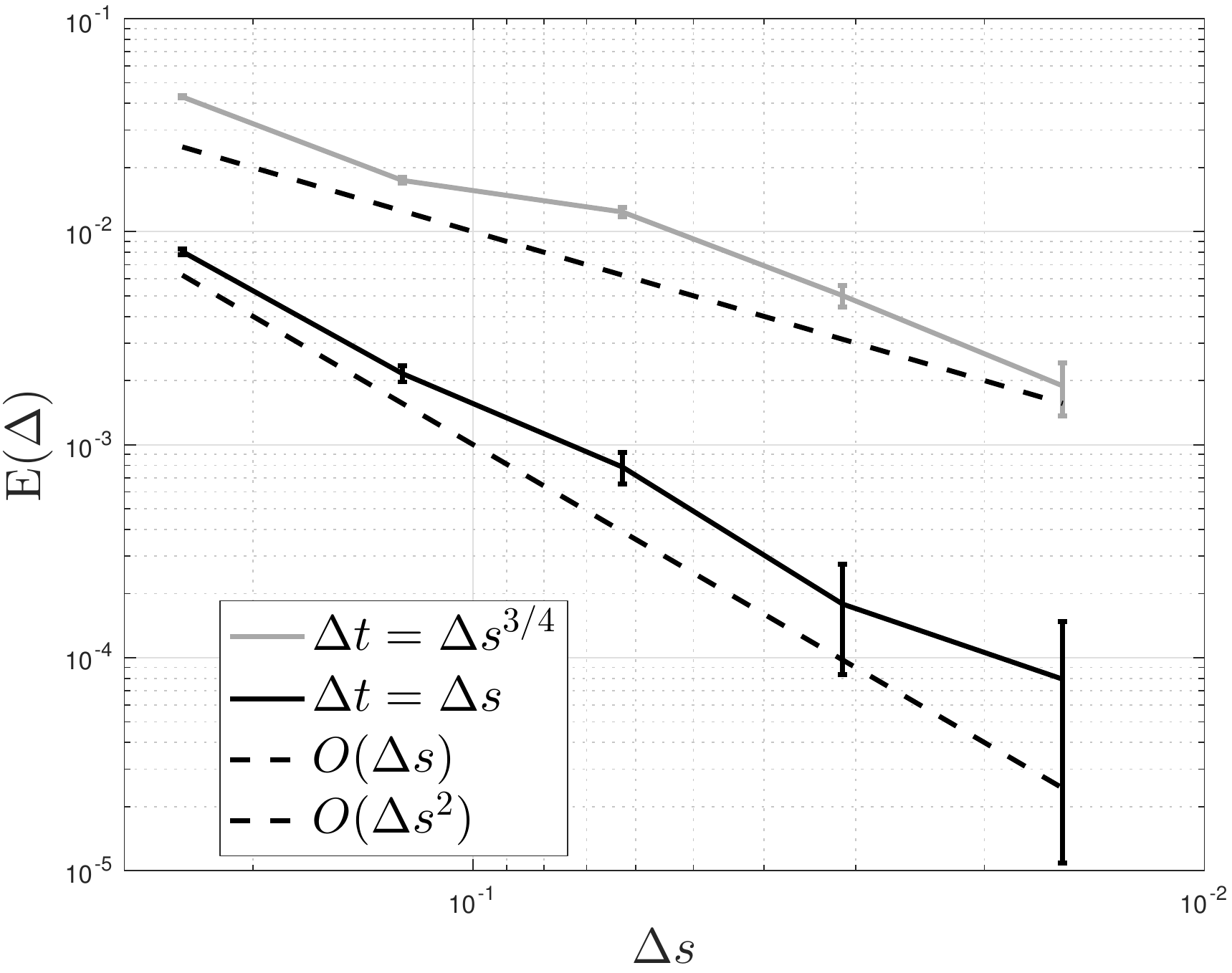}  \hspace{0.1in}
\includegraphics[width=0.45\textwidth]{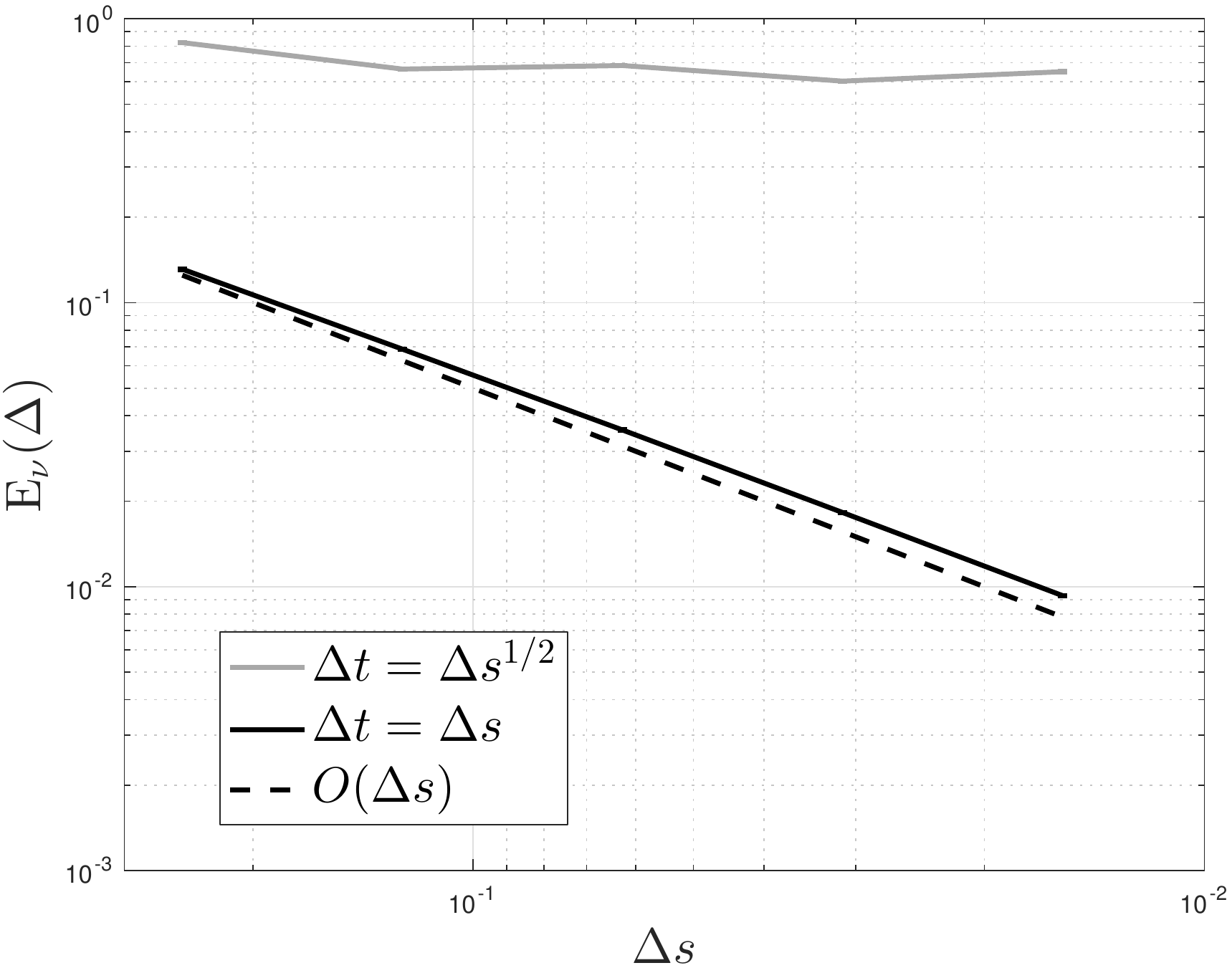}  
\end{center}
\caption{\small  {\bf Mean Energy Errors for Cayley Splitting.}   These figure graph the mean global energy error of the Cayley splitting 
as a function of the spatial step size $\Delta s$ at time $T=1$ with an initial distribution that has non-normalized
density $e^{-(\Delta s) H(\boldsymbol{u}, \boldsymbol{p})}$ (left panel); and an initial distribution that is a point mass at zero in position $\boldsymbol{u} = \boldsymbol{0}$ and 
a standard normal in momentum (right panel).  The time step sizes are related to the spatial step size as indicated in the figure legends.   
This figure is consistent with the bounds given in Prop.~\ref{prop:Cayley_Splitting_mean_dH} and Prop.~\ref{prop:Cayley_Splitting_nu_mean_dH}.  
}
  \label{fig:mean_dH_hamiltonian_pde}
\end{figure}


\subsection{Strong Accuracy of Cayley Splitting for Hamiltonian PDE} \label{eq:strong_accuracy_hamiltonian}

In this part we prove strong accuracy of the Cayley splitting.  We also a Metropolized Cayley splitting and 
estimate the effect on the dynamics of the rejections in the Metropolis-Hastings method \cite{BoVa2010}.

\medskip
The following is an infinite-dimensional analog of Lemma~\ref{lemma:Cayley_Splitting_global_error_1D}.

\begin{prop} \label{prop:Cayley_Splitting_strong_accuracy}
For all $T>0$, for all $S>0$, for all $\Delta s >0$, and for any positive $\Delta t < 2$, there exist positive constants $C_1, C_2$ such that the Cayley splitting in \eqref{eq:cayley_splitting_linear_hamiltonian} 
satisfies \begin{equation} 
\begin{aligned}
\E( \|  \boldsymbol{u}^m - \boldsymbol{u}(m \Delta t) \| )  &\le C_1 T \Delta t^2 \Delta s^{-7/2} \\
\E( \| \boldsymbol{p}^m  - \boldsymbol{p}(m \Delta t) \| )  &\le C_2 T \Delta t^2 \Delta s^{-4}
\end{aligned}
\end{equation}
where $\E$ denotes an expected value over random initial conditions with non-normalized density $e^{-(\Delta s)  H(\boldsymbol{u}, \boldsymbol{p})}$.  
\end{prop}

This result is numerically verified in Figure~\ref{fig:strong_accuracy_hamiltonian_pde}.

\begin{proof}
We transform to spectral coordinates by using the eigenvectors of the matrix $\boldsymbol{L}$.  
Given initial conditions $(\boldsymbol{u}, \boldsymbol{p})$, let $(\boldsymbol{U}, \boldsymbol{P})$ denote the corresponding
spectral variables.  By Lemma~\ref{lemma:Cayley_Splitting_global_error_1D}, \begin{align*}
\|  \boldsymbol{u}^m - \boldsymbol{u}(m \Delta t) \| &\le C_1 T \Delta t^2  \left( \sum_{i=1}^{n-1} (1+\omega_i^3) | \boldsymbol{U}_i |  + \sum_{i=1}^{n-1} (1+\omega_i^2) | \boldsymbol{P}_i |  \right)
\end{align*}
where $\{ \omega_i \}$ are given explicitly in \eqref{eq:omi2}.  
Taking expectations then yields \begin{align*}
\E ( \|  \boldsymbol{u}^m - \boldsymbol{u}(m \Delta t) \| ) &\le C_1 T \Delta t^2  \Delta s^{-1/2} \left( \sum_{i=1}^{n-1} \frac{1+\omega_i^3}{\sqrt{1+\omega_i^2}}   + \sum_{i=1}^{n-1} (1+\omega_i^2) \right)
\end{align*}
since Jensen's inequality implies that \[
\E( | \boldsymbol{U}_i |) \le  \Delta s^{-1/2} (1+\omega_i^2)^{-1/2} \;, \quad  \E( | \boldsymbol{P}_i |   ) \le \Delta s^{-1/2}  \;.
\]   By \eqref{eq:omi2},  \[
\sum_{i=1}^{n-1} \omega_i^2  = \frac{4}{\Delta s^2} \sum_{i=1}^{n-1} \sin^2 \left( \frac{i \pi}{2 n} \right) = \frac{1}{\Delta s^2} (2 n  - 2)
\]
by Lagrange's trigonometric identities.  The result then follows from applying this identity with the fact that $n = S / \Delta s$.   One can similarly obtain the bound on the 
global error in momentum.
\end{proof}

\medskip
Next we use a Metropolis method to set the invariant distribution of the integrator.  The following Metropolis method is a special case of Algorithm~5.1 in Ref.~\cite{BoSaActaN2018}.
It follows from Prop.~5.1 of Ref.~\cite{BoSaActaN2018} that this Metropolized Cayley method has an invariant distribution with non-normalized density $e^{-(\Delta s)  H(\boldsymbol{u}, \boldsymbol{p})}$.

\begin{algorithm}[Metropolized Cayley] \label{algo:metro_cayley_linear_hamiltonian_system}
The one-step $(\boldsymbol{u}^0, \boldsymbol{p}^0) \mapsto (\boldsymbol{u}^1, \boldsymbol{p}^1)$ update is given by.
\begin{description}
\item[(Step 1)] 
 Compute a proposal move $
 \begin{bmatrix}  \boldsymbol{\tilde u}^1 \\  \boldsymbol{\tilde p}_1 \end{bmatrix} = \boldsymbol{C} \begin{bmatrix} \boldsymbol{u}^0 \\ \boldsymbol{p}^0 \end{bmatrix} \;.
 $
 \medskip
\item[(Step 2)] 
 Take as actual update \[
(\boldsymbol{u}^1, \boldsymbol{\tilde p}_1) = \gamma ( \boldsymbol{\tilde u}^1, \boldsymbol{\tilde p}^1)  + (1-\gamma) (\boldsymbol{u}^0, - \boldsymbol{p}^0) 
\] where $\gamma$ is a Bernoulli random variable with parameter $
\alpha(\boldsymbol{u}^0, \boldsymbol{p}^0) = 1 \wedge e^{-(\Delta s) \left( H( \boldsymbol{\tilde u}^1,  \boldsymbol{\tilde p}^1) - H(\boldsymbol{u}^0, \boldsymbol{p}^0) \right) } 
$.
\end{description}
\end{algorithm}

Note that when a proposal move is rejected the momentum component is flipped and local accuracy is lost since
$\E (  \left\|  (\boldsymbol{u}^0, - \boldsymbol{p}^0) -  (\boldsymbol{u}^0,  \boldsymbol{p}^0) \right\| ) = \E( \|  \boldsymbol{p}^0 \|) $.   
Nevertheless, the probability of rejection turns out to be within the strong local order of accuracy of the Cayley splitting.


\begin{figure}
\begin{center}
\includegraphics[width=0.65\textwidth]{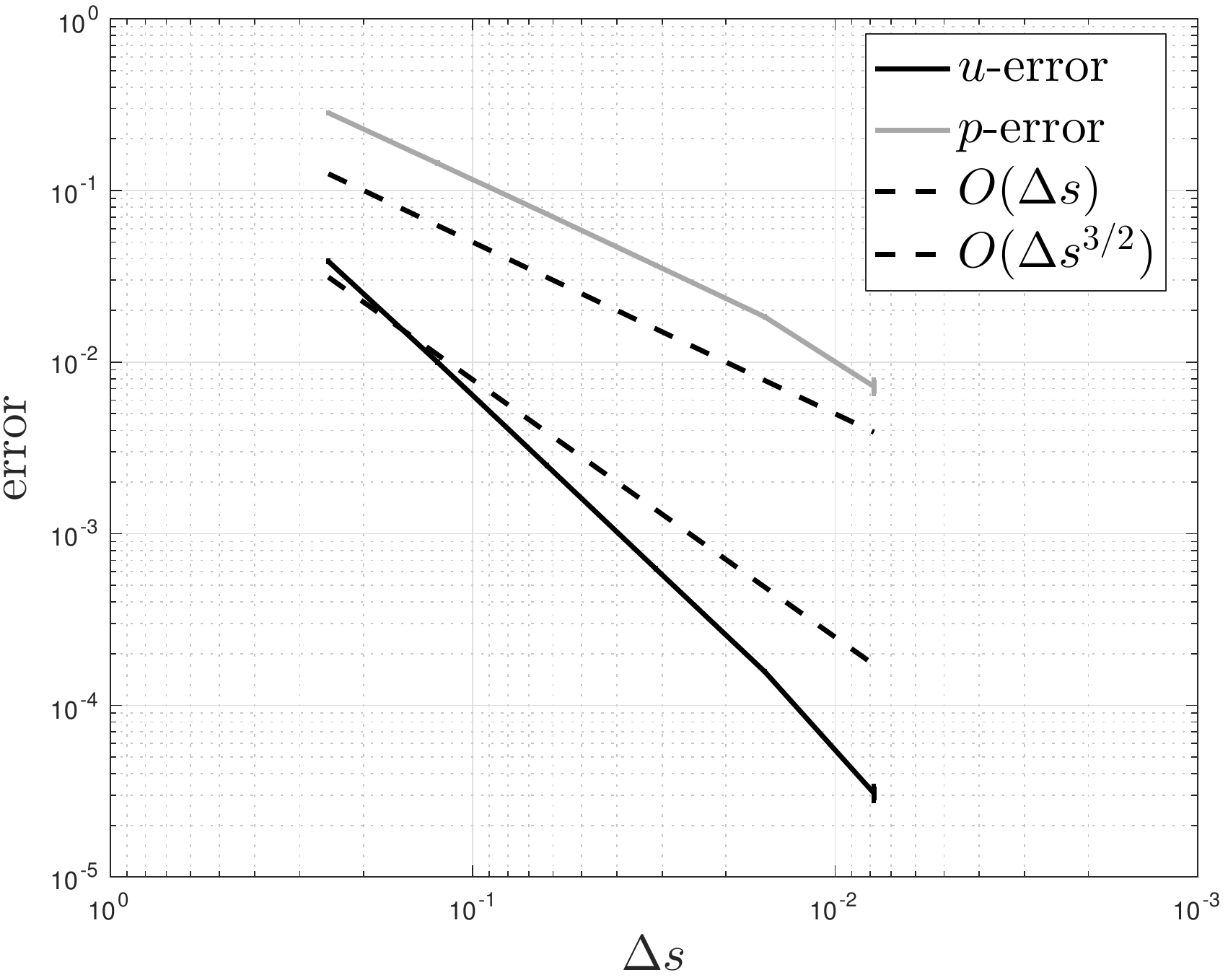}  
\end{center}
\caption{\small  {\bf Strong Accuracy of Cayley Splitting.}  This figure graphs the strong error of the Cayley splitting 
in position/momentum as a function of the spatial step size $\Delta s$ at time $T=1$ with an initial distribution that has non-normalized
density $e^{-(\Delta s) H(\boldsymbol{u}, \boldsymbol{p})}$.  The time step sizes are related to the spatial step size via $\Delta t = \Delta s^{5/2}$.
Note that the convergence rate of $O(\Delta s^{3/2})$ in position and $O(\Delta s)$ in momentum are consistent with those predicted 
in Prop.~\ref{prop:Cayley_Splitting_strong_accuracy}.
}
  \label{fig:strong_accuracy_hamiltonian_pde}
\end{figure}


\subsection{Cayley Splitting for Second Order Langevin SPDE}

Next we study weak accuracy of the Cayley splitting applied to the linear Langevin SPDE
\begin{equation} \label{eq:linear_langevin_spde}
\begin{aligned}
\partial_t u(t,s) &= p(t,s)  \;, \\
\partial_t p(t,s) &= \partial_{s}^2 u(t,s) - u(t,s) - \gamma p(t,s)  + \sqrt{2 \gamma} \partial_t W(t,s) \;, 
\end{aligned} 
\end{equation}
following here the same notation used in \eqref{eq:langevin_spde_intro} with $\beta=1$.  The corresponding semidiscrete equations are
given by: 
\begin{equation} \label{eq:linear_langevin_semidiscrete}
\begin{aligned}
d \boldsymbol{u}(t) &= \boldsymbol{p}(t)  dt \;, \\
d \boldsymbol{p}(t) &= \boldsymbol{L} \boldsymbol{u}(t)  dt - \boldsymbol{u}(t) dt - \gamma \boldsymbol{p}(t) dt + \sqrt{\frac{2 \gamma}{\Delta s}} d \boldsymbol{W}(t) \;,
\end{aligned}
\end{equation}
where we again follow the same notation used in \eqref{eq:semidiscrete_intro}.   This SDE is linear, and its solution at any time $T>0$ is given in law by \begin{align*}
\begin{bmatrix} \boldsymbol{u}(T) \\ \boldsymbol{p}(T) \end{bmatrix} \overset{d}{=} 
\boldsymbol{\Phi}(T) \begin{bmatrix} \boldsymbol{u}(0) \\ \boldsymbol{p}(0) \end{bmatrix} + \boldsymbol{\Gamma}(T)^{1/2} \begin{bmatrix} \boldsymbol{\xi} \\ \boldsymbol{\eta} \end{bmatrix}
\end{align*}
where $\boldsymbol{\xi}$ and $\boldsymbol{\eta}$ are independent standard normal $n-1$ dimensional vectors, the matrix $\boldsymbol{\Phi}(t)$ is defined as \[
\boldsymbol{\Phi}(t) =\exp\left( t \boldsymbol{K} \right)  \;, \quad \text{where} \quad \boldsymbol{K} = \begin{bmatrix} \boldsymbol{0}  & \boldsymbol{I} \\ \boldsymbol{L} -\boldsymbol{I}  & -\gamma \boldsymbol{I} \end{bmatrix}
\]
and the matrix $\boldsymbol{\Gamma}(t)$ is defined via the Lyapunov equation \[
 \boldsymbol{K}  \boldsymbol{\Gamma}(t) +  \boldsymbol{\Gamma}(t)  \boldsymbol{K}  = \frac{2 \gamma}{\Delta s} \left( \boldsymbol{\Phi}(t) \begin{bmatrix}~~ & \\ & \boldsymbol{I} \end{bmatrix} \boldsymbol{\Phi}(t)^{\mathrm{T}} -  \begin{bmatrix}~~ & \\ & \boldsymbol{I} \end{bmatrix} \right) \;.
\]
In other words, given an initial condition $(\boldsymbol{u}(0), \boldsymbol{p}(0))$, the exact solution of \eqref{eq:linear_langevin_semidiscrete} at any time $t>0$ is a $2 \times (n-1)$-dimensional Gaussian vector with mean vector \[
\boldsymbol{\mu}(t) = \boldsymbol{\Phi}(t) \begin{bmatrix} \boldsymbol{u}(0) \\ \boldsymbol{p}(0) \end{bmatrix} 
\] and covariance matrix $\boldsymbol{\Gamma}(t)$.

We approximate this solution by using the Cayley splitting given in \eqref{eq:cayley_splitting}.  
In the case of the linear Langevin SPDE in \eqref{eq:linear_langevin_spde}, the output of this  splitting can be obtained by completing the following steps.

\begin{algorithm}[Cayley Splitting for Second-order Langevin SPDE] \label{algo:cayley_linear_langevin_system}
The one-step $(\boldsymbol{u}^0, \boldsymbol{p}^0) \mapsto (\boldsymbol{u}^1, \boldsymbol{p}^1)$ update is given by.
\begin{description}
\item[(Step 1)] 
Half-step of Ornstein-Uhlenbeck flow in momentum \[
\boldsymbol{p}^{1/3} = e^{-  \frac{\gamma \Delta t}{2}} \boldsymbol{p}^0 + \sqrt{\dfrac{1}{\Delta s}} \sqrt{1 - e^{-\gamma \Delta t}} \boldsymbol{\xi}^0
\]
where $\boldsymbol{\xi}^0$ is a standard normal $(n-1)$-dimensional vector.
 \medskip
\item[(Step 2)]  
 Full-step of numerical Hamiltonian flow by Cayley splitting \[
 \begin{bmatrix}  \boldsymbol{u}^1 \\  \boldsymbol{p}^{2/3} \end{bmatrix} = \boldsymbol{C} \begin{bmatrix} \boldsymbol{u}^0 \\ \boldsymbol{p}^{1/3} \end{bmatrix} \;.
 \]
\item[(Step 3)] 
Half-step of Ornstein-Uhlenbeck flow in momentum \[
\boldsymbol{p}^{1} = e^{-  \frac{\gamma \Delta t}{2}} \boldsymbol{p}^{2/3} + \sqrt{\dfrac{1}{\Delta s}} \sqrt{1 - e^{-\gamma \Delta t}} \boldsymbol{\eta}^0
\]
where $\boldsymbol{\eta}^0$ is a standard normal $(n-1)$-dimensional vector.
\end{description}
\end{algorithm}

To set the invariant distribution of this method, one can replace (Step 2) with the following step.
\begin{description}
\item[(Step 2)$^\star$]  
 Full-step of numerical Hamiltonian flow by Metropolized Cayley splitting \[
 \begin{bmatrix}  \boldsymbol{u}^1 \\  \boldsymbol{p}^{2/3} \end{bmatrix} =\gamma \begin{bmatrix} \boldsymbol{\tilde u}^1 \\ \boldsymbol{\tilde p}^{2/3} \end{bmatrix}  + 
 (1-\gamma) \begin{bmatrix} \boldsymbol{u}^0 \\ - \boldsymbol{p}^{1/3} \end{bmatrix} \;.
 \] where $(\boldsymbol{\tilde u}^1, \boldsymbol{\tilde p}^{2/3})$ is defined as  \[
 \begin{bmatrix}  \boldsymbol{\tilde u}^1 \\  \boldsymbol{\tilde p}^{2/3} \end{bmatrix} = \boldsymbol{C} \begin{bmatrix} \boldsymbol{u}^0 \\ \boldsymbol{p}^{1/3} \end{bmatrix} 
 \] and $\gamma$ is a Bernoulli random variable with parameter \[
\alpha(\boldsymbol{u}^0, \boldsymbol{p}^{1/3}) = 1 \wedge e^{-(\Delta s) \left( H( \boldsymbol{\tilde u}^1,  \boldsymbol{\tilde p}^{2/3}) - H(\boldsymbol{u}^0, \boldsymbol{p}^{1/3}) \right) } \;.
\]
\end{description}

Note that after $m$ integration steps or iterations of Algorithm~\ref{algo:cayley_linear_langevin_system}, the Cayley splitting is itself a Gaussian vector with mean vector $\boldsymbol{\mu}^m$ and
covariance matrix $\boldsymbol{\Gamma}^m$.   Of course, this is no longer true for the Metropolized Cayley splitting after $m$ integration steps.  
More precisely, let $\boldsymbol{O}$ denote the matrix associated to the drift part of (Step 1) and (Step 3), i.e., \[
\boldsymbol{O} = \begin{bmatrix} \boldsymbol{I} & \\
& e^{- \frac{\gamma \Delta t}{2}} \boldsymbol{I} \end{bmatrix}  \;.
\] 

\medskip
Let $g$ be the function defined in \eqref{eq:gam_stability}.  The following statement gives sufficient conditions for asymptotic stability of the Cayley splitting.

\begin{prop} \label{prop:Cayley_Splitting_Stability_Langevin}
For all $T>0$, for all $S>0$, for any $\gamma>0$, for all $m \in \mathbb{N}$, for all $\Delta s >0$, and for any positive $\Delta t<2$ satisfying \begin{equation} \label{eq:gam_stability_spde}
g(\Delta t, \frac{2}{\Delta s}, \gamma) > 0
\end{equation}
the matrix $ (\boldsymbol{O} \boldsymbol{C} \boldsymbol{O})^m$ is asymptotically stable.
\end{prop}

Essentially, this stability condition requires that $\Delta t \lesssim \Delta s^2$.  Note that this condition is stronger than in the Hamiltonian case because we require that $\boldsymbol{O} \boldsymbol{C} \boldsymbol{O}$ is asymptotically stable like the drift part of the exact Langevin dynamics.

\begin{proof}
This proof is an application of Lemma~\ref{lemma:Cayley_Splitting_Stability_langevin_1D}.
In view of \eqref{eq:gam_stability}, we require that \[
\min_{1 \le i \le n-1} g(\Delta t, \omega_i, \gamma) > 0 
\] where $\omega_i$ is given in \eqref{eq:omi2}.   Note also that \[
g(\Delta t, \omega_{i}, \gamma) > g(\Delta t, \omega_{i+1}, \gamma) \ge g(\Delta t, \frac{2}{\Delta s} , \gamma)
\] for all $1 \le i \le n-2$, as required.
\end{proof}

\medskip
Under this stability condition, the Cayley splitting is accurate in the following weak or distributional sense.

\begin{prop} \label{prop:Cayley_Splitting_weak_accuracy_langevin}
For all $T>0$, for all $S>0$, for any $\gamma>0$, for all $\Delta s >0$, and for any positive $\Delta t$ satisfying \eqref{eq:gam_stability_spde}, there exist positive constants $C_1, C_2$ (depending on $\gamma$ and $S$) such that the Cayley splitting given in Algorithm~\ref{algo:cayley_linear_langevin_system} satisfies \begin{equation} 
\begin{aligned}
 \|  (\boldsymbol{O} \boldsymbol{C} \boldsymbol{O})^m - \boldsymbol{\Phi}(m \Delta t)  \|   &\le \exp(C_1 T) \Delta t^2 \Delta s^{-5} \\
 \| \boldsymbol{\Gamma}^m  - \boldsymbol{\Gamma}(m \Delta t) \|  &\le  \exp(C_2 T) \Delta t^2 \Delta s^{-4}
\end{aligned}
\end{equation}
\end{prop}

This result is numerically verified in Figure~\ref{fig:weak_accuracy_langevin_spde}.

\begin{proof}
We transform to spectral coordinates by using the eigenvectors of the matrix $\boldsymbol{L}$.  
By Lemma~\ref{lemma:Cayley_Splitting_global_error_langevin_1D} with $\beta=\Delta s$, and in view of \eqref{eq:K_scaling}, we have \begin{align*}
\|  (\boldsymbol{O} \boldsymbol{C} \boldsymbol{O})^m - \boldsymbol{\Phi}(m \Delta t)  \|   &\le \exp(C_1 T)  \Delta t^2  \sum_{i=1}^{n-1} (1 + \omega_i^4) \\
\|  \boldsymbol{\Gamma}^m  - \boldsymbol{\Gamma}(m \Delta t) \|    &\le \exp(C_2 T)  \Delta t^2 \Delta s^{-1}  \sum_{i=1}^{n-1} (1 + \omega_i^2)
\end{align*}
The result follows by using \eqref{eq:omi2} to bound $\omega_i^2$ uniformly over $1 \le i \le n-1$ and then substituting $n = S / \Delta s$.
\end{proof}

%
%

\begin{figure}
\begin{center}
\includegraphics[width=0.65\textwidth]{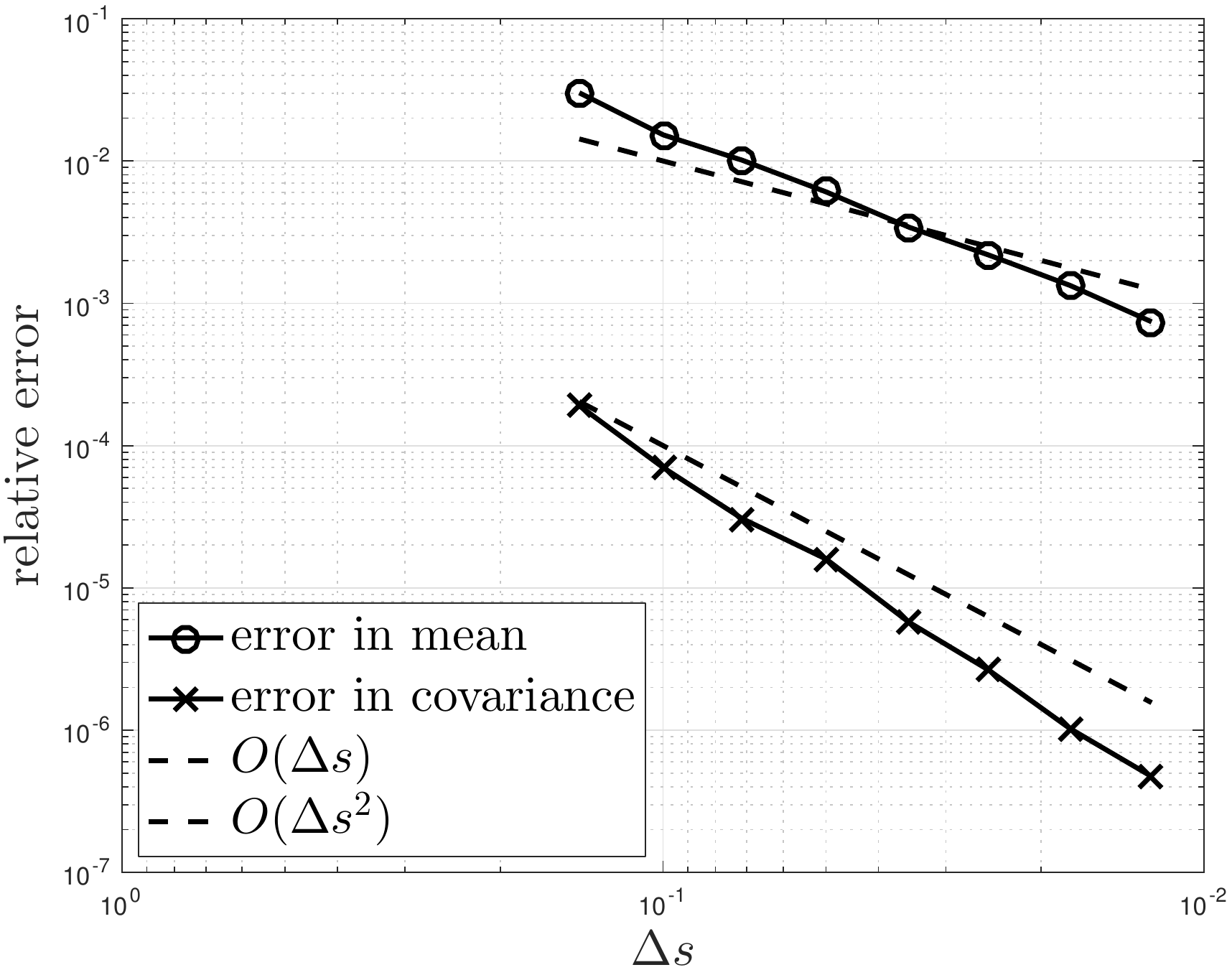}  
\end{center}
\caption{\small  {\bf Weak Accuracy of Cayley Splitting.}  This figure graphs the weak error of the Cayley splitting 
in representing the mean and covariance of the exact solution of \eqref{eq:linear_langevin_semidiscrete} as a function of the spatial step size $\Delta s$ 
at time $T=4$ and with $\gamma=0.8$.  The time step sizes are related to the spatial step size via $\Delta t = \Delta s^3$. Note that the convergence rate of 
$O(\Delta s)$ in the mean vector and $O(\Delta s^2)$ in the momentum matrix are consistent with those found in Prop.~\ref{prop:Cayley_Splitting_weak_accuracy_langevin}.
}
  \label{fig:weak_accuracy_langevin_spde}
\end{figure}


\subsection{HMC on Hilbert Spaces with Cayley Splitting} \label{sec:cayley_based_hmc}

The Hamiltonian Monte Carlo (HMC) algorithm is a tool for sampling from an absolutely continuous probability distribution (called the target distribution) given a function proportional to its probability density function (correspondingly called the target density).  HMC was first introduced in 1987 to study lattice models of quantum field theory, and about a decade later, popularized in data science \cite{DuKePeRo1987,Li2008,Ne2011,BoSaActaN2018}.   The idea in HMC is to sample from a product density whose first component is the target density and second component is a multivariate Gaussian of the same dimension.  The negative log of this product density then plays the role of a Hamiltonian function $H$.  This algorithm can also treat target distributions on Hilbert spaces \cite{BePiSaSt2011}.

The standard HMC algorithm consists of iterating Hamiltonian flows associated to $H$ for a fixed duration with initial conditions obtained from a sequence of momentum randomization steps in which all components of the momentum are randomized.  Since the exact Hamiltonian flow is typically unavailable, a numerical integrator is used instead.  The bias in the invariant measure due to time discretization is then removed by a Metropolis accept-reject step.  There are two key parameters in the algorithm: the time step size $\Delta t$ used by this numerical integrator and the number of integration steps $m$ between momentum randomizations.   Together these parameters set the durations of the Hamiltonian flows $T= m \Delta t$.  In infinite dimensions, there will be a third parameter: the spatial step size $\Delta s$.  Here we study an HMC algorithm whose target distribution (in phase space) has non-normalized density $e^{- (\Delta s) H(\boldsymbol{u}, \boldsymbol{p}) }$ where $H$ is the semidiscrete Hamiltonian function given in \eqref{eq:hamiltonian_semidiscrete_linear_hamiltonian_system}.  As a numerical integrator for the Hamiltonian dynamics in HMC, we use the Cayley splitting.

\medskip
The HMC algorithm consists of iterating the following steps.

\begin{algorithm}[Cayley-based HMC] \label{algo:hmc_linear_hamiltonian_system}
Given a \# of integration steps $m$, the one-step $\boldsymbol{u}^0 \mapsto \boldsymbol{u}^1$ update is given by.
\begin{description}
\item[(Step 1)] 
Draw a random vector $\boldsymbol{p}^0$ whose components are independent standard normal random variables.
\medskip
\item[(Step 2)] 
 Set $
 \begin{bmatrix}  \boldsymbol{\tilde u}^1 \\  \boldsymbol{\tilde p}^1 \end{bmatrix} = \boldsymbol{C}^m \begin{bmatrix} \boldsymbol{u}^0 \\ \boldsymbol{p}^0 \end{bmatrix} \;.
 $
 \medskip
\item[(Step 3)] 
 Take as actual update $
\boldsymbol{u}^1 = \gamma  \boldsymbol{\tilde u}^1 + (1-\gamma) \boldsymbol{ u}^0 
$ where $\gamma$ is a Bernoulli random variable with parameter \[
\alpha(\boldsymbol{u}^0, \boldsymbol{p}^0) = 1 \wedge e^{-(\Delta s) \left( H( \boldsymbol{\tilde u}^1,  \boldsymbol{\tilde p}^1) - H(\boldsymbol{u}^0, \boldsymbol{p}^0) \right) } \;.
\] 
\end{description}
\end{algorithm}

Note that in (Step 2) a proposal move is computed by numerically evolving the Hamiltonian flow for a fixed time duration of $T = m \Delta t$, and in (Step 3)
this proposal move is accepted with probability $\alpha(\boldsymbol{u}_0, \boldsymbol{p}_0)$.  
It is quite standard to show that this chain preserves the correct target distribution \cite{Li2008}; for a short proof see Theorem 5.2 of~\cite{BoSaActaN2018}.   The basic  idea behind the algorithm is:  if $m$ is sufficiently large and $\Delta t$ is sufficiently small, then the HMC algorithm produces states near the end of a Hamiltonian trajectory in (Step 2), and since the Hamiltonian is preserved, this state is likely to be accepted in (Step 3).  Note that if a move is rejected, then the updated state is the current state, the correlation in the chain increases and the cost of (Step 2) is wasted.   Thus, low acceptance rates are counterproductive.   Figure~\ref{fig:var_x_hmc_linear} illustrates the accuracy of this algorithm in computing the variance in each component of the $\boldsymbol{u}$-marginal of 
$e^{-(\Delta s) H( \boldsymbol{u},\boldsymbol{p} ) }$.

\medskip
The next proposition identifies a scaling of $\Delta t$ that gives an acceptance probability that converges to a nontrivial limit as $\Delta s \to 0$.

\begin{prop} \label{prop:hmc_ap}
Referring to the acceptance probability in (Step 3) of Algorithm~\ref{algo:hmc_linear_hamiltonian_system}, and using the same notation as Prop.~\ref{prop:Cayley_Splitting_mean_dH}, let $\Delta t = \Delta s^{1/4}$ and set $m = \lfloor T/ \Delta t \rfloor$.  Then for all $T>0$ and for all $S>0$ \begin{equation} \label{eq:asymptotic_mean_ap}
\lim_{\Delta s \to 0} \E( \alpha ) = 1 \;.
\end{equation}
\end{prop}

\begin{proof}
Consider the triangular array of random variables \[
\{ X_{n,i} \mid 1 \le i \le n-1 \} \quad \text{where $X_{n,i} = \Delta s (\Delta_{i} - \E( \Delta_i ) )$} \;.
\] Lemmas~\ref{lemma:Cayley_Splitting_Mean_Energy_Error_1D},~\ref{lemma:Cayley_Splitting_Variance_Energy_Error_1D}, and~\ref{lemma:Cayley_Splitting_4th_moment_Energy_Error_1D} with $\beta=\Delta s$ and $\Delta t = \Delta s^{1/4}$ imply that \begin{align*}
\E( X_{n,i} ) &=  0 \\ 
\E( X_{n,i}^2 ) &= \sin^2(m \theta_i) \frac{\Delta s}{(4 - \Delta s^{1/2} )^2} \left( 1 + O(\Delta s^{1/2}) \right)  \\
\E(  X_{n,i}^4  ) &= \sin^4(m \theta_i) \frac{\Delta s^2}{(4 -\Delta s^{1/2} )^4} \left( \frac{18432}{473}  + O(\Delta s)  \right) 
\end{align*}
Let $\tau_n^2 = \sum_{i=1}^n \E( X_{n,i}^2 )$.  Then \begin{align*}
\lim_{n \to \infty} \sum_{k=1}^n \frac{1}{\tau_n^4} \E\left(  X_{n,i}^4  \right)  &=\lim_{n \to \infty} \frac{ \sum_{k=1}^n \sin^4(m \theta_k)  \left( \frac{18432}{473}  + O(\Delta s) \right) }{\left(\sum_{j=1}^n \sin^2(m \theta_j) ( 1+ O(\Delta s^{1/2}) ) \right)^2}   \\
 &\le\lim_{n \to \infty} \frac{ \sum_{k=1}^n \sin^2(m \theta_k)  \left( \frac{18432}{473}  + O(\Delta s) \right) }{\left(\sum_{j=1}^n \sin^2(m \theta_j) ( 1+ O(\Delta s^{1/2}) ) \right)^2}  =0
\end{align*}
since $ \sum_{k=1}^n \sin^2(m \theta_k) \sim O(n^{1/2})$.  Thus, Lyapunov's condition holds with $\delta = 2$ and $\sum_{i=1}^{n-1} X_{n,i}/\tau_n \overset{d}{\to} \mathcal{N}(0,1)$ \cite[Theorem 27.3]{billingsley2012probability}. Moreover, \[
\E( \alpha ) = 1 \wedge \exp\left( - \tau_n \frac{ \sum_{i=1}^{n-1} X_{n,i} }{\tau_n} - \Delta s \sum_{i=1}^{n-1} \E( \Delta_i ) \right)
\] Since the function $u \mapsto 1 \wedge e^u$ is bounded and $\tau_n^2 \sim n^{-1/2}$, the dominated convergence theorem implies the desired result.
\end{proof}

We conclude this section with a link between non-preconditioned MALA\footnote{See \S6.1 of Ref.~\cite{BeRoStVo2008} for a detailed description of this MALA algorithm.} and the HMC Algorithm~\ref{algo:hmc_linear_hamiltonian_system} based on the Cayley splitting.  We prove this Proposition in a more general setting in Prop.~\ref{prop:mala_hmc} below.

\begin{prop} \label{prop:hmc_mala}
For all $S>0$, for any $\Delta s>0$ and for any positive $\Delta t<2$, the HMC Algorithm~\ref{algo:hmc_linear_hamiltonian_system} with $m=1$ is equivalent in law to non-preconditioned MALA with a Crank-Nicolson proposal move and operated at a time step size of $(1/2) \Delta t^2 $, and with non-normalized target density given by the $\boldsymbol{u}$-marginal of $e^{- (\Delta s) H(\boldsymbol{u}, \boldsymbol{p}) }$.
\end{prop}

\begin{figure}
\begin{center}
\includegraphics[width=0.45\textwidth]{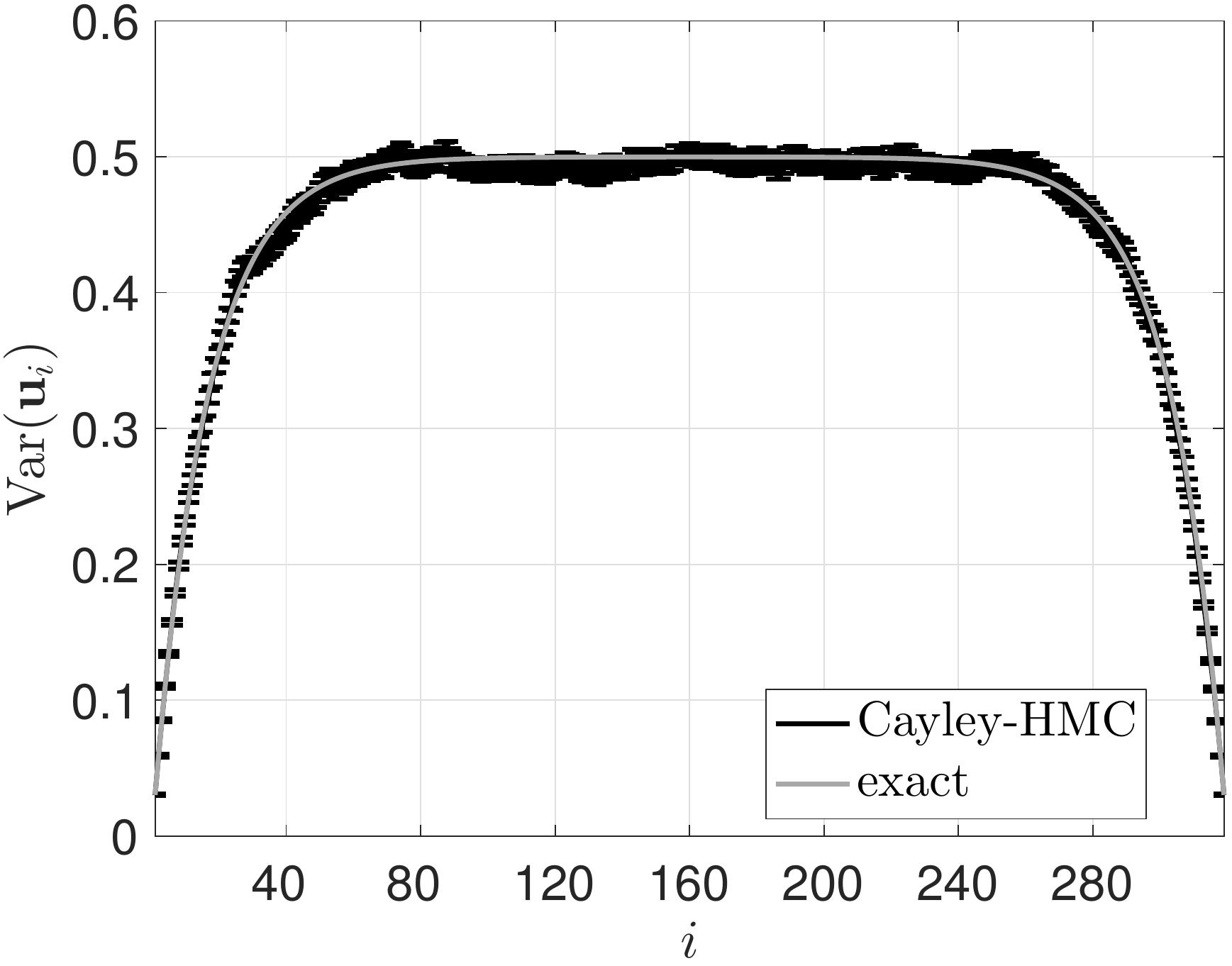}  \hspace{0.1in}
\includegraphics[width=0.45\textwidth]{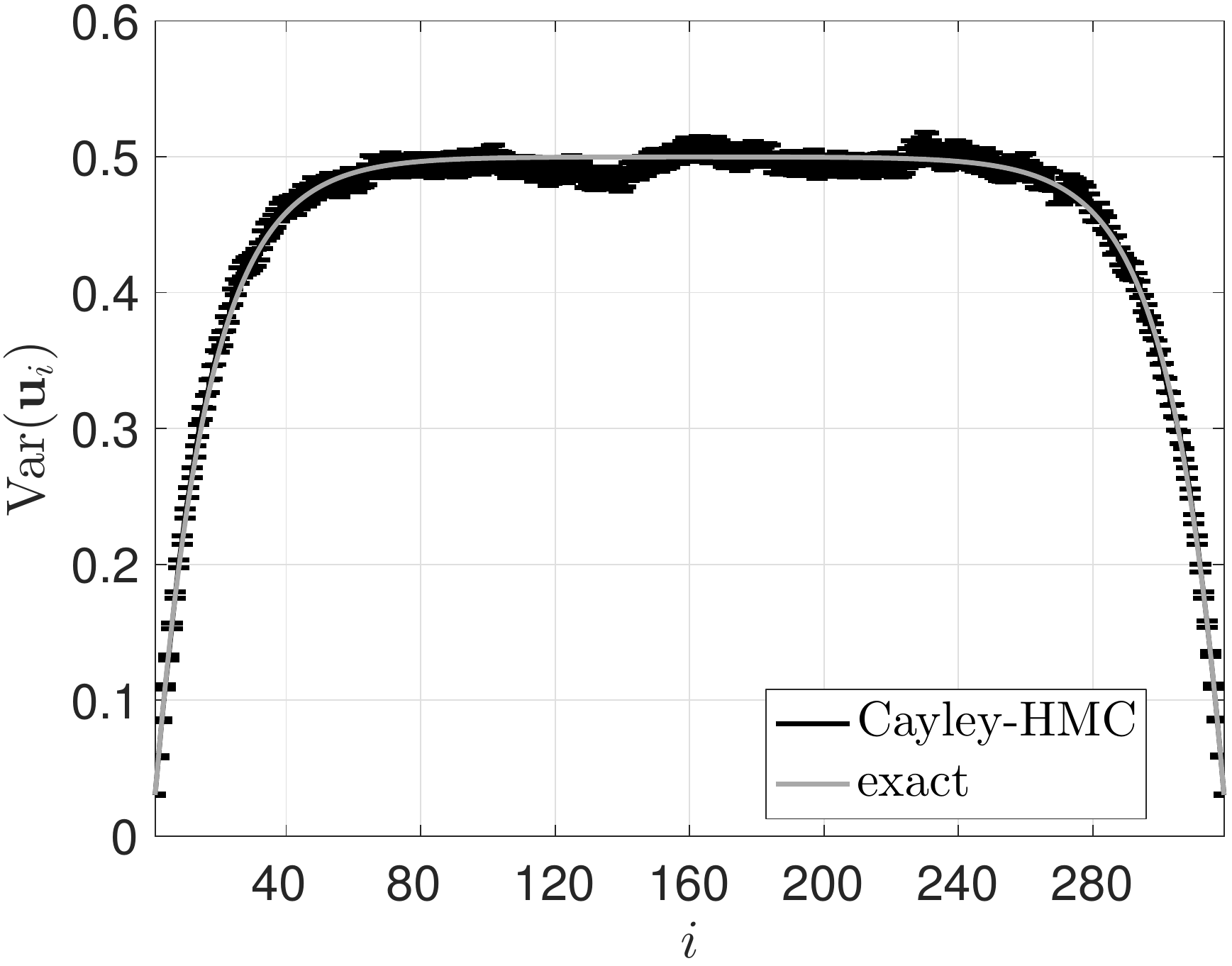}  
\end{center}
\caption{\small  {\bf Cayley-based HMC.}   
This figure verifies the accuracy of Algorithm~\ref{algo:hmc_linear_hamiltonian_system} in computing the variance in each component of the $\boldsymbol{u}$-marginal of $e^{-(\Delta s) H(\boldsymbol{u}, \boldsymbol{p}) }$.  
The $x$-axis labels the components.  The exact variance is also given for comparison.  For both panels, the number of samples is $10^4$, the duration of the Hamiltonian legs is fixed at $T=5$, and the spatial step size is $\Delta s = 0.03125$, and the time step size is $\Delta t = 0.5$ (left panel) and $\Delta t=0.25$ (right panel).  The corresponding average acceptance probabilities are $63\%$ and $91\%$, respectively.  
}
  \label{fig:var_x_hmc_linear}
\end{figure}


\newpage

\section{Application to Diffusion Bridges} \label{sec:diffusion_bridges}

As an application of the Cayley splitting, we turn to a class of second-order Langevin SPDE problems whose marginal invariant measure 
in position is the law of a multidimensional diffusion bridge.

\subsection{Diffusion Bridges} \label{sec:diffusion_bridge_definition}

We fix a time horizon $S>0$ and consider the process $\mathsf{X}: [0, S] \to \mathbb{R}^d$ which solves the SDE \begin{equation} \label{eq:diffusion_bridge}
d \mathsf{X}(s) = - \nabla V( \mathsf{X}(s) ) ds + \sqrt{2 \beta^{-1}} d \mathsf{W}(s)
\end{equation}
conditioned on both initial and final conditions \[
\mathsf{X}(0) = x^- \quad \text{and} \quad \mathsf{X}(S) = x^+ 
\]  where $\beta>0$ is an inverse temperature parameter, $V: \mathbb{R}^d \to \mathbb{R}$ is a potential energy function and $ \mathsf{W}$ is a $d$-dimensional 
standard Brownian motion.  This process is known as a \textit{diffusion bridge} \cite{ReVa2005, BeRoStVo2008, HaStVo2009}. 
The law of this diffusion bridge is a probability measure on paths from $x^-$ to $x^+$ that has a density proportional to \begin{equation} \label{eq:density_of_diffusion_bridge}
\Pi(u) = \exp\left( -\frac{\beta}{2} \int_0^S \left[ \frac{1}{2} | \partial_s u(s) |^2 + G(u(s)) \right] ds \right) \;,
\end{equation}  where $G : \mathbb{R}^d \to \mathbb{R}$ is called the path potential energy function defined by \begin{equation} \label{eq:path_potential}
G(x) = \frac{1}{2} | \nabla V(x)|^2 - \beta^{-1} \Delta V(x) \;, \quad x \in \mathbb{R}^d \;.
\end{equation}

\subsection{First-Order, Semilinear Langevin SPDE} \label{sec:first_order_langevin}

The distribution of a diffusion bridge can be used to define an overdamped Langevin SPDE on the path space of the diffusion \cite[Theorem 1.1]{ReVa2005}.   In the context of a $d$-dimensional diffusion bridge, the simplest way to do this is by defining an energy functional  \begin{equation} \label{eq:energy_functional}
\mathcal{E}(u) = \int_0^S \left[ \frac{1}{2} | \partial_s u(s) |^2 + G(u(s)) \right] ds \;,
\end{equation}  such that $\Pi$ in \eqref{eq:density_of_diffusion_bridge} can be written as  \[
 \Pi(u) = \exp\left( -\frac{\beta}{2} \mathcal{E}(u) \right) \;.
\]
Then an overdamped Langevin SPDE whose invariant distribution is the law of the diffusion bridge $\mathsf{X}$ is given by \begin{equation} \label{eq:overdamped_langevin_spde}
\begin{aligned}
\partial_t u(t,s) &= -  \dfrac{\delta \mathcal{E}}{\delta u}( u(t,s)) ds + 2 \sqrt{\beta^{-1}} \partial_t W(t,s) \;,  \\
u(t,0) &= x^- \;, \quad u(t,S)=x^+ \;, \quad u(0,s) = u_0(s) \;, 
\end{aligned}
\end{equation}  for all $(t,s) \in [0, \infty) \times [0,S]$.  
Here $u_0$ is an initial path with endpoints at $x^-$ and $x^+$, $W$ is a $d$-dimensional space-time, cylindrical Wiener process, and $\delta \mathcal{E}/\delta u$ is the functional derivative of $\mathcal{E}$.  This functional derivative equals \[
 \frac{\delta \mathcal{E}}{\delta u}(u) = - \partial_s^2 u + \nabla G(u)
\]  
since for smooth functions $\delta u : [0,S] \to \mathbb{R}^d$ that vanish at the endpoints
\begin{align*}
\frac{d}{d \epsilon}  \mathcal{E}(u + \epsilon \delta u) \bigg|_{\epsilon=0} &= \int_0^S \left[ \partial_s u \cdot   \partial_s \delta u + \nabla G(u) \cdot \delta u \right] ds \\
&=  \partial_s u \cdot  \delta u  \bigg|_{s=0}^{s=S} + \int_0^S \left( - \partial_s^2 u + \nabla G(u) \right) \cdot \delta u ds  \\
&= \int_0^S \frac{\delta \mathcal{E}}{\delta u} \cdot \delta u ds
\end{align*} where in the last step we used the endpoint conditions $\delta u(0) = \delta u(S)=0$.
Note that the variable $s$, which played the role of a time variable in the diffusion bridge, represents a spatial variable in \eqref{eq:overdamped_langevin_spde}, 
while the variable $t$ represents a time variable in \eqref{eq:overdamped_langevin_spde}.   To be sure, as $t \to \infty$ the law of the solution to \eqref{eq:overdamped_langevin_spde} 
tends to the law of the diffusion bridge $\mathsf{X}$ defined in \S\ref{sec:diffusion_bridge_definition}.

For the purpose of constructing numerical approximations,  it helps to transform \eqref{eq:overdamped_langevin_spde} into an SPDE with homogeneous Dirichlet boundary conditions.
In particular, numerical approximations of these transformed equations are easier to construct and better behaved than directly approximating \eqref{eq:overdamped_langevin_spde} and then imposing the inhomogeneous Dirichlet boundary conditions in \eqref{eq:overdamped_langevin_spde}.  To this end, let  \begin{equation} \label{eq:psi} 
\psi(s)=x^- \frac{S-s}{S} + x^+ \frac{s}{S} \;, \quad s \in [0,S] \;.
\end{equation} Then $u^*(t,s) = u(t,s) - \psi(s)$ satisfies \begin{equation} \label{eq:first_order_spde}
\begin{aligned}
\partial_t u^*(t,s) &=  -  \dfrac{\delta \mathcal{E}}{\delta u}( u^*(t,s) + \psi(s) ) dt + 2 \sqrt{\beta^{-1}} \partial_t W(t,s) \;,   \\
u^*(t,0) &= 0 \;, \quad u^*(t,S)= 0 \;, \quad u^*(0,s) = u_0(s) - \psi(s) \;,
\end{aligned}
\end{equation}  
 for all $(t,s) \in [0, \infty) \times [0,S]$. Note that the Dirichlet boundary conditions on $u^*$ in \eqref{eq:first_order_spde} are homogeneous.

\subsection{Second-Order, Semilinear Langevin SPDE}  \label{sec:second_order_langevin}

Here we present a second-order Langevin SPDE whose marginal invariant measure is the law of this diffusion bridge.  We choose the marginal distribution in the other component to be the law of a spatial Gaussian white noise \cite{NeVa2016}. Specifically, a second-order Langevin SPDE whose invariant distribution is this product distribution is given by \begin{equation} \label{eq:underdamped_langevin_spde}
\begin{aligned}
\partial_t u(t,s) &= p(t,s) \;, \\
\partial_t p(t,s) &= -  \dfrac{\delta \mathcal{E}}{\delta u}( u(t,s)) dt - \gamma p(t,s) dt + 2 \sqrt{\gamma \beta^{-1}} \partial_t W(t,s) \;,  \\
u(t,0) &= x^- \;, \quad u(t,S)=x^+ \;, \\
 u(0,s) &= u_0(s) \;, \quad p(0,s) = p_0(s) \;, 
\end{aligned}
\end{equation}
where $u_0$ and $p_0$ are initial position and momentum paths respectively, $W$ is a $d$-dimensional space-time, cylindrical Wiener process, $\gamma>0$ is a friction coefficient, $\mathcal{E}$ is the (potential) energy functional defined in \eqref{eq:energy_functional}, and as before, $(t,s) \in [0, \infty) \times [0,S]$.  
The dynamics of this SPDE preserves the measure with density proportional to \begin{equation}
 \exp\left( - \frac{\beta}{2} \left[  \frac{1}{2} \int_0^S \abs{p(s)}^2 ds \right] \right)   \Pi( u )
\end{equation}
where $\Pi(u)$ is the density of the diffusion bridge defined in \eqref{eq:density_of_diffusion_bridge}.
Additionally, by using the coordinate transformation $u^*(t,s) = u(t,s) - \psi(s)$ where $\psi$ is given in \eqref{eq:psi} and $p^*(t,s) = p(t,s)$, we can turn this SPDE into one with homogeneous Dirichlet boundary conditions
\begin{equation} \label{eq:second_order_spde}
\begin{aligned}
\partial_t u^*(t,s) &= p^*(t,s) \;, \\
\partial_t p^*(t,s) &= -  \dfrac{\delta \mathcal{E}}{\delta u}( u^*(t,s) + \psi(s) ) dt \\
& \qquad - \gamma p^*(t,s) dt + 2 \sqrt{\gamma \beta^{-1}} \partial_t W(t,s) \;,  \\
u^*(t,0) &=  0 \;, \quad u^*(t,S)=0 \;, \\
 u^*(0,s) &= u_0(s) - \psi(s) \;, \quad p^*(0,s) = p_0(s) \;.
\end{aligned}
\end{equation}
A special case of \eqref{eq:second_order_spde} is when $\gamma=0$.  In this case, the noise and friction vanish, and the second-order Langevin SPDE in \eqref{eq:second_order_spde} reduces to a semilinear, Hamiltonian PDE \begin{equation} \label{eq:hamiltonian_pde}
\begin{aligned}
\partial_t u^*(t,s) &= p^*(t,s) \;, \\
\partial_t p^*(t,s) &= -  \dfrac{\delta \mathcal{E}}{\delta u}( u^*(t,s) + \psi(s) ) \;, \\
u^*(t,0) &=  0 \;, \quad u^*(t,S)=0 \;, \\
 u^*(0,s) &= u_0(s) - \psi(s) \;, \quad p^*(0,s) = p_0(s) \;,
\end{aligned}
\end{equation}
with associated Hamiltonian functional \begin{equation} \label{eq:hamiltonian_functional}
\mathcal{H}(u,p) = \mathcal{E}( u + \psi ) + \frac{1}{2} \int_0^S |p(s)|^2 ds \;.
\end{equation}

\subsection{Semidiscrete Approximations}  \label{sec:semi_discrete}

Here we construct semi-discrete (continuous time, discrete space) approximations of \eqref{eq:first_order_spde} and \eqref{eq:second_order_spde} 
by using finite difference methods.  Other finite-dimensional truncations are possible including pseudospectral, finite-element, Galerkin and finite-volume methods.   For simplicity's sake, we consider only uniform grids.  More precisely, we use the uniform grid shown in Figure~\ref{fig:grid} where the interval $[0,S]$ is discretized using the evenly-spaced grid consisting of $n+1$ grid points defined in \eqref{eq:uniform_grid}  with step size $\Delta s = S/n$.

For each $1 \le \ell \le d$, approximate the $\ell$-th component of $W$ by a truncation \[
\boldsymbol{W}_{\ell}(t,s) = \sum_{k=1}^n \beta^{\ell}_k(t) e_k(s)
\] where $\{ \beta_k^{\ell} \}_{k=1}^n$ are $n$ i.i.d.~standard Brownian motions and $\{ e_k \}_{k=1}^n$ are the leading $n$ orthonormal eigenfunctions of the second derivative endowed with homogeneous Dirichlet boundary conditions \[
e_k(s) = \sqrt{\frac{2}{S}} \sin\left( \frac{k \pi s}{S} \right) \;,  \quad 1 \le k \le n \;.
\]  On the grid \eqref{eq:uniform_grid}, we obtain \[
\mathrm{E} \left\{ \boldsymbol{W}_{\ell}(t,s_i) \boldsymbol{W}_{\ell}(t, s_j) \right\} = \frac{t}{\Delta s} \delta_{ij} 
\] by Lagrange's trigonometric identities.  Thus, $\{ \boldsymbol{W}_{\ell}(t, s_i) \}_{i=1}^n$ are independent Brownian motions each with variance $1/\Delta s$.

Additionally, we discretize the second derivatives in $s$ appearing in \eqref{eq:first_order_spde} and \eqref{eq:second_order_spde} using a central difference method \begin{equation}
 \partial_s^2 f_i \approx \frac{f_{i+1} - 2 f_i + f_{i-1}}{\Delta s^2}   
 \end{equation}
where we use the shorthand notation $f_i = f(s_i)$ for $0 \le i \le n$.

\paragraph{Approximating SDE for First-Order, Semilinear Langevin SPDE}

For any $s \ge 0$ and for any $i \in \{0, \cdots, n \}$, let $\boldsymbol{u}_{i}(t) \approx u^*(t,s_i) $ denote a semi-discrete approximation of the solution to \eqref{eq:first_order_spde} on the grid \eqref{eq:uniform_grid}.    Because of the Dirichlet boundary conditions in \eqref{eq:first_order_spde},  note that \[
\boldsymbol{u}_0(t) = 0 \;, \quad \text{and} \quad \boldsymbol{u}_n(t) = 0 \;, \quad \forall t \ge 0 
 \] and hence, there are only $n-1$ unknown variables: $\boldsymbol{u}_1(t), \cdots, \boldsymbol{u}_{n-1}(t)$.   Using this spatial discretization, we obtain the following approximating SDEs \begin{equation} \label{eq:semidiscrete_first_order}
 d \boldsymbol{u}(t) = ( \boldsymbol{L}  \boldsymbol{u}(t) + F( \boldsymbol{u}(t)) ) dt + \frac{2 \sqrt{\beta^{-1}}}{\sqrt{\Delta s}} d \boldsymbol{W}(t)
 \end{equation}
where we have introduced: 
\begin{itemize}
\item
 $d (n-1)$-dimensional vector of unknown functions \[
\boldsymbol{u} = (\boldsymbol{u}_1, \cdots, \boldsymbol{u}_{n-1}) \;,
\]
obtained by stacking $n-1$, $d$-vectors one above the other,
\item
 $d (n-1)$-dimensional Brownian motion \[
\boldsymbol{W} =  ( \boldsymbol{W}_1, \cdots,  \boldsymbol{W}_{n-1})  \;,
\] 
\item
  $d (n-1) \times d (n-1)$ sparse, symmetric matrix  \begin{equation} \label{eq:matrix_L}
 \boldsymbol{L}_{ij} = \begin{dcases}  \dfrac{-2}{\Delta s^2} & \text{if $|i-j|=0$} \;, \\
 \dfrac{1}{\Delta s^2} & \text{if $|i-j|=d$} \;, \end{dcases} 
 \end{equation}
 \item
and,  $d (n-1)$-dimensional vector field
\begin{equation} \label{eq:vector_field_F}
  \boldsymbol{F}( \boldsymbol{u}) := \left(- \nabla G( \boldsymbol{u}_1 + \psi(s_1)), \cdots, - \nabla G( \boldsymbol{u}_{n-1} + \psi(s_{n-1})) \right) \;.
\end{equation} 
\end{itemize}
 Note that \eqref{eq:semidiscrete_first_order} is first-order Langevin dynamics with invariant density proportional to \begin{equation} \label{eq:approx_stationary_density_u}
\pi( \boldsymbol{u} ) = \exp\left( - \frac{\beta}{2} \Delta s \left( \sum_{ i=1}^{n-1} G(\boldsymbol{u}_i + \psi(s_i)) - \dfrac{1}{2} \boldsymbol{u}^T \boldsymbol{L} \boldsymbol{u} \right) \right) \;.
 \end{equation}
Up to a normalizing constant, $\pi$ can be viewed as a finite-dimensional approximation of the density of the diffusion bridge $\mathsf{X}$  defined in \S\ref{sec:diffusion_bridge_definition}.


\paragraph{Approximating SDE for Second-Order, Semilinear Langevin SPDE}

To construct a finite difference approximation, we again use the uniform grid in \eqref{eq:uniform_grid}.  The analog of \eqref{eq:semidiscrete_first_order} for the second-order equations in \eqref{eq:second_order_spde} is given by the following  approximating SDE \begin{equation} \label{eq:semidiscrete_second_order}
\begin{aligned}
d \boldsymbol{u}(t) &= \boldsymbol{p}(t) dt \;, \\
 d\boldsymbol{p}(t) &= ( \boldsymbol{L} \boldsymbol{u}(t) + \boldsymbol{F}( \boldsymbol{u}(t) ) ) dt - \gamma \boldsymbol{p}(t) dt + \frac{2 \sqrt{\beta^{-1} \gamma}}{\sqrt{\Delta s}} d \boldsymbol{W}(t) \;.
\end{aligned}
\end{equation}
where $\boldsymbol{L}$ is the $d (n-1) \times d (n-1)$ sparse, symmetric matrix defined in \eqref{eq:matrix_L}, 
$ \boldsymbol{F}$ is  the $d (n-1)$-dimensional vector field defined in \eqref{eq:vector_field_F}, $\boldsymbol{W}$ is a $d (n-1)$-dimensional Brownian motion, 
and $(\boldsymbol{u}(s), \boldsymbol{p}(s))$ are the random vector-valued position and momentum processes respectively.  These processes are written in components as 
\begin{align*}
\boldsymbol{u}(t)&=(\boldsymbol{u}_1(t), \cdots, \boldsymbol{u}_{n-1}(t)) \in \mathbb{R}^{d (n-1)} \;,  \\
\boldsymbol{p}(t)&=(\boldsymbol{p}_1(t), \cdots, \boldsymbol{p}_{n-1}(t))\in \mathbb{R}^{d (n-1)} \;.
\end{align*}
 Note that \eqref{eq:semidiscrete_second_order} is second-order Langevin dynamics at temperature $2 \beta^{-1}/ \Delta s$ and with invariant density proportional to \begin{equation} \label{eq:approx_stationary_density}
\pi( \boldsymbol{u}, \boldsymbol{p}) = \exp\left( - \frac{\beta}{2} \Delta s H(\boldsymbol{u}, \boldsymbol{p}) \right) \;,
 \end{equation}
 where we have introduced the Hamiltonian function \begin{equation} \label{eq:hamiltonian_underdamped}
 H(\boldsymbol{u}, \boldsymbol{p}) = \dfrac{1}{2} \| \boldsymbol{p} \|^2 +  \sum_{ i=1}^{n-1} G(\boldsymbol{u}_i + \psi(t_i)) - \dfrac{1}{2} \boldsymbol{u} \cdot \boldsymbol{L} \boldsymbol{u} \;.
 \end{equation}
Up to a normalization constant, the $\boldsymbol{u}$-marginal of $\pi$ is a finite-dimensional approximation of the law of the diffusion bridge $\mathsf{X}$  defined in \S\ref{sec:diffusion_bridge_definition}.



\subsection{Time Discretizations}  \label{sec:cayley_splitting}

In this part, we present time discretizations for \eqref{eq:semidiscrete_first_order} and \eqref{eq:semidiscrete_second_order}.

\paragraph{Crank-Nicolson for First-Order, Semilinear Langevin SPDEs}

Here we give a Metropolized integrator for \eqref{eq:semidiscrete_first_order}.   The idea for this method comes from MCMC and numerical SDE theory \cite{RoTw1996A, RoTw1996B, BeRoStVo2008, BoVa2010,BoVa2012,BoHa2013,BoDoVa2014, Bo2014, Fa2014, FaHoSt2015}, and basically combines an explicit time integrator for this approximating SDE with the Metropolis-Hastings method. The purpose of the latter is to eliminate the bias in the invariant measure introduced by time discretization error.   In other words, Metropolis-Hastings allows one to set the invariant distribution of the integrator to be the one with density given in \eqref{eq:approx_stationary_density}.  More generally, the Metropolis-Hastings method is a general purpose tool for producing samples from an absolutely continuous distribution (known as the target distribution) given a function proportional to its density (correspondingly known as the target density) \cite{MeRoRoTeTe1953,Ha1970}. The method generates a Markov chain from a given proposal Markov chain as follows. A proposal move is computed according to the proposal chain and then accepted with a probability that ensures the Metropolized chain preserves the target distribution.

Here we shall focus on the Metropolis-Hastings method with target density given by \eqref{eq:approx_stationary_density_u} and with proposal move given by discretizing in time the approximating SDE in \eqref{eq:semidiscrete_first_order} using a $\theta$-method \begin{equation} \label{eq:theta_method}
 \boldsymbol{A}_{\theta} \boldsymbol{u}' =  \boldsymbol{B}_{\theta} \boldsymbol{u} +   \Delta t  \boldsymbol{F}(\boldsymbol{u}) + 2 \sqrt{\beta^{-1}} \sqrt{\dfrac{\Delta t}{\Delta s}} \boldsymbol{\xi}^0
\end{equation}
where $\theta \in [0,1]$ is a parameter, $\boldsymbol{\xi}^0 \in \mathcal{N}(0,1)^{d (n-1)}$ is a $d (n-1)$-dimensional standard Gaussian vector, and $\boldsymbol{A}_{\theta}, \boldsymbol{B}_{\theta}$ are the following $d (n-1) \times d (n-1)$ matrices \[
 \boldsymbol{A}_{\theta} = (\boldsymbol{I} - \theta \Delta t \boldsymbol{L})  \qquad \boldsymbol{B}_{\theta} =  (\boldsymbol{I} + (1- \theta) \Delta t \boldsymbol{L})  \;,
\]  
and $\boldsymbol{I}$ is the $d (n-1) \times d (n-1)$ identity matrix.  By change of variables, the transition density of this $\theta$-integrator is given by  \begin{equation} \label{eq:transition_density}
q(\boldsymbol{u},\boldsymbol{u}') = \left( \frac{8 \pi \Delta t}{\beta \Delta s} \right)^{-(n-1) d /2} \exp\left( - \frac{\beta \Delta s}{8 \Delta t} | \boldsymbol{A}_{\theta} y - \boldsymbol{B}_{\theta} \boldsymbol{u} - \Delta t \boldsymbol{F}(\boldsymbol{u})  |^2 \right) 
\end{equation} 
where $\boldsymbol{u},\boldsymbol{u}' \in \mathbb{R}^{d (n-1)}$.

\begin{algorithm}[Non-Preconditioned MALA] \label{algo:MALA}
Given a spatial step size $\Delta s$, a time step size $\Delta t$, a parameter $\theta \in [0,1]$ and the state $\boldsymbol{u}^0$ at time $t$, the algorithm calculates an updated state $\boldsymbol{u}^1$ at time $t+\Delta t$ in two steps: 
\begin{description}
\item[(Step 1)] 
compute a proposal move $\boldsymbol{\tilde u}^1$ using \eqref{eq:theta_method} with input $\boldsymbol{u}=\boldsymbol{u}^0$ and output $\boldsymbol{\tilde u}^1 = \boldsymbol{u}'$; \medskip
\item[(Step 2)] 
accept or reject the proposal move $\boldsymbol{\tilde u}^1$ by 
taking as actual update for the state at time $t+\Delta t$ 
\begin{equation}  \label{actualupdate}
\boldsymbol{u}^1 = \gamma \boldsymbol{\tilde u}^1 + (1-\gamma) \boldsymbol{u}^0
\end{equation}
where $\gamma$ is a Bernoulli random variable with parameter $\alpha(\boldsymbol{u}^0,  \boldsymbol{\tilde u}^1)$ and $\alpha$ is the acceptance probability function defined as \begin{equation} \label{eq:mala_alpha}
\alpha(\boldsymbol{u},\boldsymbol{u}') = 1 \wedge \frac{q(\boldsymbol{u}',\boldsymbol{u}) \pi(\boldsymbol{u}')}{q(\boldsymbol{u},\boldsymbol{u}') \pi(\boldsymbol{u})} \;,  \quad \boldsymbol{u},\boldsymbol{u}' \in \mathbb{R}^{d (n-1)} \;.
\end{equation}
To be sure, $q$ is the transition density of the proposal move given in \eqref{eq:transition_density} and $\pi$ is the (not necessarily normalized) density given in \eqref{eq:approx_stationary_density_u}.
\end{description}
\end{algorithm}

This Metropolized $\theta$ integrator is commonly known as the Metropolis-Adjusted Langevin Algorithm (MALA) \cite{RoTw1996A, RoTw1996B, BeRoStVo2008, BoVa2010}. Since MALA is a Metropolis-Hastings method, it immediately follows that it preserves the invariant density of the approximating SDE in \eqref{eq:semidiscrete_first_order}. The choice $\theta=0$ corresponds to an Euler-Maruyama time discretization of the approximating SDE in \eqref{eq:semidiscrete_first_order}.  The choice $\theta=1/2$ corresponds to a Crank-Nicholson time discretization.  The results in Ref.~\cite{BeRoStVo2008} state that unless one chooses $\theta=1/2$ (a Crank-Nicholson time discretization), the convergence of the algorithm is mesh-dependent in the sense that the acceptance rate of the Metropolized $\theta$-scheme depends strongly on the (spatial) step size $\Delta t$.  For this reason, we will only consider MALA with $\theta=1/2$ for the rest of this paper.

For the special case $\theta=1/2$, a direct calculation using \eqref{eq:approx_stationary_density_u} and \eqref{eq:transition_density} shows that $\alpha$ in \eqref{eq:mala_alpha} is given explicitly by \begin{equation} \label{eq:pmala_alpha}
\begin{aligned}
&\alpha(\boldsymbol{u},\boldsymbol{u}') = 1 \wedge \exp\left( -  \frac{\beta}{2}  \Delta s \left( \mathcal{G}(\boldsymbol{u}') - \mathcal{G}(\boldsymbol{u}) - \frac{1}{2} \langle \boldsymbol{u}' - \boldsymbol{u}, \nabla \mathcal{G}(\boldsymbol{u}) + \nabla \mathcal{G}(\boldsymbol{u}') \rangle \right. \right. \\
& \qquad   \left. \left. + \frac{\Delta t}{4}  \left( | \nabla \mathcal{G}(\boldsymbol{u}') |^2 -  | \nabla \mathcal{G}(\boldsymbol{u})|^2 \right)   - \frac{\Delta t}{4} \langle \boldsymbol{L} (\boldsymbol{u}'+\boldsymbol{u}), \nabla \mathcal{G}(\boldsymbol{u}') - \nabla \mathcal{G}(\boldsymbol{u}) \rangle \right) \right) 
\end{aligned}
\end{equation}
where $\mathcal{G}(\boldsymbol{u}) := \sum_{i=1}^{n-1} G(\boldsymbol{u}_i + \psi(t_i))$ and $\boldsymbol{u} = (\boldsymbol{u}_1, \cdots, \boldsymbol{u}_{n-1}) \in \mathbb{R}^{d (n-1)}$.  Note that if $\boldsymbol{L}=0$ one recovers the acceptance probability function given in, e.g., Lemma~4.7 of Ref.~\cite{BoVa2010}.

\paragraph{Cayley Splitting for Second-Order, Semilinear Langevin SPDEs}

As described in \S\ref{sec:main_results}, we split the dynamics in \eqref{eq:semidiscrete_second_order} into two parts: Hamilton's equations for the Hamiltonian in \eqref{eq:hamiltonian_underdamped}
\[
\tag{H} d \boldsymbol{u}(t) = \boldsymbol{p}(t) dt \;, \quad d  \boldsymbol{p}(t) = \boldsymbol{L} \boldsymbol{u}(t) dt + \boldsymbol{F}( \boldsymbol{u}(t) ) dt \;,
\]
and an Ornstein-Uhlenbeck equation in momentum
\[
\tag{O} d \boldsymbol{u}(t) = 0 \;, \quad d \boldsymbol{p}(t) =  - \gamma \boldsymbol{p}(t) dt + \frac{2 \sqrt{\beta^{-1} \gamma}}{\sqrt{\Delta s}} d \boldsymbol{W}(t)  \;.
\]
To evolve (H) over a time step of length $\Delta t$, we use the Cayley splitting defined in \eqref{eq:cayley_splitting_hamiltonian} and which we denote by $\phi_{\Delta \tau}^{(H)}$. To evolve (O), we use the exact solution (in law) of the Ornstein-Uhlenbeck equations $\varphi_{\Delta \tau}^{(O)}: (\boldsymbol{u}^0, \boldsymbol{v}^0) \mapsto (\boldsymbol{u}^0, \boldsymbol{v}^1)$ where $\boldsymbol{v}^1$ is defined in a distributional sense by \begin{equation} \label{eq:O}
 \boldsymbol{v}^1 \overset{d}{=} e^{-\gamma \Delta \tau}  \boldsymbol{v}^0 + \frac{\sqrt{2 \beta^{-1}}}{\sqrt{\Delta s}}  \sqrt{1-e^{-2 \gamma \Delta \tau}} \boldsymbol{\xi}^0 
\end{equation} and $\boldsymbol{\xi}^0  \sim \mathcal{N}(0,1)^{n (d-1)}$.
To obtain an approximate flow map $\varphi_{\Delta \tau}^{(L)}$ for \eqref{eq:semidiscrete_second_order}, we combine these maps in a palindromic way \begin{equation} \label{eq:L}
\varphi_{\Delta \tau}^{(L)} = \varphi_{(1/2) \Delta \tau}^{(O)} \circ \phi_{\Delta \tau}^{(H)}  \circ \varphi_{(1/2) \Delta \tau}^{(O)} 
\end{equation} Any other palindromic splitting of these maps would have a similar order of accuracy (in a distributional sense), 
but the error constants may differ \cite{LeMaSt2015}.

\subsection{Non-Preconditioned MALA as an HMC Algorithm} \label{sec:cayley_splitting_mala}

Here we link non-preconditioned MALA given in Algorithm~\ref{algo:MALA} to the HMC algorithm, which we introduced in \S\ref{sec:cayley_based_hmc}.  This link generalizes the well-known relationship between MALA and HMC when $\boldsymbol{L}=0$.  First, we briefly recall the HMC method and then make this link in Prop.~\ref{prop:mala_hmc} below.

We set the target density of the HMC method to be the density given in \eqref{eq:approx_stationary_density_u}.    The HMC algorithm is defined in an extended position-velocity space, $\{  (\boldsymbol{u},\boldsymbol{p}) \in \mathbb{R}^{d (n-1)} \times \mathbb{R}^{d (n-1)} \}$, where positions $\boldsymbol{u} \in \mathbb{R}^{d (n-1)}$ belong to the domain of the target density, and velocities $\boldsymbol{p} \in \mathbb{R}^{d (n-1)}$ are an auxiliary variable.   On this extended space, an extended density $\nu_{\text{extended}}$ is introduced, which is a product of the target density and a Gaussian density in velocities, and such that the marginal density in the position component is the target density.  In its simplest form, the algorithm produces a Markov chain on $\mathbb{R}^{d (n-1)}$ with the required invariant density by using a volume-preserving and reversible map $\varphi: \mathbb{R}^{d (n-1)} \times  \mathbb{R}^{d (n-1)}   \to  \mathbb{R}^{d (n-1)} \times  \mathbb{R}^{d (n-1)} $  on the extended space and iterating the following update rule.

\begin{algorithm}[HMC] \label{algo:hmc}
The one-step $\boldsymbol{u}^0 \mapsto \boldsymbol{u}^1$ update is given by.
\begin{description}
\item[(Step 1)] 
Draw $\boldsymbol{p}^0$ from the measure $\nu_{\text{extended}}(\boldsymbol{u}^0, \boldsymbol{p}) d\boldsymbol{p}$.
\medskip
\item[(Step 2)] 
 Set $(\boldsymbol{\tilde u}^1, \boldsymbol{\tilde p}^1) = \varphi(\boldsymbol{u}^0, \boldsymbol{p}^0)$ where $\varphi$ is a volume-preserving and reversible map.
 \medskip
\item[(Step 3)] 
 Take as actual update: \[
\boldsymbol{u}^1 = \gamma \boldsymbol{\tilde u}^1 + (1-\gamma) \boldsymbol{u}^0
\] where $\gamma$ is a Bernoulli random variable with parameter: \[
\alpha(\boldsymbol{u}^0, \boldsymbol{p}^0) = 1 \wedge \frac{\nu_{\text{extended}}(\boldsymbol{\tilde u}^1, \boldsymbol{\tilde p}^1)  }{\nu_{\text{extended}}(\boldsymbol{u}^0, \boldsymbol{p}^0) } \;.
\]
\end{description}
\end{algorithm}

The above description of the HMC algorithm is standard; see, e.g., the description given in \cite[Section 9]{Sa2014} or \cite[Section II]{FaSaSk2014}.   Note that the updated velocity is discarded by the algorithm.

\begin{prop} \label{prop:mala_hmc}
MALA -- Algorithm~\ref{algo:MALA} with $\theta=1/2$ -- is an HMC algorithm with
\begin{description}
\item[(i)] the extended density defined by  \[
\nu_{\text{extended}}(\boldsymbol{u},\boldsymbol{p}) = Z^{-1}  \exp\left( -  \frac{\beta}{2}  \Delta s H(\boldsymbol{u},\boldsymbol{p}) \right)
\]
where $Z$ is a normalization constant, $H$ is a Hamiltonian function defined as \begin{equation} \label{eq:hamiltonian}
H(\boldsymbol{u},\boldsymbol{p}) = \frac{1}{2} \| \boldsymbol{p} \|^2  -\frac{1}{2} \boldsymbol{u} \cdot \boldsymbol{L} \boldsymbol{u}  + \mathcal{G}(\boldsymbol{u}) 
\end{equation} and $\mathcal{G}$ is a potential energy function defined as \[
\mathcal{G}(\boldsymbol{u}) = \sum_{i=1}^{n-1} G(\boldsymbol{u}_i + \psi(s_i)) \;, \quad \boldsymbol{u} = (\boldsymbol{u}_1, \cdots, \boldsymbol{u}_{n-1}) \in \mathbb{R}^{d (n-1)} \;,
\]
where $G$ is the path potential function given in \eqref{eq:path_potential}.
\item[(ii)] the volume-preserving and reversible map $\varphi_{\Delta \tau}: ( \boldsymbol{u}^0,  \boldsymbol{p}^0) \mapsto ( \boldsymbol{u}^1,  \boldsymbol{p}^1)$ defined by \begin{equation} \label{eq:strang_splitting}
\begin{aligned} 
 \boldsymbol{u}^1 &= \left( \boldsymbol{I}- \frac{\Delta \tau^2}{4} \boldsymbol{L} \right)^{-1} \left( \left(\boldsymbol{I} + \frac{\Delta \tau^2}{4} \boldsymbol{L} \right)  \boldsymbol{u}^0 + \Delta \tau  \boldsymbol{p}^0 - \frac{\Delta \tau^2}{2}  \nabla \mathcal{G}( \boldsymbol{u}^0) \right)  \\
 \boldsymbol{p}^1 &=  \left( \boldsymbol{I}- \frac{\Delta \tau^2}{4} \boldsymbol{L} \right)^{-1} \left( \left(\boldsymbol{I} + \frac{\Delta \tau^2}{4} \boldsymbol{L} \right) \left( \boldsymbol{p}^0 - \frac{\Delta \tau}{2}  \nabla \mathcal{G}( \boldsymbol{u}^0) \right) + \Delta \tau \boldsymbol{L} \boldsymbol{u}^0 \right) \\
 & \quad - \frac{\Delta \tau}{2} \nabla \mathcal{G}(\boldsymbol{u}^1)
 \end{aligned}
\end{equation}
with $\Delta \tau = \sqrt{2 \Delta t}$.  
\end{description}
\end{prop}

By casting MALA as an HMC algorithm, it immediately follows that it preserves the invariant density \eqref{eq:approx_stationary_density_u}.

\begin{proof}
With $\Delta \tau = \sqrt{2 \Delta t}$, and since the initial momentum satisfies $\boldsymbol{p}^0 \sim \nu_{\text{extended}}(\boldsymbol{u}^0, \boldsymbol{p}) d\boldsymbol{p}$, the position component in \eqref{eq:strang_splitting} is equal in law to the proposal move in non-preconditioned MALA \eqref{eq:theta_method} with $\theta=1/2$.  Moreover, the $\boldsymbol{u}$-marginal of $\nu_{\text{extended}}(\boldsymbol{u},\boldsymbol{p})$ in Prop.~\ref{prop:mala_hmc} is the non-normalized target density of MALA and the HMC acceptance probability is \[
\alpha(\boldsymbol{u}^0, \boldsymbol{p}^0) = 1 \wedge \exp\left(  -  \frac{ \beta }{2}  \Delta s ( H(\boldsymbol{\tilde u}^1, \boldsymbol{\tilde p}^1)  - H(\boldsymbol{u}^0, \boldsymbol{p}^0) ) \right) \;.
\]  
Expanding this out in a straightforward calculation yields the acceptance probability of MALA given in \eqref{eq:pmala_alpha}.  It remains to show that the map $\varphi_{\Delta \tau}$ defined in \eqref{eq:strang_splitting} is volume-preserving and reversible. To prove this, we simply show that  $\varphi_{\Delta \tau}$ is just the Cayley splitting method defined in \eqref{eq:cayley_splitting_hamiltonian}  applied to the Hamiltonian dynamics associated to \eqref{eq:hamiltonian}, i.e.,
\begin{equation} \label{eq:hamiltonian_dynamics}
\begin{cases}
 \boldsymbol{\dot u}(t)  = \boldsymbol{p}(t)  \\
 \boldsymbol{\dot{p}}(t) = \boldsymbol{L} \boldsymbol{u}(t) -\nabla \mathcal{G}( \boldsymbol{u}(t) )
\end{cases} \;.
\end{equation}
Specifically, we split \eqref{eq:hamiltonian_dynamics} into 
\[
\tag{A} \boldsymbol{\dot u}(t) = \boldsymbol{p}(t) \;, \quad \boldsymbol{\dot{p}}(t) = \boldsymbol{L} \boldsymbol{u}(t)  \;,
\]
and
\[
\tag{B}  \boldsymbol{\dot u}(t) = 0 \;, \quad \boldsymbol{\dot{p}}(t) = - \nabla \mathcal{G}( \boldsymbol{u}(t) ) \;.
\]
Instead of using the exact flow for (A), we use an approximation of the matrix exponential given by the  Cayley transform.  More precisely, we approximately evolve (A) over $[0, \Delta \tau]$ using the linear map $\phi_{\Delta \tau}^{(A)}: (\boldsymbol{u}^0, \boldsymbol{p}^0) \mapsto (\boldsymbol{u}^1, \boldsymbol{p}^1)$ defined by \begin{equation} \label{eq:A}
 \begin{pmatrix} \boldsymbol{u}^1 \\ \boldsymbol{p}^1 \end{pmatrix}  = \begin{pmatrix} \boldsymbol{I} & -\frac{\Delta \tau}{2} \boldsymbol{I} \\  - \frac{\Delta \tau}{2} \boldsymbol{L} & \boldsymbol{I} \end{pmatrix}^{-1} 
\begin{pmatrix} \boldsymbol{I} & \frac{\Delta \tau}{2} \boldsymbol{I} \\   \frac{\Delta \tau}{2} \boldsymbol{L} & \boldsymbol{I}  \end{pmatrix}      \begin{pmatrix} \boldsymbol{u}^0 \\ \boldsymbol{p}^0 \end{pmatrix} \;,
\end{equation}  
To evolve (B), we use the exact flow $\varphi_{\Delta \tau}^{(B)}: (\boldsymbol{u}^0, \boldsymbol{p}^0) \mapsto (\boldsymbol{u}^1, \boldsymbol{p}^1)$ for the Hamiltonian vector field in (B), which is explicitly given by \begin{equation} \label{eq:B}
\varphi_{\Delta \tau}^{(B)}(\boldsymbol{u}^0,\boldsymbol{p}^0) =  (\boldsymbol{u}^0, \boldsymbol{p}^0 - \Delta \tau \nabla \mathcal{G}(\boldsymbol{u}^0) ) \;.
\end{equation} 
It is then easy to check that \begin{equation} \label{eq:strang_splitting2}
\varphi_{\Delta \tau} = \varphi_{(1/2) \Delta \tau}^{(B)} \circ \phi_{\Delta \tau}^{(A)} \circ \varphi_{(1/2) \Delta \tau}^{(B)} \;.
\end{equation}
By Lemmas~\ref{lemma:cayley_symplectic} and~\ref{lemma:cayley_reversible}, the Cayley approximation is symplectic and reversible.  Since the composition of symplectic maps is again symplectic and the palindromic composition of reversible maps is reversible, the map $\varphi_{\Delta \tau}$ is a volume-preserving and reversible map.  This completes the proof.
\end{proof}


\subsection{Non-Preconditioned, Randomized HMC on Hilbert Spaces}  \label{sec:cayley_splitting_hmc}

To reduce correlations in the chain and improve convergence, one can integrate the Hamiltonian dynamics in \eqref{eq:hamiltonian_dynamics} for a longer duration than in MALA, which just takes one step of the Hamiltonian flow.  Since the exact flow preserves energy and volume, the numerical flow generated by iterating the Cayley splitting will be nearby an isocontour of the extended density, and hence, the resulting proposal move  is likely to be accepted.

However, increasing the duration of integration can backfire if periodicities (or near periodicities) in the Hamiltonian flow cause the Hamiltonian trajectory to make a U-turn and fold back on itself, thus increasing correlation in the chain.  To prevent this, Mackenzie in 1989 suggested duration randomization in the Hamiltonian flows in HMC \cite{Ma1989}.   More recently, in the exact integration scenario, a randomized HMC (RHMC) algorithm was introduced, and proved to be geometrically ergodic \cite{BoSa2015}.  In this section, we briefly review this algorithm in the context of sampling diffusion bridges.

The RHMC process $\boldsymbol{z}_t = (\boldsymbol{u}_t, \boldsymbol{p}_t)$ is a Piecewise Deterministic Markov Process (PDMP) on the state space $\mathbb{R}^{d (n-1)} \times \mathbb{R}^{d (n-1)}$ \cite{davis1984piecewise, Da1993}.   While Algorithm~\ref{algo:hmc} was formulated in $\mathbb{R}^{d (n-1)} $, the process $\boldsymbol{z}_t $ is defined in the enlarged space to include the possibility of partial randomization of the momentum, as in the generalized Hybrid Monte Carlo of Horowitz \cite{Ho1991, KePe2001, AkRe2008}.  This process $\boldsymbol{z}_t $ can be simulated by iterating the following steps. The mean duration $\lambda>0$ and the Horowitz angle $\varphi \in (0, \pi/2]$ are deterministic parameters.

\begin{algorithm}[Non-preconditioned RHMC] \label{algo:rhmc}
Given the current time $t_0$ and the current state $\boldsymbol{z}_{t_0} = (\boldsymbol{u}_{t_0}, \boldsymbol{p}_{t_0})$, the method computes the next momentum randomization time $t_1>t_0$ and the path of the process $\boldsymbol{z}_{t} = (\boldsymbol{u}_{t},\boldsymbol{p}_{t})$ over $(t_0, t_1]$ as follows.
\begin{description}
\item[(Step 1)] Update time via $t_1 = t_0 + \delta t_0$ where $ \delta t_0 \sim \Exp(1/\lambda)$.
\medskip
\item[(Step 2)] Evolve over $[t_0, t_1]$  Hamilton's equations \eqref{eq:hamiltonian_dynamics} with initial condition $(\boldsymbol{u}(t_0), \boldsymbol{p}(t_0))  = (\boldsymbol{u}_{t_0}, \boldsymbol{p}_{t_0})$.
\medskip
\item[(Step 3)] Set \[
\boldsymbol{z}_t = (\boldsymbol{u}_t, \boldsymbol{p}_t) = (\boldsymbol{u}(t), \boldsymbol{p}(t) ) \quad \text{for $t_0 \le t < t_1$} \;.
\]
\item[(Step 4)] Randomize momentum by setting \[
\boldsymbol{z}_{t_1} = (\boldsymbol{u}_{t_1}, \cos(\varphi) \boldsymbol{p}_{t_1} + \sin(\varphi) \boldsymbol{\xi} )
\] where $\boldsymbol{\xi} \sim \mathcal{N}(0,1)^{ n (d-1)}$.
\end{description}
\end{algorithm}

Note that the Hamiltonian dynamics in (Step 2) is non-preconditioned.

\subsection{Two-dimensional, Three-Well Potential Example}  \label{sec:three_well_example}

Consider a two-dimensional diffusion bridge \eqref{eq:diffusion_bridge} with the three-well potential energy function $V$ illustrated in the left panel of Figure~\ref{fig:three_well_example}. The associated path potential energy function in \eqref{eq:path_potential} is shown in the right panel of Figure~\ref{fig:three_well_example} with $\beta=2$.   The corresponding target density
is given in \eqref{eq:approx_stationary_density_u}. Here we use the HMC and RHMC algorithms using the Cayley and exact splittings given in \eqref{eq:cayley_splitting_hamiltonian} and \eqref{eq:exact_splitting_hamiltonian} respectively, to integrate the Hamiltonian dynamics in \eqref{eq:hamiltonian_dynamics}.  We stress that this dynamics is not preconditioned.  The initial path in position is taken to be a line connecting the bottom two wells located at $x^{\pm}\approx(\pm 1.048, -0.042)$. Figures~\ref{fig:three_well_example_hmc_ap} and~\ref{fig:three_well_example_rhmc_ap} plot the mean acceptance probabilities of these algorithms as a function of $\Delta t$ at three different spatial step sizes $\Delta s$ as indicated in the figure legends.  This figure illustrates how the artifacts we observed in the exact splitting in Figure~\ref{fig:nonlinear_hamiltonian_system_energy} can impair the performance of the HMC and RHMC algorithms based on the exact splitting. Figure~\ref{fig:three_well_example_means} and~\ref{fig:three_well_example_variances} plot the means (shifted by $\psi$ in \eqref{eq:psi}) and variance of each component of the target density $e^{-(\Delta s) \Phi_N(\boldsymbol{u})}$ using a (mean) duration leg of $T=2$, spatial domain size $S=1$, $\Delta s =0.02$, $\Delta t=0.03$, and $10^5$ samples.  The statistical errors are reported in the figure captions.

\begin{figure}
\begin{center}
\includegraphics[width=0.45\textwidth]{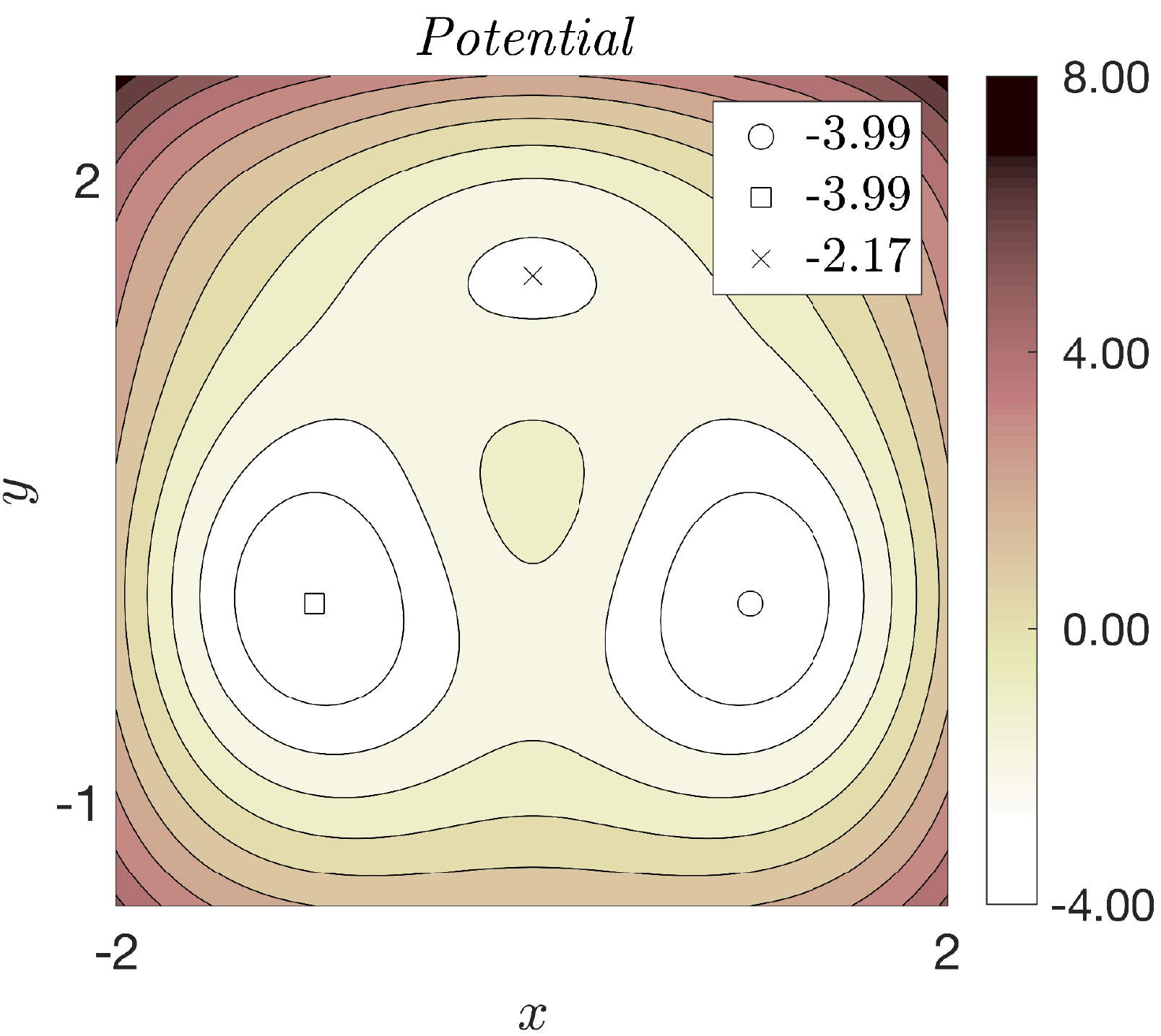}   \hspace{0.1in}
\includegraphics[width=0.45\textwidth]{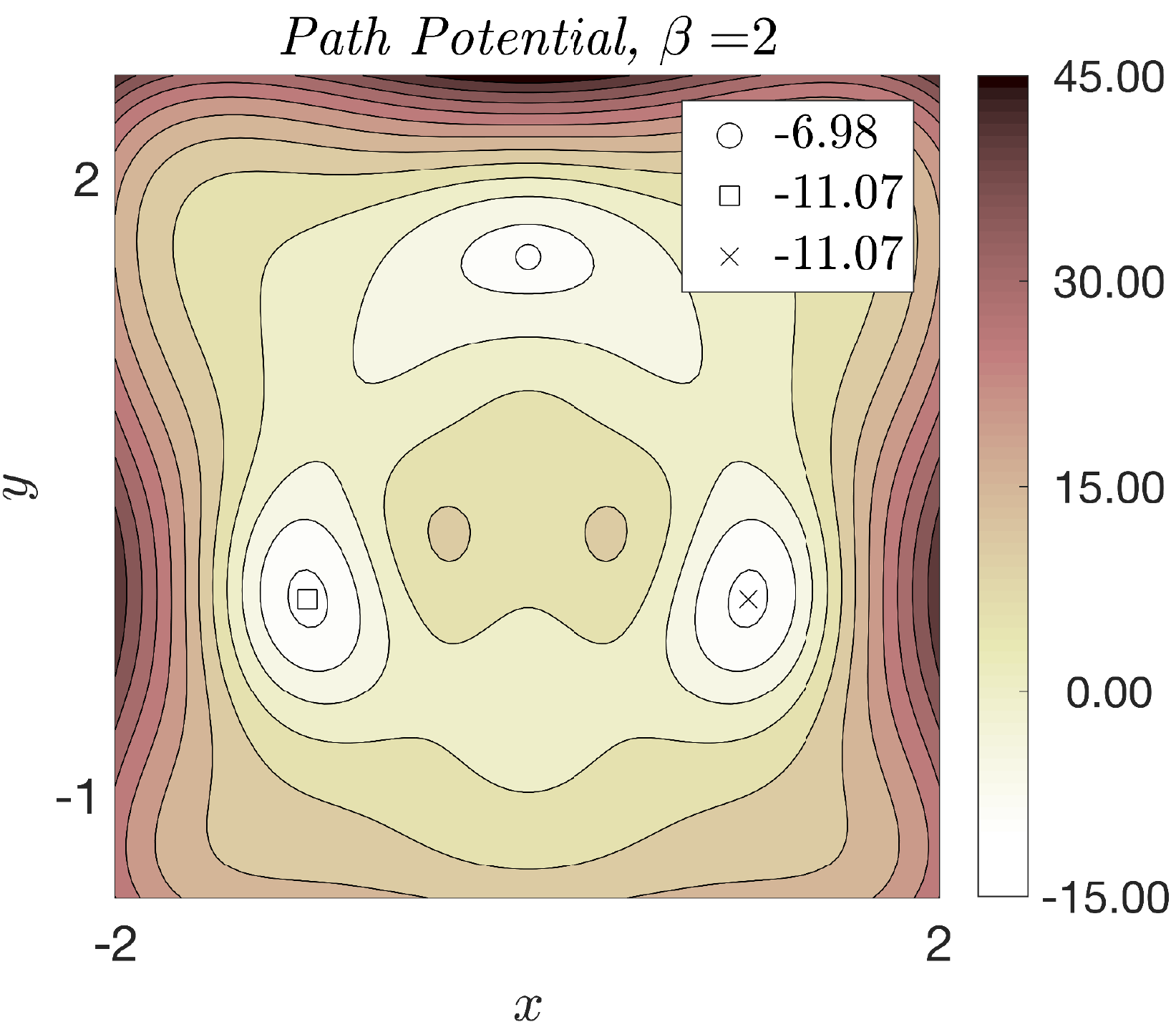} 
\end{center}
\caption{\small  {\bf Three-Well Example.}   
The left panel of this figure shows filled contour lines of a potential energy function $V$ of a diffusion bridge, which is described in \eqref{eq:diffusion_bridge}.
Alongside on the right panel are filled contour lines of its associated path potential energy function $G$, which is given in \eqref{eq:path_potential}.  
The legends give the value of $V$ and $G$ at the three minima of each function.  In our numerical tests we use the bottom two wells shown in the right panel.
}
  \label{fig:three_well_example}
\end{figure}

\begin{figure}
\begin{center}
\includegraphics[width=0.45\textwidth]{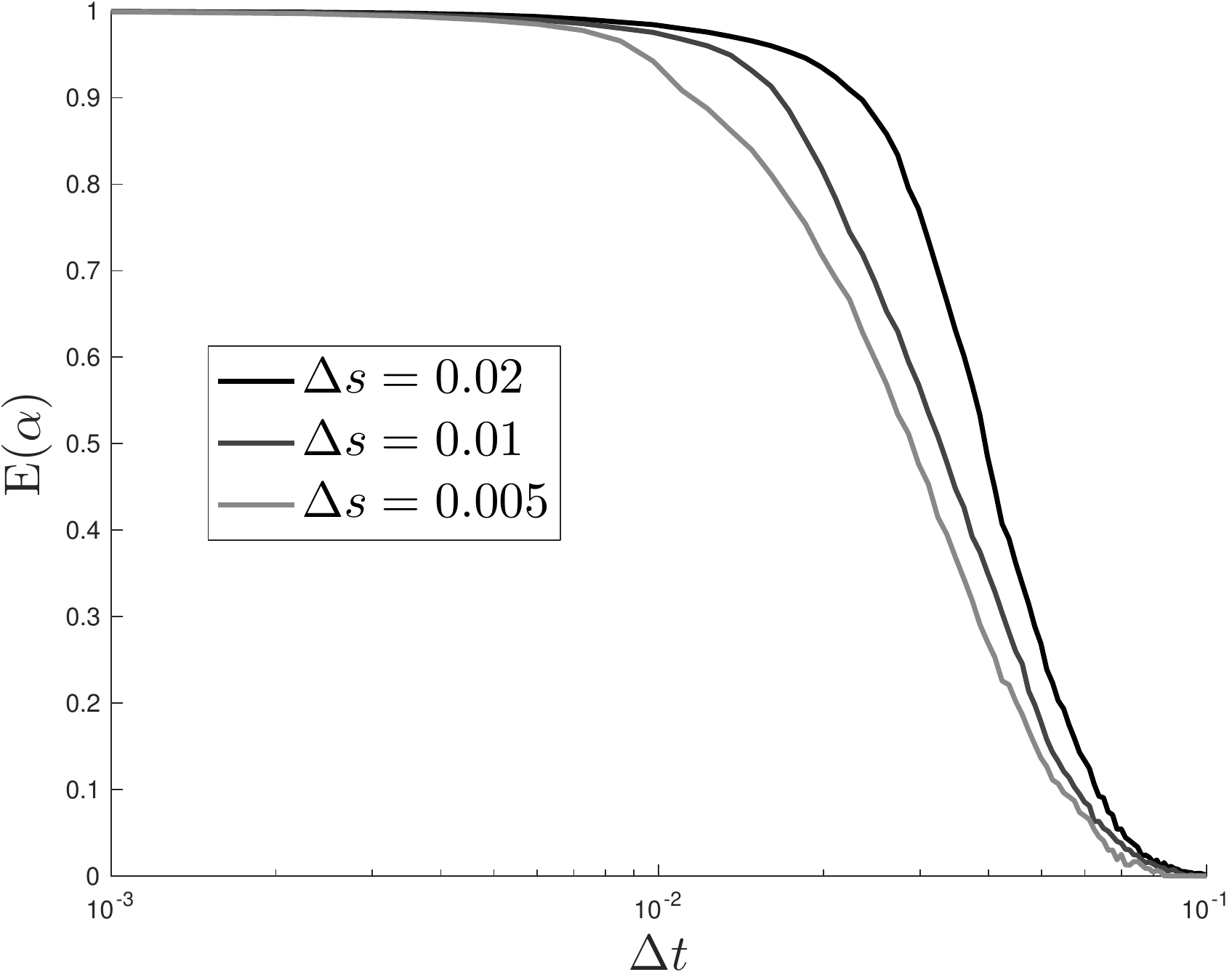}   \hspace{0.1in}
\includegraphics[width=0.45\textwidth]{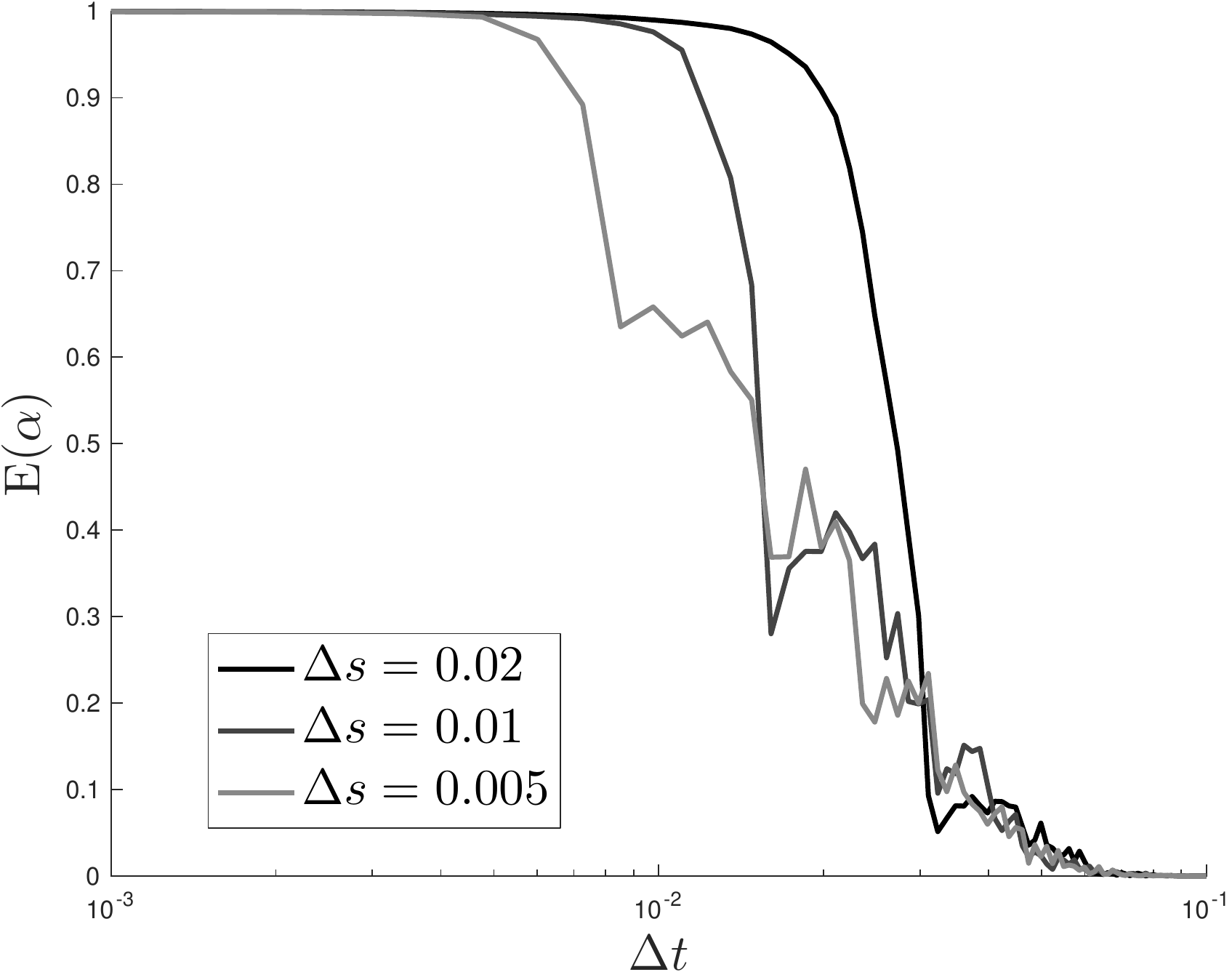} 
\end{center}
\caption{\small  {\bf Three-Well Example: HMC Acceptance Probability.}   This figure shows the mean acceptance probability
of the Cayley-based HMC (left panel) and the exact splitting based HMC (right panel).  The latter is more expensive
computationally (as discussed in \S\ref{sec:main_results}) and also prone to instabilities, which are reflected in the irregular 
behavior in the acceptance probabilities. The parameters used in the simulation are described in \S\ref{sec:three_well_example}.  
}
  \label{fig:three_well_example_hmc_ap}
\end{figure}

\begin{figure}
\begin{center}
\includegraphics[width=0.45\textwidth]{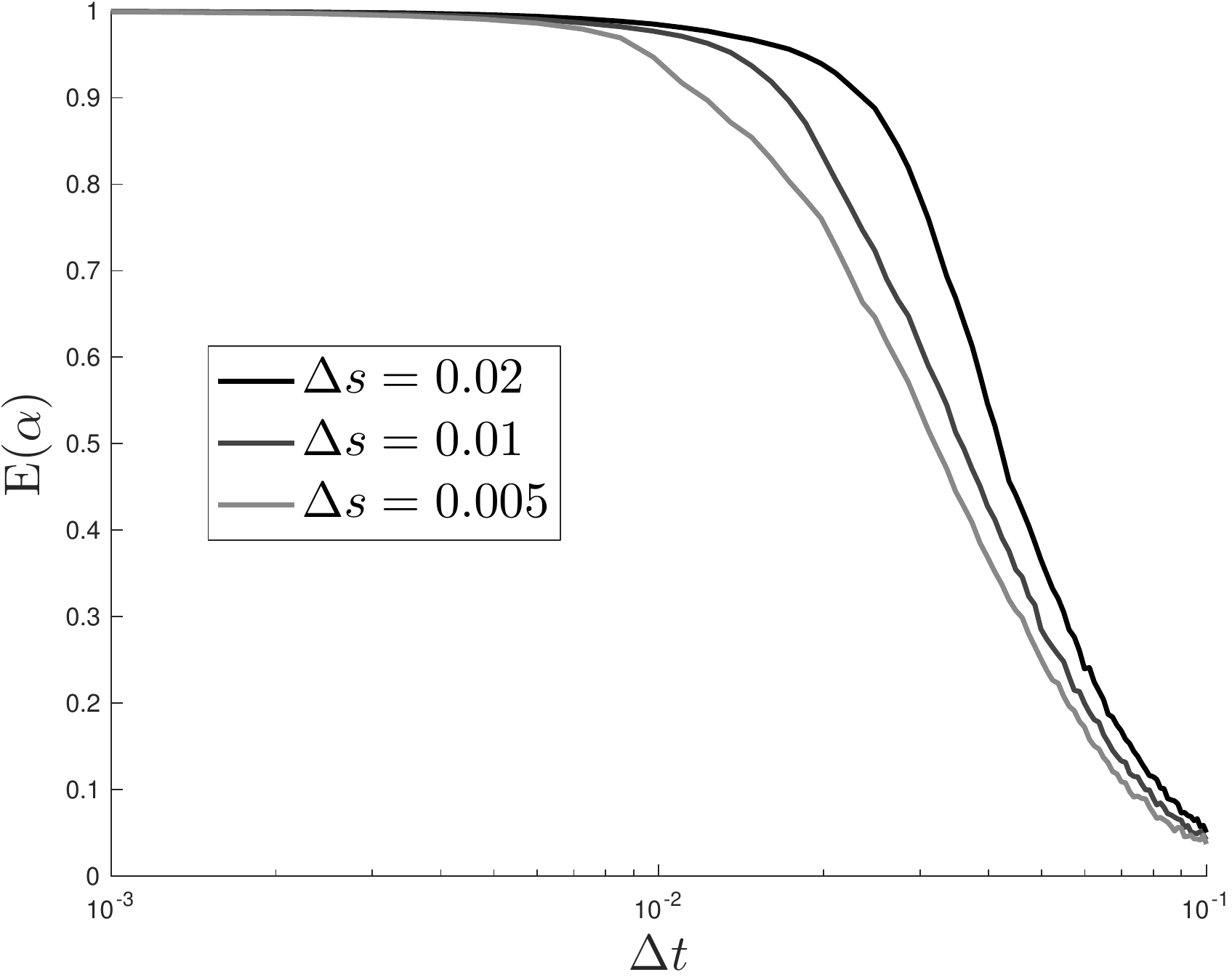}   \hspace{0.1in}
\includegraphics[width=0.45\textwidth]{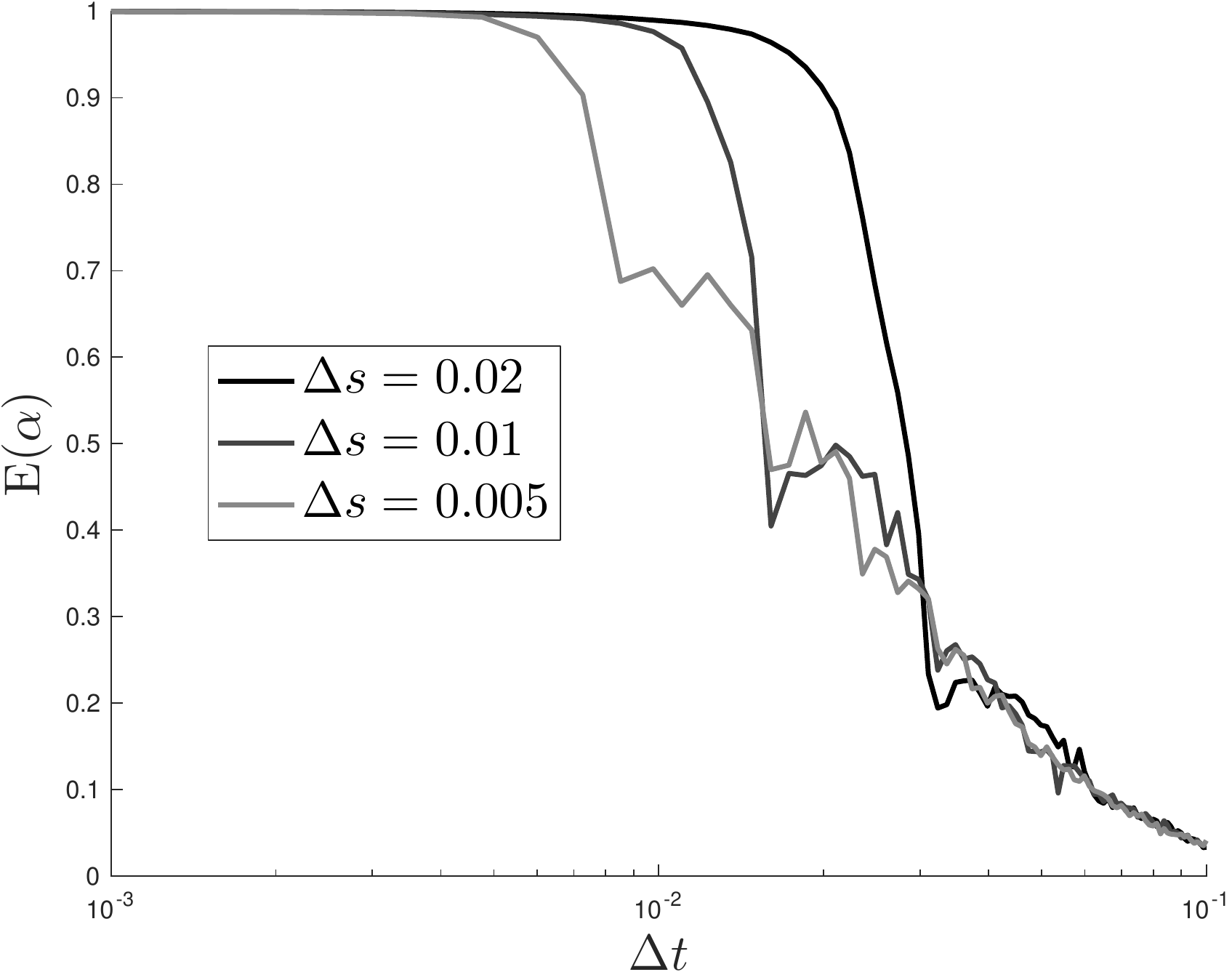} 
\end{center}
\caption{\small  {\bf Three-Well Example: RHMC Acceptance Probability.}   Same algorithms as the previous figure, but with randomized durations.
Specifically, this figure shows the mean acceptance probability of the Cayley-based RHMC (left panel) and the exact splitting based RHMC (right panel).  
The latter is more expensive computationally (as discussed in \S\ref{sec:main_results}) and also prone to instabilities, which are reflected in the irregular 
behavior in the acceptance probabilities.  The parameters used in the simulation are described in \S\ref{sec:three_well_example}.  
}
  \label{fig:three_well_example_rhmc_ap}
\end{figure}

\begin{figure}
\begin{center}
\includegraphics[width=0.45\textwidth]{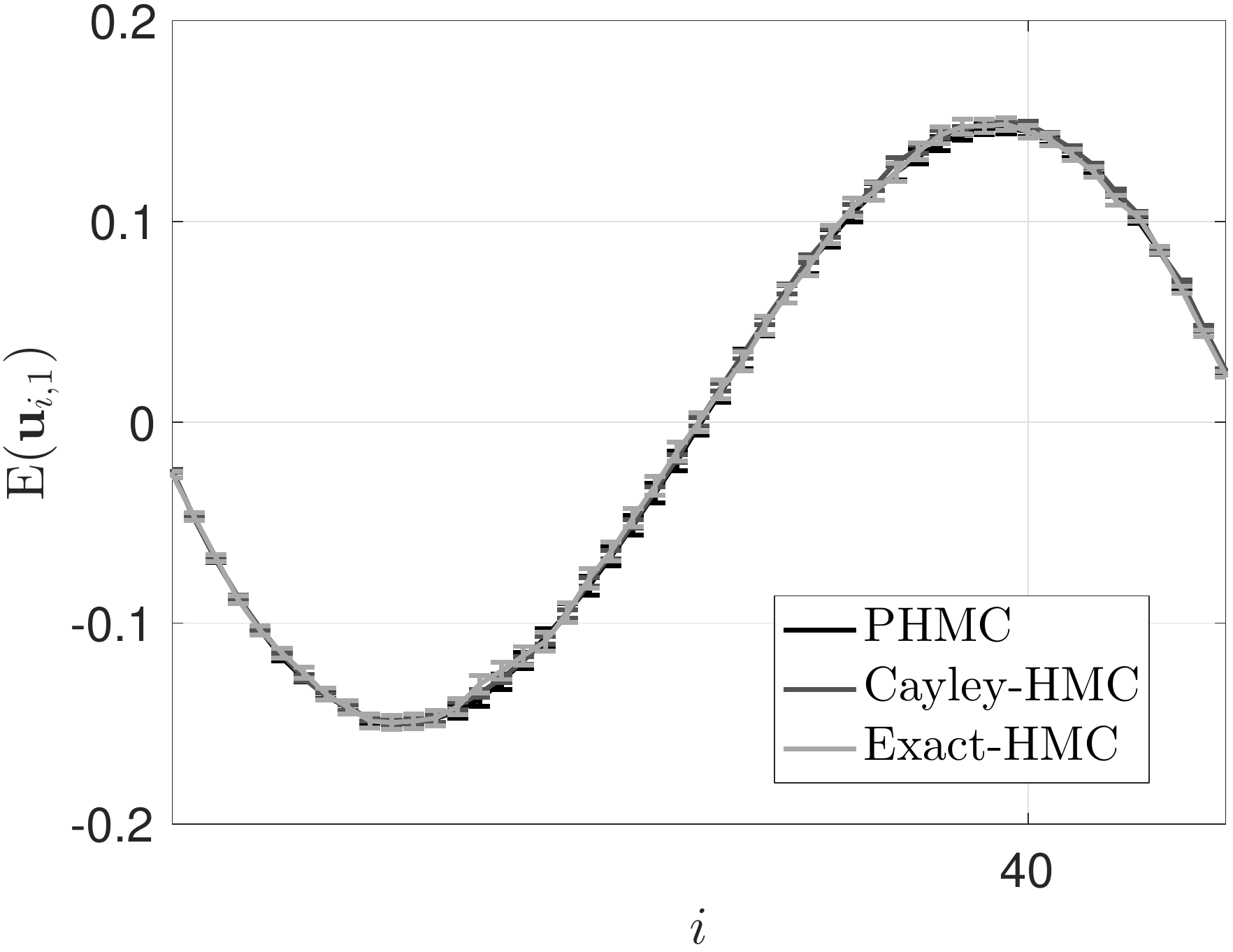}   \hspace{0.1in}
\includegraphics[width=0.45\textwidth]{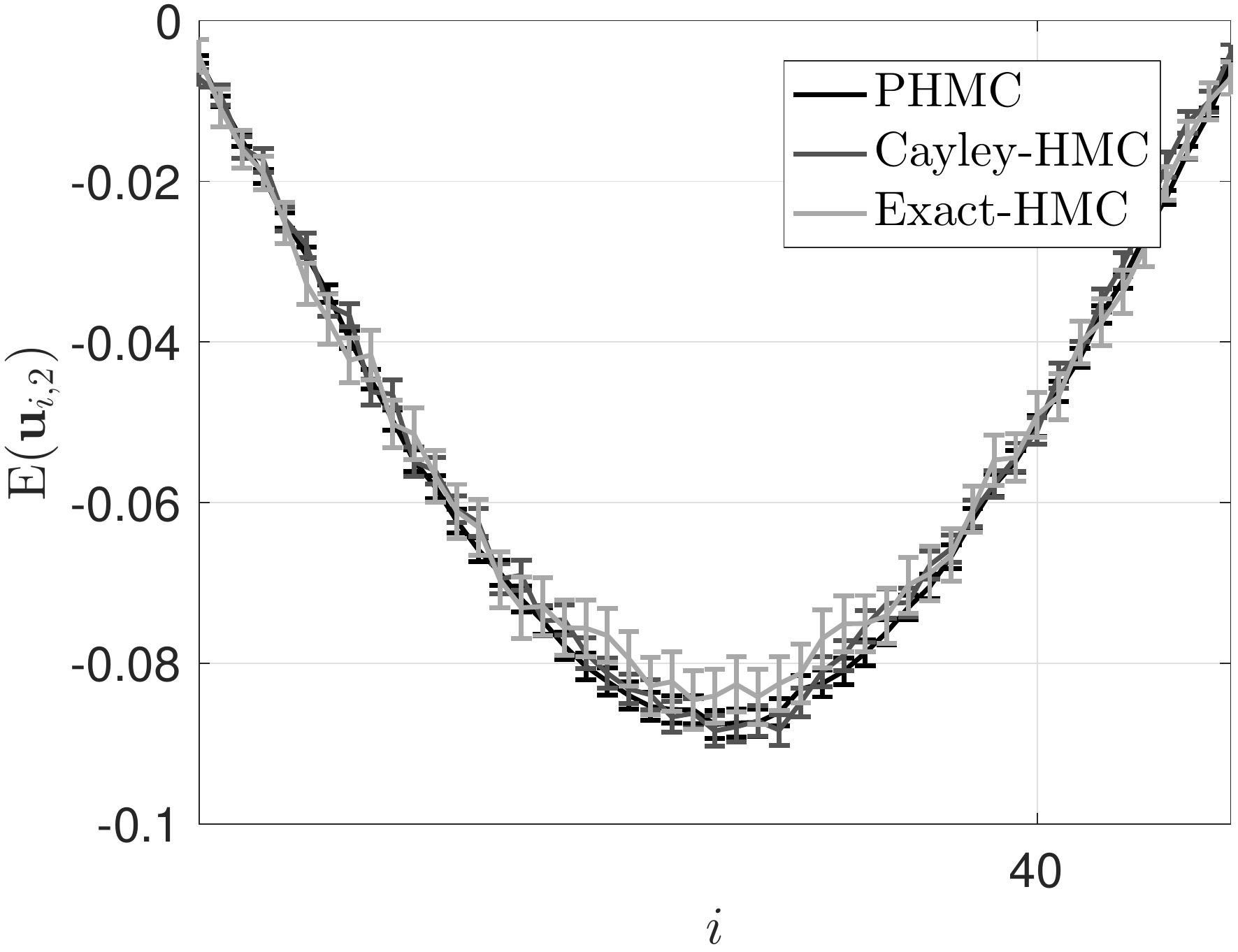} 
\includegraphics[width=0.45\textwidth]{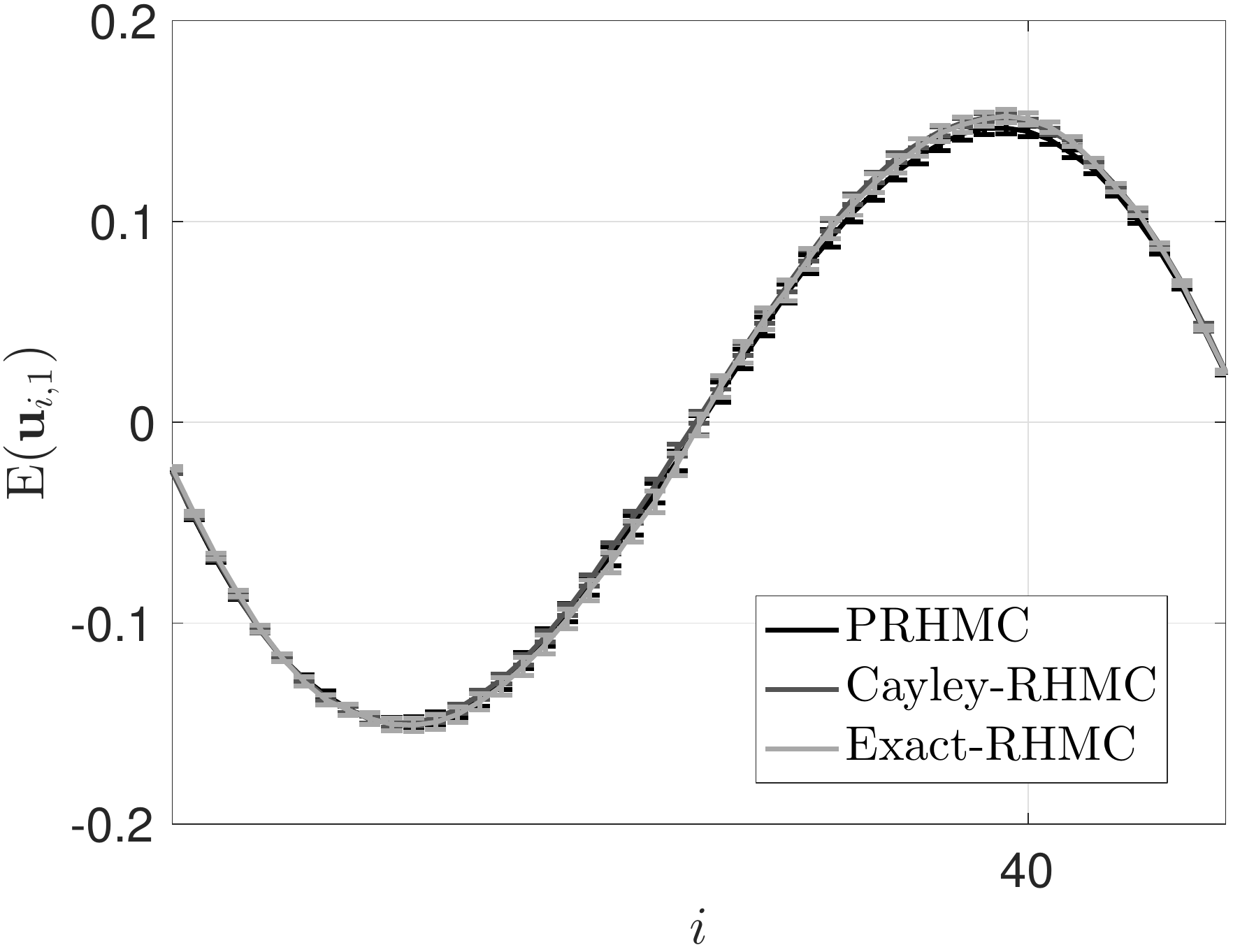}   \hspace{0.1in}
\includegraphics[width=0.45\textwidth]{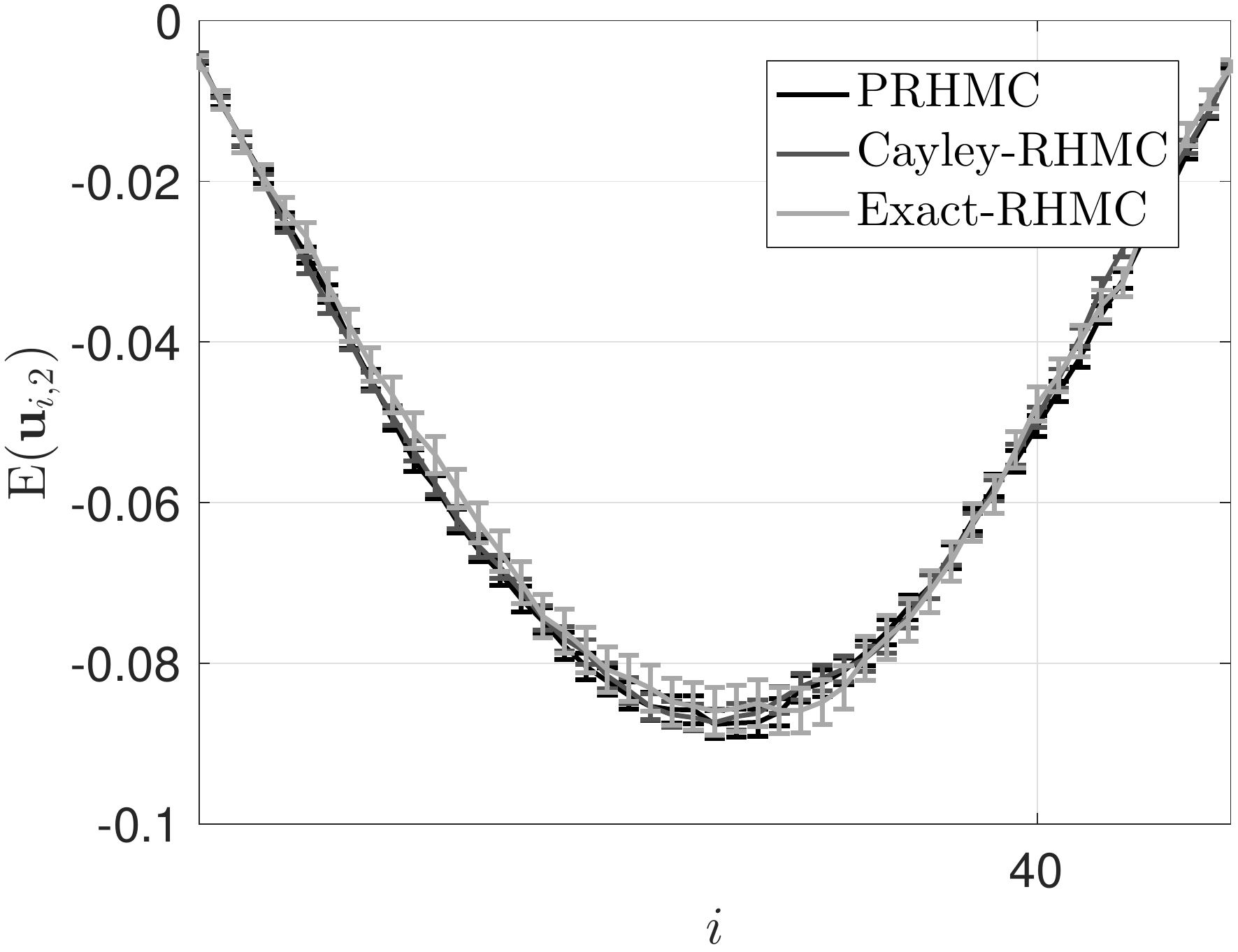} 
\end{center}
\caption{\small  {\bf Three-Well Example: Means of Target.}   This figure verifies the accuracy of HMC and RHMC algorithms in computing the variance in each component of the target density.  The $x$-axis labels the components.  The number of samples is $10^5$, the duration of the Hamiltonian legs is fixed at $T=2$, the spatial step size is $\Delta s=0.02$ (corresponding to $49$ interior grid points), the time step size is $\Delta t=0.03$.  As a benchmark for comparison, we use preconditioned HMC (PHMC) and preconditioned RHMC (PRHMC) algorithms \cite{BePiSaSt2011,BoSaActaN2018}.
 For the algorithms tested -- PHMC, Cayley and Exact -- the mean acceptance probabilities are approximately $99\%$, $78\%$, and $38\%$, respectively.  
  The corresponding sum of statistical errors in the mean of the first component are: $14\%$,  $9.5\%$, and $17\%$.
    The corresponding sum of statistical errors in the mean of the second component are: $6.5\%$,  $6.3\%$, and $11\%$. }
  \label{fig:three_well_example_means}
\end{figure}

\begin{figure}
\begin{center}
\includegraphics[width=0.45\textwidth]{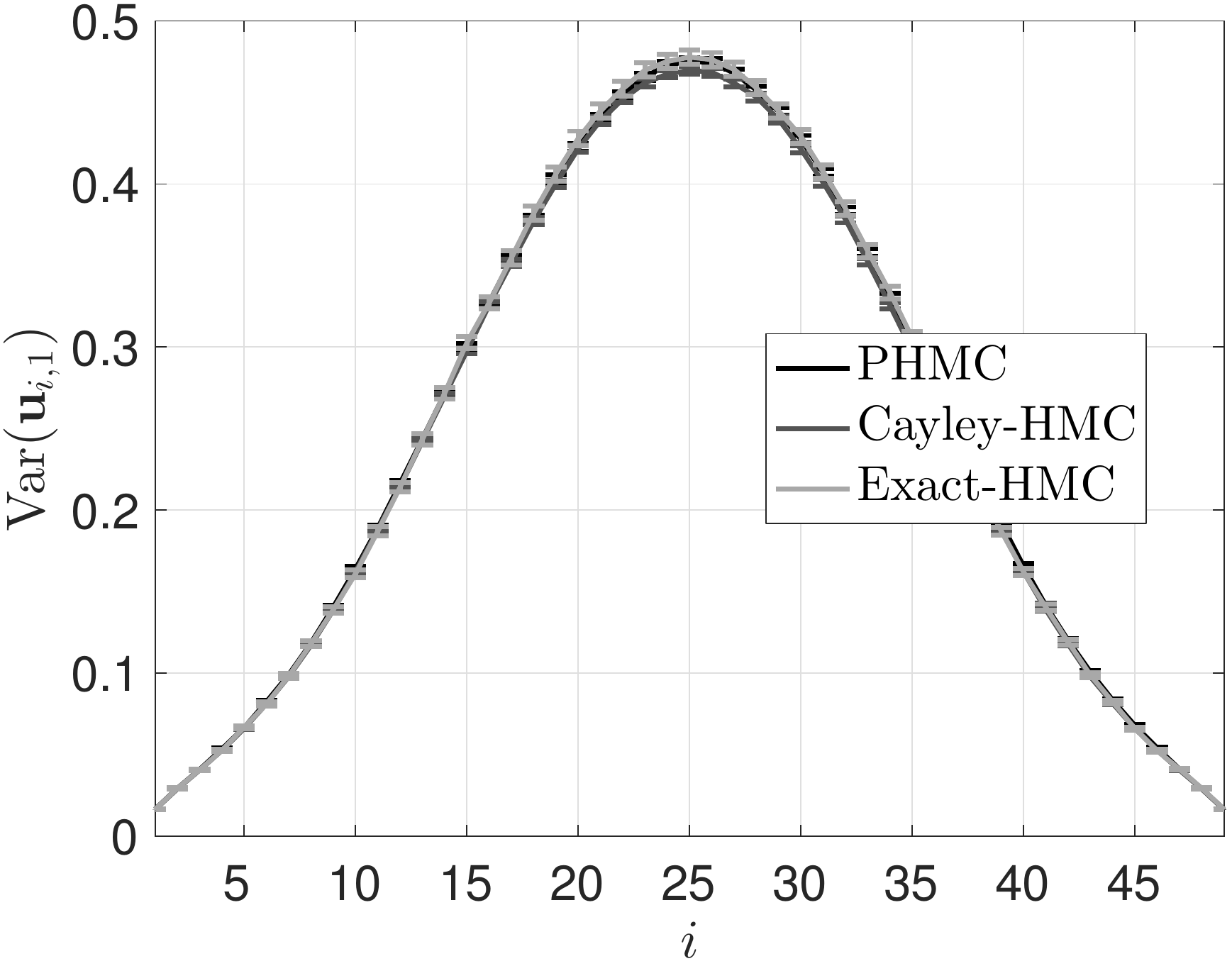}   \hspace{0.1in}
\includegraphics[width=0.45\textwidth]{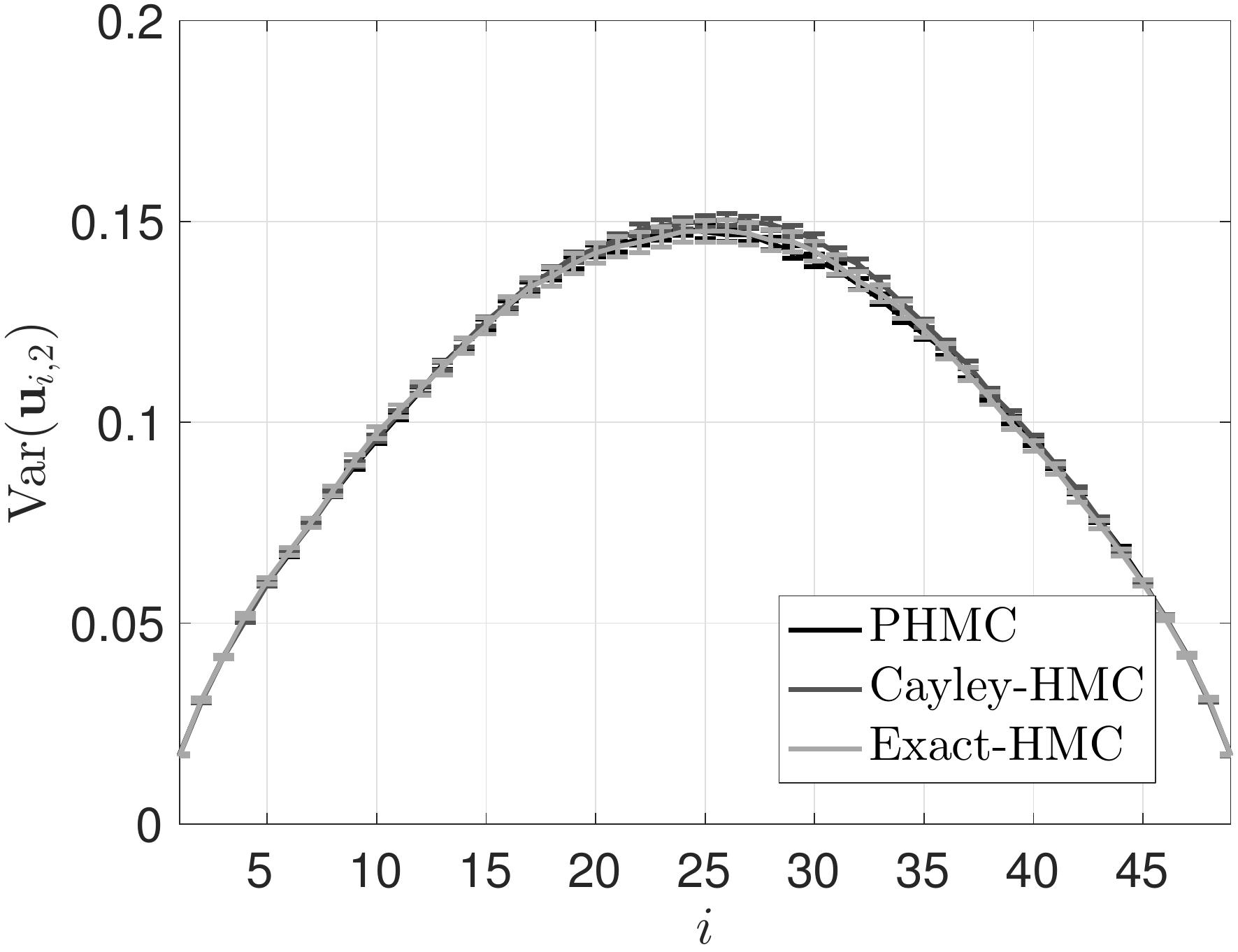} 
\includegraphics[width=0.45\textwidth]{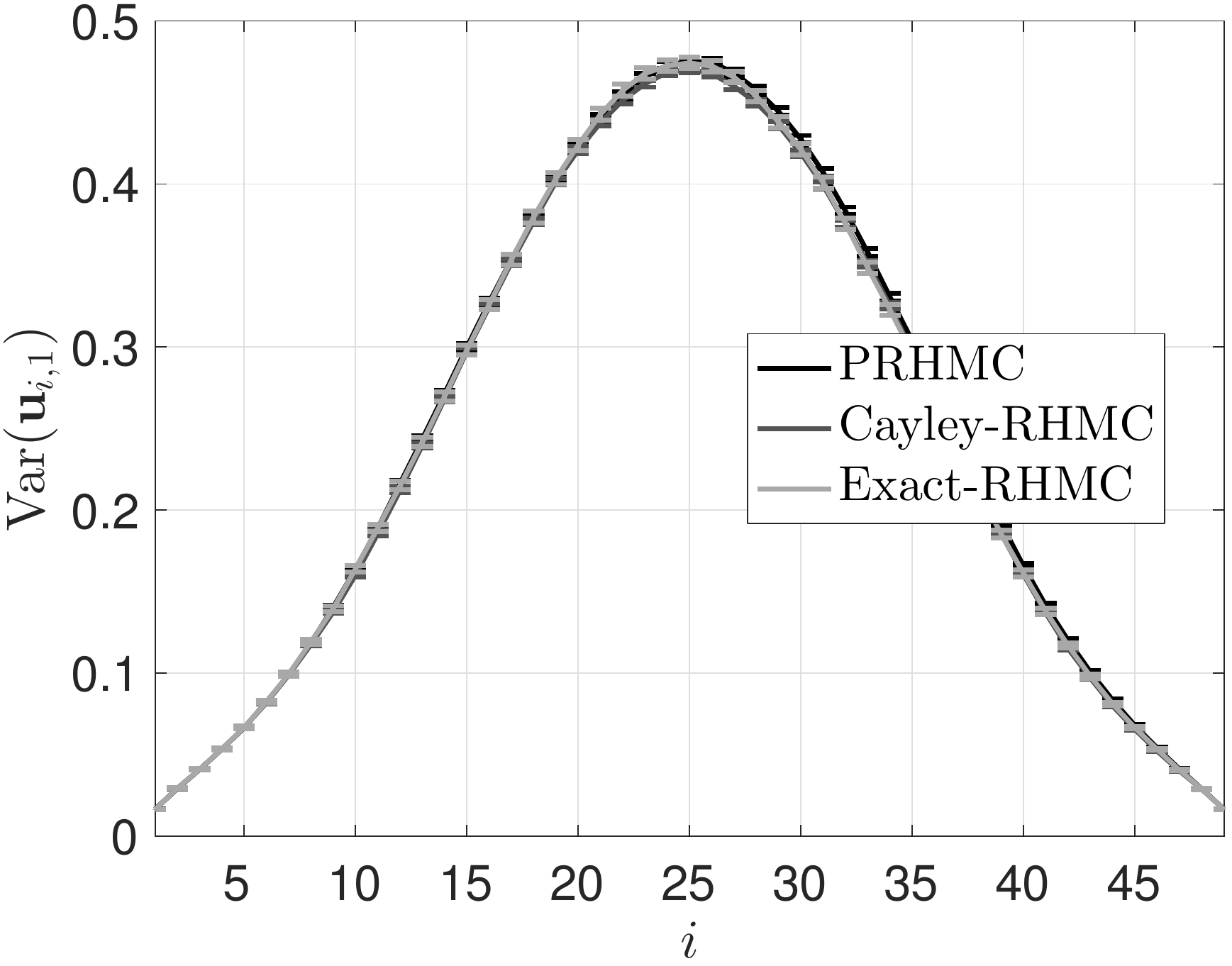}   \hspace{0.1in}
\includegraphics[width=0.45\textwidth]{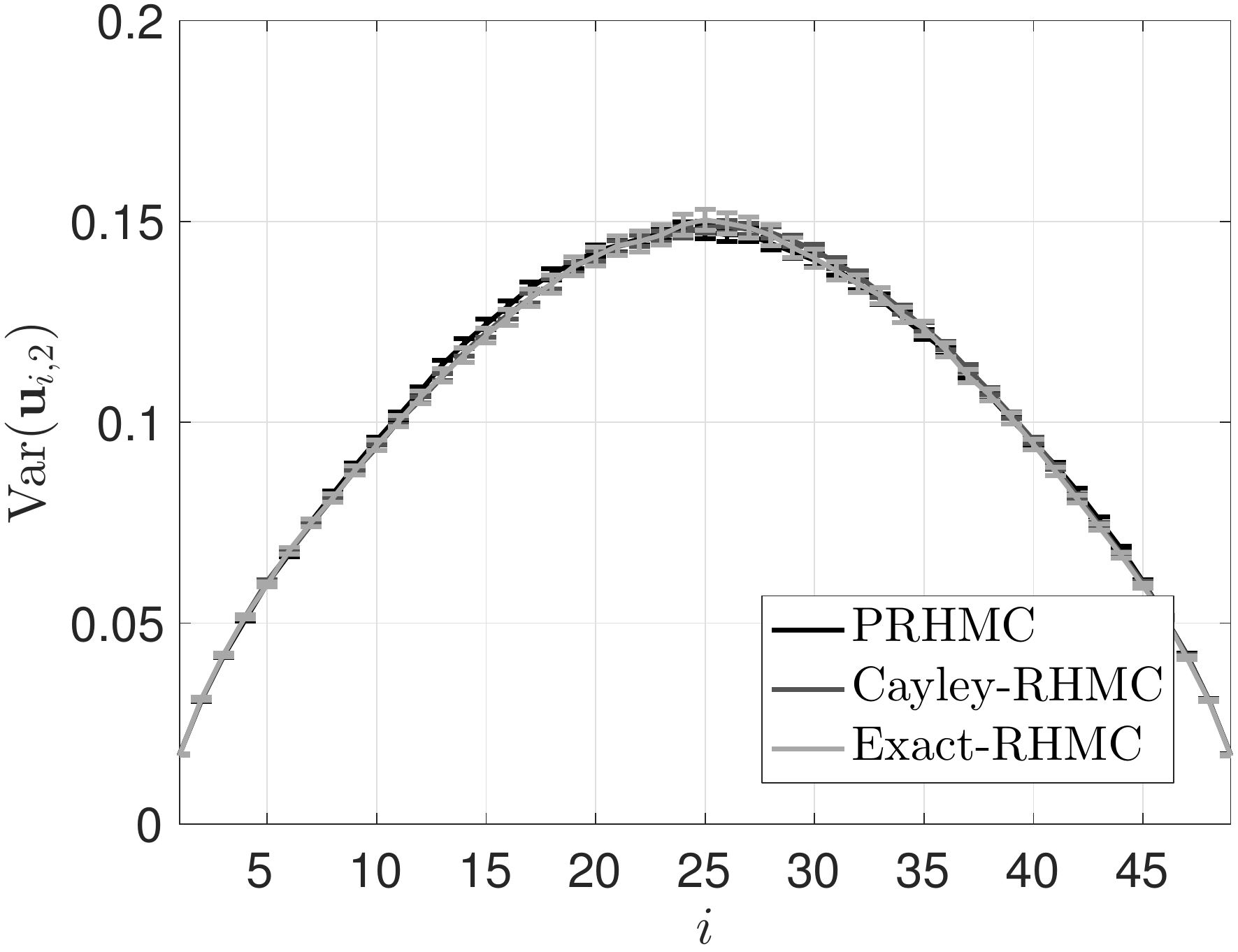} 
\end{center}
\caption{\small   {\bf Three-Well Example: Variances of Target.}   This figure verifies the accuracy of HMC and RHMC algorithms in computing the variance in each component of the target density. The $x$-axis labels the components.  The number of samples is $10^5$, the duration of the Hamiltonian legs is fixed at $T=2$, the spatial step size is $\Delta s=0.02$ (corresponding to $49$ interior grid points), the time step size is $\Delta t=0.03$. 
As a benchmark for comparison, we use preconditioned HMC (PHMC) and preconditioned RHMC (PRHMC) algorithms \cite{BePiSaSt2011,BoSaActaN2018}.
 For the algorithms tested -- PHMC, Cayley and Exact -- the mean acceptance probabilities are approximately $99\%$, $78\%$, and $38\%$, respectively.  
  The corresponding sum of statistical errors in the variance of the first component are: $7.8\%$,  $6.7\%$, and $11\%$.
    The corresponding sum of statistical errors in the variance of the second component are: $5.0\%$,  $4.3\%$, and $7.5\%$. 
}
  \label{fig:three_well_example_variances}
\end{figure}

\clearpage

\section{Conclusion}

The Cayley splitting \eqref{eq:cayley_splitting} generalizes the Crank-Nicolson method for first-order Langevin SPDEs to second-order Langevin SPDEs.   Similarly, the Cayley-based HMC algorithm is a generalization of the non-preconditioned MALA algorithm that uses a Crank-Nicolson proposal move.   This non-preconditioned MALA algorithm has nice properties and good performance; see in particular the middle panel of Figure 1, the top panel of Figure 4, and Proposition 6.1 of Ref.~\cite{BeRoStVo2008}.  Like non-preconditioned MALA, the performance of this Cayley-based non-preconditioned HMC algorithm is not strongly mesh-dependent, in the sense that to obtain an acceptance probability that converges to a nontrivial limit for a global move in state space as the spatial step size $\Delta s$ tends to zero, it is sufficient to select the time step size as $\Delta t \lesssim \Delta s^{1/4}$.

The stability condition for the Cayley splitting applied to a one dimensional, highly oscillatory Hamiltonian problem requires that $\Delta t<2$.  This condition does not change for the Cayley splitting applied to an analogous infinite dimensional linear Hamiltonian PDE with a spatial Gaussian white noise initial condition.  These results reflect the fact that the Cayley approximation offers a strongly stable symplectic map for the highly oscillatory part of the Hamiltonian dynamics.  Of course, in general, this stability requirement depends on the Lipschitz constant of the vector field appearing in \eqref{eq:semidiscrete_intro_B}, but the key point is that the stability requirement is independent of the fast frequencies present in \eqref{eq:semidiscrete_intro_A}.  This criterion for stability needs to be strengthened in the Langevin case, since we require asymptotic stability for the numerical approximation.  Nevertheless, this stability requirement does not involve fast frequencies present in \eqref{eq:semidiscrete_intro_A}.

For a linear Hamiltonian PDE, the Cayley splitting in \eqref{eq:cayley_splitting_hamiltonian} is strongly accurate in representing the position and momentum components, and in representing the energy in a weak sense.  To obtain accurate methods, the time step size has to be adjusted according to the spatial step size as described in these results.  We provided similar results for the Cayley splitting  \eqref{eq:cayley_splitting} applied to a linear Langevin SPDE.   In general, the conditions for accuracy are much more demanding than the conditions for stability.  We also used the Cayley splitting in  \eqref{eq:cayley_splitting_hamiltonian} to integrate non-preconditioned Hamiltonian dynamics in an HMC algorithm.  This Cayley-based non-preconditioned HMC performs comparably to preconditioned HMC using a Verlet integrator, in the sense that the mean energy errors have the same rates of convergence \cite{BoSaActaN2018}.

As an application, we considered Langevin SPDEs whose marginal invariant measure in position is the distribution of a diffusion bridge.  In this context, we showed how to set the invariant distribution of the Cayley splitting \eqref{eq:cayley_splitting} by combining \eqref{eq:cayley_splitting_hamiltonian} with Metropolis-Hastings Monte-Carlo.  We also showed that a step of non-preconditioned MALA using as proposal move a Crank-Nicolson discretization of non-preconditioned first-order Langevin dynamics is equivalent in law to a step of HMC with one integration step of a Cayley splitting for non-preconditioned Hamiltonian dynamics.   This is an infinite-dimensional analog of a well-known link between finite-dimensional MALA and HMC.  

A general takeaway from this paper is that exact splitting methods are prone to linear resonance instabilities in infinite dimensional Hamiltonian systems, and therefore, should be substituted with the Cayley splitting.


\section*{Acknowledgements}
We wish to thank Jes\'{u}s Sanz-Serna, Andreas Eberle, Raphael Zimmer, Katherine Newhall \& Gideon Simpson for helpful discussions.

\bibliographystyle{amsplain}
\bibliography{nawaf}

\end{document}